\documentclass[a4paper,makeidx]{amsproc}
\makeindex
\makeglossary
\usepackage{multicol}

\usepackage[centertags]{amsmath} %centertags is default, therefore obsolete
\usepackage{amssymb,amsfonts,amsthm}
\usepackage{graphicx} % yes we include graphics 
\usepackage{xcolor}

\usepackage{xy}       % yes we have xy matrix diagrams
\xyoption{all}

\newcommand{\arxiv}{1}

\let\modus 1%arxiv v3 version after submitting proofsheets
\newcommand{\ifarxiv}[2]{
\if \modus\arxiv
      #1
\else
      #2
\fi
}

\numberwithin{equation}{section}

\theoremstyle{plain}

\newtheorem{theorem}{Theorem}[section]
\newtheorem{prop}[theorem]{Proposition}
\newtheorem{lemma}[theorem]{Lemma}
\newtheorem{cor}[theorem]{Corollary}

\theoremstyle{definition}
\newtheorem{dfn}[theorem]{Definition}

\theoremstyle{remark} % i.e. same as plain, but theorembodyfont is rm

\newtheorem{example}[theorem]{Example}
\newtheorem{remark}[theorem]{Remark}

\theoremstyle{plain}
\providecommand{\mathscr}{\mathcal} % a priori mathscr is mathcal
\IfFileExists{mathrsfs.sty}{
\usepackage{mathrsfs}}         %  \mathscr produces rsfs fonts
   {%  else try euler fonts
     \IfFileExists{eucal.sty}{
        \usepackage[mathscr]{eucal}} % \mathscr produces Euscript fonts,
   {% else both unavailable, do nothing, mathscr remains mathcal
   }
}
\renewcommand{\labelenumi}{\textup{\templabelenumi}}

\newcommand{\enumref}[1]{\textup{(\ref{#1})}}
\newcommand{\enum}[1]{\textup{(#1)}}

\newcommand{\plref}[1]{{\normalfont \ref{#1}}}

\newcommand{\bigsetdef}[2]{\bigl\{ #1 \,\bigm|\, #2\bigr\}}

\newcommand{\scalar}[2]{\langle #1,#2\rangle}

\newcommand{\cinfz}[1]{C_0^\infty(#1)}% compact support
\newcommand{\cinf}[1]{C^\infty(#1)}
\newcommand{\ginfz}[1]{\Gamma_0^\infty(#1)}% compact support
\newcommand{\ginf}[1]{\Gamma^\infty(#1)}

\newcommand{\ovl}[1]{\overline{#1}}
\newenvironment{thmenum}{
\renewcommand{\labelenumi}{\textnormal{(\arabic{enumi})}}

\begin{enumerate}}{\end{enumerate}}

\newlength{\boxwidth}

\newcommand\ga{\alpha} 
\newcommand\gb{\beta}  
\newcommand\gG{\Gamma}
  
\newcommand\gk{\kappa}
\newcommand\pl{\partial}%<---abweichende Terminologie
\newcommand\gve{\varepsilon}

\newcommand\gl{\lambda} 
\newcommand\gL{\Lambda}
\newcommand\go{\omega}

\newcommand\gs{\sigma} 
\newcommand\gS{\Sigma}

\newcommand\eps{\varepsilon}

\newcommand{\N}{\mathbb{N}}
\newcommand{\C}{\mathbb{C}}
\newcommand{\R}{\mathbb{R}}

\newcommand{\Z}{\mathbb{Z}}

\newcommand\cB{\mathscr{B}}

\newcommand\cD{\mathscr{D}}
\newcommand\cE{\mathscr{E}}
\newcommand\cF{\mathscr{F}}

\newcommand\cH{\mathscr{H}}

\newcommand\cL{\mathscr{L}}

\newcommand\cP{\mathscr{P}}

\newcommand\cU{\mathscr{U}}
\newcommand\cV{\mathscr{V}}

\newcommand\coker{\operatorname{coker}}

\newcommand\Id{\operatorname{Id}}
\newcommand\id{\operatorname{id}}
\newcommand\im{\operatorname{im}}

\newcommand\ind{\operatorname{ind}}

\newcommand\loc{\operatorname{loc}}

\renewcommand\Re{\operatorname{Re}}

\newcommand\sign{\operatorname{sign}}
\newcommand\spec{\operatorname{spec}}

\newcommand\supp{\operatorname{supp}}

\newenvironment{thmenumm}{
\renewcommand{\labelenumi}{\textnormal{(\alph{enumi})}}

\begin{enumerate}}{\end{enumerate}}

\theoremstyle{definition}
\newtheorem{convent}[theorem]{Convention}

\newcommand{\CL}{\operatorname{CL}}
\newcommand{\CS}{\operatorname{CS}}
\newcommand{\comp}{\operatorname{comp}}
\newcommand{\Diff}{\operatorname{Diff}}
\newcommand{\dstr}{d_{\textrm{str}}}
\newcommand{\dvol}{\operatorname{dvol}}

\newcommand{\Ell}{\operatorname{Ell}}
\newcommand{\End}{\operatorname{End}}
\newcommand\flg[1]{\cF\cL_{#1}}
\newcommand{\half}{{1/2}}
\newcommand{\Hom}{\operatorname{Hom}}

\newcommand{\ort}{\operatorname{ort}}
\newcommand{\Red}{\operatorname{Red}}
\newcommand{\rrr}{\!\!\upharpoonright\!} %% Restriction
\newcommand{\sa}{\textup{sa}}
\newcommand\UCP{\textup{UCP}}

\newcommand\fle[1]{\stackrel{#1}{\le}}

\newcommand\dpa{\partial}

\newcommand{\dd}[1]{\frac{d}{d#1}}

\newcommand\e{\varepsilon}
\newcommand\lla{\langle}
\newcommand\noi{\noindent}
\newcommand{\norm}[1]{\lVert#1\rVert}
\newcommand\ran{\im}
\newcommand\rra{\rangle}
\newcommand\tand{\mbox{\ \rm  and }}

\newcommand\too{\longrightarrow}

\newcommand\ii{^{-1}}

\newcommand\<{\subset}
\newcommand{\textqm}[1]{``#1'' }
\begin{document}

\title[The Calder{\'o}n Projection]
{The Calder{\'o}n Projection: \\ New Definition and Applications}

\author{Bernhelm Boo{\ss}-Bavnbek}
\address{Department of Science, Systems, and Models\\ Roskilde
University, DK-4000 Ros\-kilde, Denmark} \email{booss@ruc.dk}
\urladdr{http://milne.ruc.dk/$\sim$booss}

\author{Matthias Lesch}
\address{Mathematisches Institut,
Universit\"at Bonn, Endenicher Allee 60, D-53115 Bonn, Germany}
\email{ml@matthiaslesch.de, lesch@math.uni-bonn.de}
\urladdr{http://www.matthiaslesch.de, http://www.math.uni-bonn.de/$\sim$lesch}

% arxiv version
%\author{{\tiny Chaofeng Zhu}}
%\address{{\tiny Chern Institute of Mathematics\\
%Nankai University\\
%Tianjin 300071, P. R. China}} \email{{\tiny zhucf@nankai.edu.cn}}

% official version
\author{Chaofeng Zhu}
\address{Chern Institute of Mathematics\\
Nankai University\\
Tianjin 300071, P. R. China} \email{zhucf@nankai.edu.cn}

\thanks{This work was supported by the network 
\textqm{Mathematical Physics and
Partial Differential Equations} of the Danish Agency for Science,
Technology and Innovation. 
The second named author was partially supported by Sonderforschungsbereich/Transregio
\textqm{Symmetries and Universality in Mesoscopic Systems}
(Bochum--Duisburg/Essen--K\"oln--Warszawa)
and the Hausdorff Center for Mathematics (Bonn). 
The third named author 
%did just enough such that it would be
%inappropriate to remove him from the list of authors, at least by the ethical
%standards of the remaining authors. He refused to contribute to any of the tasks 
%which need to be done to finalize a paper. The smallest available latex font, 5pt, still overstates
%his contribution by an order of magnitude; nevertheless he 
 was partially supported
by FANEDD 200215, 973 Program of MOST, Fok Ying Tung Edu. Funds
91002, LPMC of MOE of China, and Nankai University.}

% Version management, to be removed later
%\thanks{\fbox{Version: \MLversion\quad date-processed: \today}}

%    General info


%\date{}

\begin{abstract}
We consider an arbitrary
linear elliptic first--order differential operator $A$ with smooth coefficients
acting between sections of complex vector bundles $E,F$ over a compact
smooth manifold $M$ with smooth boundary $\Sigma$.
We describe the analytic and topological properties of $A$ in a collar neighborhood $U$ of
$\Sigma$ and analyze various ways of writing $A\rrr{U}$ in product form.
We discuss the sectorial projections of the corresponding tangential operator,
construct various invertible doubles of $A$ by suitable local boundary conditions,
obtain Poisson type operators with different mapping properties, and provide
a canonical construction of the Calder{\'o}n projection.
We apply our construction to generalize the Cobordism Theorem
and to determine sufficient conditions for continuous variation of
the Calder{\'o}n projection and of \emph{well--posed} selfadjoint
Fredholm extensions under continuous variation of the data.
\end{abstract}

\maketitle

% Draft issues, Remove in final version
\tableofcontents
\listoffigures
%\markboth{Version: \MLversion, \today}{Version: \MLversion, \today}

% end of draft issues

\section{Introduction}\label{s:intro}

This paper is about basic analytical properties of elliptic operators on compact manifolds with smooth boundary.
Our main achievements are (i) to develop the basic elliptic analysis in full generality, and
not only for the generic case of operators of Dirac type in product metrics
(i.e., we assume neither constant coefficients in normal direction nor symmetry of
the tangential operator);
(ii) to establish the cobordism invariance of the index in greatest generality; and
(iii) to prove the continuity of the Calder{\'o}n projection and of related
families of global elliptic boundary value problems under
parameter variation.
We take our point of departure in the following observations.

\subsection{Dirac operator {\em folklore}}
Most analysis of geometrical and physical problems involving  a Dirac operator $A$ on a compact manifold $M$ with smooth boundary $\gS$ acting on sections of a (complex) bundle $E$
seems to rely on quite a few basic facts which are part of the shared {\it folklore} of people working in this field of global analysis
(e.g., see Boo{\ss}--Bavnbek and Wojciechowski \cite{BooWoj:EBP} for properties (WiUCP), (InvDoub) and
(Cob), and Nicolaescu \cite[Appendix]{Nic:GSG} and Boo{\ss}--Bavnbek, Lesch and Phillips \cite{BooLesPhi:UFOSF} for property (Param)):

\begin{description}

\item[WiUCP] the weak inner unique continuation property (also called {\it weak \UCP\ to the
boundary}), i.e., there are no nontrivial elements in the null space $\ker A$ vanishing at the boundary $\gS$ of $M$;

\item[InvDoub] the existence of a suitable elliptic invertible continuation $\widetilde A$ of $A$, acting on sections of a vector bundle over the closed double or another suitable closed manifold $\widetilde M$ which contains $M$ as submanifold; this yields a Poisson type operator $K_+$ which maps sections over the boundary into sections over $M$; and
 a precise Calder{\'o}n projection $C_+$, i.e., an idempotent mapping of sections over the boundary onto the Cauchy data space which consists of the traces at the boundary of elements in the nullspace of $A$ (possibly in a scale of Sobolev spaces);

\item[Cob] the existence of a selfadjoint regular Fredholm extension of any total 
(formally selfadjoint) Dirac operator $A$ in the underlying $L^2$-space 
with domain given by a pseudodifferential boundary condition; 
that implies the vanishing of the signature of the associated quadratic form, 
induced by the leading symbol in normal direction at the boundary; 
moreover, that actually is equivalent to the Cobordism Theorem asserting a 
canonical splitting of the induced tangential operator $B = B^+\oplus B^-$ over $\Sigma$ with 
$\ind B^+=0$;

\item[Param] the continuous dependence of a family of operators, their associated Calder{\'o}n projections, and of any family of well-posed (elliptic) boundary value problems on continuous or smooth variation of the coefficients.

\end{description}

\subsection{In search of generalization}
With the renewed interest in geometrically defined elliptic operators of first order of general type,
arising, e.g., from perturbations of Dirac operators,
we ask to what extent the preceding list can be generalized to arbitrary linear 
elliptic differential operators with smooth coefficients. 
It is hoped that the results of this paper can serve as guidelines for similar constructions
and results for hypo- and sub-elliptic operators where the symbolic calculus is not fully available.

There are immediate limits for generalization of some of the
mentioned features by counter examples: UCP, even weak inner UCP may
not hold for arbitrary elliptic systems of first order, see
indications in that direction in Pli{\'s} \cite[Corollary 1, p.
610]{Pli:SLE} and the first-order Alinhac type counterexample to
strong UCP \cite[Example, p. 184]{Bae:ZSS}.  Moreover, from just
looking at the deficiency indices, we see that the formally
selfadjoint operator $i\dd x$ on the positive line does not admit a
selfadjoint extension. This example is instructive because, quite
opposite to the half-infinite domain, on a bounded one-dimensional
interval {\it any} system of first-order differential equations
satisfies property (Cob) by a deformation argument.

\smallskip
We go through the list. 

\subsubsection*{Property ({\bf WiUCP})} 
It seems that the precise domain of validity is unknown. The local stability of weak
inner UCP has been obtained by Boo{\ss}--Bavnbek and Zhu in \cite[Lemma
3.2]{BooZhu:FMD}. In spite of the {\it local} definition of UCP, the
property (WiUCP) has a threefold {\em global} geometric meaning: (i)
there are no ghost solutions, i.e., each section $u\in \ginf{M;E}$
belonging to the null space $\ker A$ over the manifold $M$ has a
non-trivial trace $u\rrr \gS$ at the boundary; (ii) equivalently,
the maximal extension $A_{\max}=(A^t_{\min})^*$ is surjective in
$L^2(M,E)$ for densely defined closed minimal
$A_{\min}:\cD(A_{\min})=H^1_0(M,E)\to L^2(M,E)$; and (iii), as noted
by Boo{\ss}--Bavnbek and Furutani in \cite[Section 3.3]{BooFur:MIF} and in
various follow-ups, it seems that assuming weak inner UCP of $A$ and
$A^t$ is mandatory for obtaining the continuity of Cauchy data
spaces and the continuous change of the Calder{\'o}n projection under
variation of the coefficients, i.e., property (Param).

\subsubsection*{Property ({\bf InvDoub})} Different approaches are available:
one approach (\cite[Chapter 9]{BooWoj:EBP}) has been
the gluing of $A$ and its formal adjoint $A^t$ to an invertible
elliptic operator $\widetilde A$ over the closed double $\widetilde M$.
This construction is explicit, if the metric structures underlying
the Dirac operator's definition are product near the boundary.
In the selfadjoint case, it yields at once the Lagrangian property
of the Cauchy data space in the symplectic Hilbert space $L^2(\gS,E_\gS)$
of square integrable sections in $E_\gS:=E\rrr \gS$. Then (Cob) follows.

Property (InvDoub) generalizes to the non-product case for operators
of Dirac type and, as a matter of fact, for any elliptic operator
satisfying weak inner UCP under the somewhat restrictive condition
that the tangential operator has selfadjoint leading symbol. Here
the trick is that this condition permits the prolongation of the
given operator to a slightly larger manifold $M'$ with boundary
reaching constant coefficients in normal direction close to the new
boundary {\em and} maintaining UCP under the prolongation (as well
as formal selfadjointness of the coefficients, if present at the
old boundary). This is explained in the Appendix.

But what can be done for general elliptic operators?
A very general and elegant construction of the Calder{\'o}n projection
was given by H{\"o}rmander in \cite[Theorem 20.1.3]{Hor:ALPDOIII} on
the symbol level. Unfortunately, he obtains only an {\it almost}
projection (up to smoothing operators) which limits its applicability in our context.

In this paper, we shall exploit another general definition of the
Calder{\'o}n projection which is due to
Seeley \cite[Theorem 1 and Appendix, Lemma]{See:TPO}. Seeley's construction
provides a precise projection, not only an approximate one, and does not
require UCP. First he replaces $A$ by an invertible operator $A_1$
by adding the projection onto the finite--dimensional space of inner
solutions. Then he extends the operator $D:=\begin{pmatrix} 0&
A_1^t\\ A_1& 0\end{pmatrix}$ to the closed double $\widetilde M'$ of a
slightly extended manifold $M'$ with boundary. In general, such a
prolongation may, however, destroy weak inner UCP even when UCP was
established on the original manifold. Seeley constructs on the
symbol level (and by adding a suitable correction term) an elliptic
extension $\widetilde D$ of $D$ over the whole of the closed manifold $\widetilde
M'$ which is always invertible. He shows that $\widetilde D$ provides the
wanted Poisson operator and a truly pseudodifferential Calder{\'o}n
projection $\cP_+$ along the original boundary $\gS$\,. If the tangential
operator is formally selfadjoint then $\cP_+$ has selfadjoint leading symbol
and can be replaced by the orthogonal projection which is also
pseudodifferential and has the same symbol (and may be denoted by
the same letter $\cP_+$). In this way, the choices in
the \emph{construction of the invertible double} are removed totally, as the 
operation of the resulting $\cP_+$ is concerned. This
makes $\cP_+$ a good candidate for property (Cob).

However, Seeley's general construction, similarly the recent Grubb \cite[Section 11.1]{Gru:DO},
in difference to the simple gluing
in the case of Dirac type operators of \cite[Chapter 9]{BooWoj:EBP},
does not immediately lead to the Lagrangian property of $\im \cP_+$\,.
Moreover and more seriously, when working with curves of elliptic problems
Seeley's construction does not give a hint under what conditions the
Calder{\'o}n projections vary continuously when varying a parameter.
There are too many choices involved in Seeley's construction.

\subsection{Our present approach}
This motivates our present approach (inspired by Himpel, Kirk and Lesch \cite{HimKirLes:CPH}),
namely the construction of the invertible double as a {\em canonically} given local boundary problem
for the double $D$ exactly on the original manifold $M$, without any choices, prolongations etc.
This leaves us with full control of the UCP situation; leads directly to the wanted Fredholm Lagrangian property (Cob);
and, moreover and here most decisively, provides explicit formulas for treating the parameter dependence in property (Param).

This program is opened in Section \ref{s:sectwo} by explaining our basic choice of product
 structures near the boundary
for the sake of comprehensible analysis, even if the original geometric structures
are non-product; moreover, for the convenience of the reader and for fixing our notation we
summarize a few basic facts about regular boundary conditions.

To begin with, we do {\em not} assume selfadjoint leading symbol
of the tangential operator $B_0$ {\em nor} constant coefficients
$B_x=B_0$ along an inward coordinate $x$. Most of our estimates
depend on the single fact that $B_0-\gl$ is parameter
dependent elliptic for $\gl$ in a conic neighborhood of $i\R$ in the
sense of Shubin \cite[Section II.9]{Shu:POST}. More precisely, we
depend on the related concept of sectorial spectral projections
introduced in 1970 by Burak \cite{Bur:SPE} and recently further
developed in Ponge \cite[Section 3]{Pon:SAZ} as part of the current
upsurge of interest in spectral properties of non-selfadjoint
elliptic operators. Because of our interest in the continuous
dependence of this kind of generalized positive spectral projections
on the input data we found it necessary to develop the concept of
sectorial projections once again from scratch. This is done in
Section \ref{s:sectorial} where we develop an abstract Hilbert space
framework for the concept of sectorial projections and apply it to
tangential operators perceived as parametric elliptic operators.

In Section \ref{s:invdoub} we provide the construction and the relevant properties of the
invertible double yielded by
a local elliptic boundary value problem, induced by fixing an invertible
bundle homomorphism $T$ over $\Sigma$.
%% \marginpar{Check ref, Sections changed}

In Section \ref{s:calderon} we establish suitable Sobolev regularity of the inverse operator
which leads to
the definition and basic properties of Poisson operator and Calder{\'o}n projection,
both definitions made dependent of the choice of the above mentioned homomorphism
$T$. We shall show that the range of the Calder{\'o}n projection does not depend on the choice of $T$ and is, 
in fact, equal to the Cauchy data space. That yields the relation between the {\em canonical} 
Calder{\'o}n projection defined as orthogonal projection onto the Cauchy data space, and our {\em relative} 
Calder{\'o}n projections, which depend on $T$. However, it is not the canonical, but only the relative definition 
that establishes the Lagrangian property of the Cauchy data space and its continuous
dependence of the coefficients for general elliptic differential operators of first order.

That is the subject of the two closing sections of this paper, which present the fruits of the analysis endeavour of Sections 1-5. 

\subsubsection*{Property ({\bf Cob})} In Section \ref{s:cobord}, 
we give a first application of our construction of the Calder{\'o}n projection: 
we give our reading of Ralston \cite{Ral:DISO} and infer that
the arguments of this 1970 paper establish the following findings for any formally selfadjoint differential operator $A$ over a compact manifold $M$ with smooth boundary $\gS$:
\begin{itemize}

\item the existence of a selfadjoint Fredholm extension $A_P$ given by a pseudodifferential boundary condition $P$;

\item the vanishing of the signature of $i\go$ on the space $V(B_0)$ of eigensections 
to purely imaginary eigenvalues of the tangential operator $B_0$ of $A$ over the boundary (or on $\ker B_0$
in the case that $B_0$ is formally selfadjoint);
here $i\go$  denotes the form induced by the Green form of $A$ on the
symplectic von Neumann space $\gb(A):=\cD(A_{\max})/\cD(A_{\min})$, i.e., the leading symbol of $A$
over the boundary $\gS$ in normal direction;

\item and, equivalently, but seemingly never recognized by people working in global analysis, the General
Cobordism Theorem, stating that the index of {\it any} elliptic
differential operator $B^+$ over a closed manifold $\gS$ must
vanish, if $B^+$ can be written as the left lower corner of a
formally selfadjoint tangential operator over $\gS$ induced by an
elliptic formally selfadjoint $A$ on a smooth compact manifold $M$
with $\partial M=\gS$.
\end{itemize}

\subsubsection*{Property ({\bf Param})} 
In Section \ref{s:parameter}, as a second application of our
construction of the Calder{\'o}n projection, we establish that property
in great generality. Roughly speaking, we let an
operator family $(A_z)$ and the family $(A^t_z)$ of formally
adjoint operators vary continuously in the operator norm from
$L^2_1(M)$ to $L^2(M)$ and assume that the leading symbol
$(J_0(A_z))$ of $(A_z)$ over the boundary $\gS$ in normal
direction also varies continuously in the $L^2_{1/2}(\gS)$ operator
norm with $z$ running in a parameter space $Z$. We assume for all
$A_z$ and $A_z^t$ property (WiUCP) or, almost equivalently, that the
dimensions of the spaces of \textqm{ghost solutions} without trace at the
boundary remain constant under the variation. Then in Theorems
\ref{thm4.10} and \ref{thm4.14}a we show that the inverse of the
\textqm{invertible double} and, under slightly sharpened continuity, the
Poisson operator in respective operator norms vary continuously; and
so does the resolvent of a family $({A_z}_{P_z})$ of well-posed
Fredholm extensions of now formally selfadjoint $(A_z)$ with
orthogonal pseudodifferential projections $(P_z)$ varying in
$L^2_{1/2}(\gS)$ operator norm.

Unfortunately, we can neither prove nor disprove the 
continuous variation
of the Calder{\'o}n projection in the same generality. 
However, if the leading symbol of the tangential operator is selfadjoint,
we can prove the continuous variation of the sectorial projection (Proposition \ref{p:dependence-Pplus-self-adjoint})
and so (Corollary \ref{thm4.14-sa}) of the Calder{\'o}n projections by our correction formula \eqref{eq4.24} and
Theorem \ref{thm4.14}b. Our Proposition \ref{p:lower-order} shows that the difficulties for proving
continuous variation of the sectorial projection disappear also for non-selfadjoint leading symbol,
if the variation is of order $<1$.

In Appendix A we discuss various special cases with emphasis on constant coefficients in normal direction in a collar around the boundary.

The main results of this paper have been announced in \cite{BooLes:IDE}.

\section{Elliptic differential operators of first order on manifolds with boundary}
\label{s:sectwo}

\subsection{Product form and metric structures near the boundary}\label{ss:product}

\index{operator!{D}irac}
\index{domain|see{operator}}

We shall begin with a basic observation: {\it Dirac operators} emerge from a Riemannian structure on the manifold and a Hermitian metric on the vector bundle
(together with Clifford multiplication and a connection). Talking about a {\it general differential operator} it is in our view very misleading to  pretend
that the operator will depend on metrics and such. All we need is the operator and an $L^2$--structure on the sections. The latter basically only requires a
density (take $\dvol$ in the Riemannian case) on the manifold and a metric on the bundle.
In this paper, we prefer to {\it choose} metrics and such as simple
as possible and push all complications into the operator.

The message is this: we can always work in the product case and do
have to worry only about the operator. In detail:

\glossary{$M$}
\glossary{$L^2(M,E;g,h)$}
\glossary{$E,F$}
Let $M$ be a compact manifold with boundary and
$\pi:E\to M$ a vector bundle. Given a Hermitian metric $h$ on $E$
and a Riemannian metric $g$ on $M$ we can form the Hilbert space
$L^2(M,E;g,h)$ which is the completion of $\ginfz{M\setminus\pl
M;E}$ with regard to  the scalar product
\begin{equation}\label{eq2.1}
    \scalar{u}{v}_{g,h}:=\int_M h(u(x),v(x))\dvol_g(x).
\end{equation}
The space of smooth sections of the vector bundle $E$ over $M$ is denoted
by $\ginf{M;E}$; the corresponding space of smooth compactly supported sections
is denoted by $\ginfz{M;E}$.
\glossary{$\Gamma_{(0)}^\infty(M;E)$}
Given another Riemannian metric $g_1$ and another Hermitian metric
$h_1$ on $E$ there is a smooth positive function $\varrho
\in\cinf{M}$ such that
\begin{equation}\label{eq2.2}
   \dvol_{g_1}=\varrho \dvol_g,
\end{equation}
and there is a unique smooth section $\theta\in\ginf{M;\End E}$ such that for $x\in M, \xi,\eta\in E_x$ we have
\begin{equation}
    h_{1,x}(\xi,\eta)=h_x(\theta(x)\xi,\eta),\qquad h_x(\theta(x)\xi,\eta)=h_x(\xi,\theta(x)\eta).
    \label{eq2.3}
\end{equation}

With regard to $h$ the operator $\theta(x)$ is selfadjoint and positive definite, thus we may
form $\sqrt{\theta}$ which is again a smooth selfadjoint and positive definite section of $(\End E, h)$.
It is clear that \eqref{eq2.3} determines $\theta(x)$ uniquely and the
claimed smoothness of $x\mapsto \theta(x)$ can be checked easily in local coordinates. In sum, we find for $u,v\in\ginf{M;E}$
\begin{equation}\label{eq2.4}\begin{split}
    \scalar{u}{v}_{g_1,h_1}&= \int_Mh_1(u(x),v(x))\dvol_{g_1}(x)\\
   & = \int_M h(\theta(x)u(x),v(x))\varrho(x)\dvol_{g}(x)\\
   & = \scalar{\sqrt{\varrho\theta}u}{\sqrt{\varrho\theta}v}_{g,h}.
                \end{split}
\end{equation}
Thus we arrive at

\begin{lemma}\label{l2.1} The map
\[ \Psi:\ginf{M;E}\longrightarrow\ginf{M;E},\quad u\longmapsto \sqrt{\varrho\theta}\ii u
\]
extends to an isometry from $L^2(M,E;g,h)$ onto
$L^2(M,E;g_1,h_1)$.
\end{lemma}

Now assume that we are given a differential operator $A$ in $L^2(M,E;g,h)$ of first order.
It may be a Dirac operator which is constructed from the metrics
$g$ and $h$. $g$ and $h$ may be wildly non--product near the boundary.
Suppose there are metrics $g_1$, $h_1$ which we like more, e.g., product near the
boundary. Then consider the differential operator $\Psi A\Psi^ {-1}$ in $L^2(M,E;g_1,h_1)$.
$\Psi A \Psi^{-1}$ is still a differential operator and since $\Psi$ is
unitary all spectral properties are preserved.

Let us be even more specific and choose a neighborhood $U$ of
$\partial M=:\Sigma$ and a diffeomorphism $\phi:U\to
[0,\eps)\times\Sigma$ with $\phi\rrr\Sigma=\id_\Sigma$. Furthermore,
we choose a metric $g_1$ on $M$ such that
\begin{equation}\label{eq2.5}\begin{split}
    \phi_*g_1&=dx^2\oplus g_\Sigma,\\
     g_\Sigma&:= g\rrr\Sigma,
   \end{split}
\end{equation}
is a product metric which induces the same metric on the boundary as $g$. Here $x$ denotes the
normal inward coordinate near the boundary in the metric $g_1$.

$\phi$ is covered by a bundle isomorphism $\cF:E\rrr{U}\to [0,\eps)\times
E_\Sigma$, $E_\Sigma:=E\rrr\Sigma$, i.e., we have the commutative
diagram
\begin{equation}\label{eq2.6}\xymatrix{
          E\rrr{U} \ar[r]^-{\cF}\ar[d]^\pi & [0,\eps)\times E_\Sigma\ar[d]^{\id\times\pi} \\
          U \ar[r]^-\phi &[0,\eps)\times \Sigma.
}
\end{equation}
Likewise, we may now choose a metric $h_1$ on $E$ such that
$h_1(x):=\cF_*h_1\rrr\{x\}\times\Sigma
=h_1\rrr\Sigma=h\rrr\Sigma=:h_{\Sigma}$ is
independent of $x\in [0,\eps)$. The mappings $\cF$ and $\phi$ induce a map
\begin{equation}\label{eq2.7}\begin{matrix}
       \Psi_1:&\ginf{U;E}&\too &\cinf{[0,\eps),\ginf{E_\Sigma}}\\
              \ &f&\longmapsto&\Bigl( x\mapsto \bigl (p\mapsto
              \cF(f(\phi^{-1}(x,p)))\bigr)\Bigr),
                \end{matrix}
\end{equation}
which extends to a unitary isomorphism $L^2(U,E;g_1,h_1)\to
L^2([0,\eps],L^2(\Sigma,E_\Sigma;g\rrr\Sigma,h\rrr\Sigma))$. On $\Sigma$ and $E_\Sigma$
we have the fixed metrics $g_\Sigma$ respectively $h_{\Sigma}$ and we will
suppress the reference to them in the notation.

Together with the unitary isomorphism $\Psi$ of Lemma \plref{l2.1} we obtain the claimed isometry
\begin{equation}\label{eq2.8}
    \Phi:=\Psi_1\circ\Psi:L^2(U,E;g,h)\too L^2([0,\eps],L^2(E_\Sigma)).
\end{equation}
Now $\Phi A \Phi^ {-1}$ is a first order differential operator in the product Hilbert space
$L^2([0,\eps))\otimes L^2(\Sigma,E_\Sigma;g_\Sigma,h_\Sigma)$ and hence it takes
the form
\begin{equation}\label{eq2.9}
   D:=\Phi A \Phi^ {-1}=:J\Bigl(\frac{d}{dx}+B\Bigr)
\end{equation}
with a bundle endomorphism $J\in\cinf{[0,\eps),\ginf{\Sigma;\End E_\Sigma}}$ and a smooth family of first order differential operators
$B\in\cinf{[0,\eps),\Diff^1(\Sigma;E_\Sigma)}$; $\Diff^1(\Sigma;E_\Sigma)$ denoting the space of first order
differential operators acting on sections of $E_\Sigma$. For the moment we consider here only the smooth case, 
but so far one can replace "smooth" by "continuous" or
"Lipschitz" or whatever.

\glossary{$\Diff(M;E,F)$}

Let us repeat: now all metric structures are product near the
boundary and we do not have to worry about them. If we start, e.g., with
a Dirac operator $A$ on a Riemannian manifold with non--product metric
the 'non--product situation' is reflected in the varying
coefficients of $D$. From now on we have to worry only about those
varying coefficients and nothing else. 

After these somewhat pedagogical remarks we are ready to formulate the general set--up
of the paper.

%%
% ML writing new section on sectorial operators
%\input{sectorial-ml.tex}

\subsection{The general set--up}\label{ss:general-set-up}

We are going to fix some notation which will be used throughout the paper.
Assume that the following data are given:
\begin{itemize}
\item a compact smooth Riemannian manifold $(M,g)$ with smooth boundary $\Sigma := \partial M$,
\item Hermitian vector bundles $(E,h^E), (F,h^F)$,
\item a first order elliptic differential operator
\begin{equation}\label{eq3.1}
  A: \ginf{M;E} \longrightarrow \ginf{M;F}.
\end{equation}
\item $A^t:\ginf{M;F}\longrightarrow \ginf{M;E}$ denotes the formal adjoint of $A$
with respect to the metrics $g,h^E,h^F$.
\end{itemize}

\index{adjoint!formal}
\index{basic data}
\index{data|see{basic data}}
\glossary{$A$}
\glossary{$A^t$}
\glossary{$\Sigma$}
\glossary{$M$}
\glossary{$E,F$}

\index{Green's formula}

For further reference we record Green's formula for $A$. 

\begin{lemma}\label{l1.2}\label{l:Green-formula}
Let $\nu\in\ginf{\Sigma;TM\rrr\gS}$ be the outward normal vector
  field. Then we have with the notation of \eqref{eq2.1} and \eqref{eq2.5} 
for $u\in\ginf{M;E}, v\in\ginf{M;F}$
\begin{equation}\begin{split}\label{eq2.10}
    \scalar{&Au}{v}_{g,h}-
    \scalar{u}{A^tv}_{g,h}\\
        &=\frac 1i \int_{\Sigma}h_\gS\bigl(\sigma_A^1(\nu^b)u\rrr\gS, v\rrr\gS\bigr)dvol_{g_\gS}
         =-\scalar{J(0)u\rrr\Sigma}{v\rrr\Sigma}_{g_\Sigma,h_\gS},
        \end{split}
\end{equation}
where $i:=\sqrt{-1}$, $\nu^b$ denotes the cotangent vector field
corresponding to $\nu$ in the metric $g$, and $\sigma_A^1$ denotes
the leading symbol of $A$.
\end{lemma}

\index{symbol!leading}
\glossary{$\sigma_A^1$}

Note that $\phi_*\nu=-\frac{d}{dx}$. Recall also from the previous Subsection \plref{ss:product} that by
construction all transformations are trivial on the boundary, that
is,
\begin{equation}
    \begin{split}
        \phi\rrr\Sigma&=\id, \quad \cF\rrr{E}_\Sigma=\id,\\
        (\Psi_1f)\rrr\Sigma&=f\rrr\Sigma, \quad (\Psi f)\rrr\Sigma=f\rrr\Sigma,\\
        (\Phi f)\rrr\Sigma &=f\rrr\Sigma.
    \end{split}\label{eq2.11}
\end{equation}

\index{operator!unbounded}
\index{Sobolev space}
\glossary{$L^2_s(M,E)$}

We consider $A$ as an unbounded operator between the Sobolev (and
Hilbert) spaces
\begin{equation}\label{eq3.2}
 L^2_s(M,E;g,h^E),\quad L^2_s(M,F;g,h^F),\quad  s\geq 0.
\footnote{For simplicity we content ourselves with Sobolev spaces of
nonnegative order. On a manifold with boundary Sobolev spaces of
negative order are a nuisance, although with some care they could
be dealt with here, cf. \cite{LioMag:NHBVPAI}.}
\end{equation}

\index{adjoint!functional analytic}
If $A$ acts as an unbounded operator between the Hilbert spaces $H_1,H_2$, we denote by $A^*$ its
functional analytic adjoint. For $0$th order operators and for elliptic operators
on \emph{closed} manifolds the distinction between formal adjoint and (true) adjoint does
not really matter; so in this case we use both notations interchangeably.

\glossary{$A_{\min}$}
\glossary{$A_{\max}$}
\glossary{$\cD$} 
\index{operator!domain}

The closure of $A\rrr \ginfz{M\setminus\Sigma;E}$ in
$L^2$ is denoted by $A_{\min}$ and we put
\begin{equation}\label{eq3.3}
\cD(A_{\max}) := \bigsetdef{f\in L^2}{Af\in L^2},
\end{equation}
the domain of an unbounded operator $T$ will always be denoted by $\cD(T)$.
As explained in Subsection \ref{ss:product} there exists a collar $U \approx [0,\varepsilon) \times \Sigma$ and linear isomorphisms
\begin{equation}\label{eq3.4}
\Phi^G: \ginf{U;G} \too C^\infty([0,\varepsilon), \ginf{G_\Sigma}), \quad G = E,F,
\end{equation}
which extend to isometries
\begin{equation}\label{eq3.5}
L^2(U,G;g,h^G) \too L^2([0,\varepsilon],
L^2(\Sigma,G_\Sigma;g_\Sigma,h^{G_\Sigma}), \quad G = E,F,
\end{equation}
where $g_\Sigma = g\rrr\Sigma, G_\Sigma:=G\rrr\Sigma$ and $h^{G_\Sigma} =
h^G\rrr G_\Sigma$.

Now we consider
\begin{equation}\label{eq3.6}
D := \Phi^F A(\Phi^E)^{-1}: C^\infty([0,\varepsilon),
\ginf{E_\Sigma}) \too C^\infty([0,\varepsilon),
\ginf{F_\Sigma}).
\end{equation}
Since $A$ is a first order elliptic differential operator we find
\begin{equation}\label{eq:operator-collar}
D = J_x\Bigl(\frac{d}{dx} + B_x\Bigr),
\end{equation}
where $J_x\in \operatorname{Hom}(E_\Sigma,F_\Sigma), 0\leq x\leq\varepsilon,$ 
is a smooth family of bundle homomorphisms and $(B_x)_{0\leq x\leq\varepsilon}$ is a smooth family
of first order elliptic differential operators between sections of
$E_\Sigma$. Note that in view of \eqref{eq2.10} $J_0$ equals
$i\sigma^1_A(\nu^b)$, where $\nu = -\frac{d}{dx}$ is the outward
normal vector field.

\glossary{$J_x,B_x$}

To avoid an inflation of parentheses we will most often use the notation
$B_x, J_x$ instead of $B(x), J(x)$ etc. Only to avoid double subscripts we
will write $B(x), J(x)$ in subscripts.

Since $\Phi^E, \Phi^F$ are unitary (cf. \eqref{eq3.4}, \eqref{eq3.5}) we have $D^t=\Phi^E A^t (\Phi^F)\ii$
and hence
\begin{align}
   %  A&=J_x\Bigl(\dd x+B_x\Bigr) \quad\text{and}\label{eq:operator-collar}\\
   -D^t&=J_x^t\dd x-B_x^tJ_x^t +(J_x')^t\label{eq:formal-adjoint}\\
       &=J_x^t\Bigl(\dd x - (J_x^t)\ii B_x^tJ_x^t\Bigr) +(J_x')^t.\nonumber
\end{align}

If $A$ is formally selfadjoint, we have the relations
\begin{equation}
    J^t=-J, \quad JB=J'-B^tJ
   \label{eq2.12}
\end{equation}
($'$ denotes differentiation by $x$).
 
Alternatively, we may choose the following
normal form in a collar of the boundary:
\begin{equation}\label{e:standard-symmetric}
%\begin{split}
D=J_x\Bigl(\dd x + B_x\Bigr) +\frac 12 J'_x,\\
% &=:J_x\Bigl(\dd x + B_x\Bigr) +\frac 12 J_x'.
%\end{split}
\end{equation}
In this normalization, $A=A^t$ implies the relations
\begin{equation}\label{eq2.12-modified}
    J^t=-J, \quad JB=-B^tJ.
\end{equation}
The normal form \eqref{e:standard-symmetric} determines $J$ and $B$ uniquely.

\index{elliptic|see{operator}}
\index{operator!elliptic}
\index{operator!parameter dependent elliptic}
Returning to general (not necessarily formally selfadjoint) $A$ we remark
that the ellipticity of $A$ and hence of $D$ imposes various restrictions.
The obvious ones are that $J_x$ is invertible and that $B_x$ is elliptic for
all $x$. What's more, ellipticity of $D$ means that for $\lambda\in\R, \xi\in
T_p^*\Sigma, (\lambda,\xi)\not= (0,0)$, the operator
\begin{equation}\label{eq3.8}
i\lambda + \sigma^1_{B(x)}(p,\xi)
\end{equation}
is invertible for all $(x,p)\in[0,\varepsilon)\times\gS$. Here,
$\sigma^1_{B(x)}$ denotes the leading symbol of $B_x$. In other
words, for $\xi\in T_p^*\Sigma\setminus\{0\}$ the endomorphism
$\sigma^1_{B(x)}(p,\xi)\in$ End$(E_p)$ has no eigenvalues on the
imaginary axis $i\R$.

%
%
%% \addcontentsline{toc}{subsection}{Regular boundary conditions}
\subsection{Regular boundary conditions}\label{ss:regular-boundary-conditions}
For the convenience of the reader and to fix some notation we briefly summarize a few basic facts about boundary conditions for $A$. Standard references are
\cite{BooWoj:EBP, Hor:ALPDOIII, LioMag:NHBVPAI, See:TPO}. We will adopt the point of view of the elementary functional analytic presentation
\cite{BruLes:BVP}. However, we try to be as self-contained as possible.

\index{trace map}
\glossary{$\varrho$}
It is well-known that the trace map
\begin{equation}\label{eq3.9}
\varrho: \ginfz{M;E} \too \ginf{\Sigma;E}, \quad f\longmapsto
f\upharpoonright\Sigma
\end{equation}
extends by continuity to a bounded linear map between Sobolev
spaces
\begin{equation}\label{eq3.10}
L^2_s(M,E) \too L^2_{s-\half}(\Sigma,E_\gS),\quad  s > \half.
\end{equation}
For the domain of $A_{\max}$ this can be pushed a bit further.
Namely, for $s\geq 0$ the trace map extends by continuity to a
bounded linear map
\begin{equation}\label{eq3.11}
\cD(A_{\max,s}) \too L_{s-\half}^2(\Sigma,E_\Sigma), \quad s\geq 0,
\end{equation}
that is, there is a constant $C_s$, such that for $f\in L_s^2(M,E)$ with $Af\in L^2_s(M,E)$
\begin{equation}\label{eq3.12}
\|\varrho f\|_{s-\half} \leq C_s(\|f\|_s+\|Af\|_s)\qquad (s\ge 0).
\end{equation}
Here $\|f\|_s$ denotes the Sobolev norm of order s. Furthermore, norms of
operators from $L^2_s$ to $L^2_{s'}$ will be denoted by $\| \cdot
\|_{s,s'}$, and $\| \cdot \|_{\infty}$ denotes the sup-norm of a
function.
\glossary{$\|\cdot\|_{s,s'}$}
\glossary{$\|\cdot\|_s$}
\glossary{$\CL^0(\Sigma;E,F)$}

The proof of \eqref{eq3.11}, \eqref{eq3.12}
in \cite[Theorem 13.8 and Corollary 13.9]{BooWoj:EBP} simplifies \cite{LioMag:NHBVPAI} for operators of Dirac type but remains valid for any
elliptic differential operator of first order, cf. also \cite[Lemma 6.1]{BruLes:BVP}.
\index{boundary condition}
\index{boundary condition!regular}
\index{boundary condition!strongly regular}
\glossary{$A_P$}
\glossary{$A_{\max,P}$} 

\begin{dfn}\label{def3.1}
{\rm (a)} Let $\CL^0(\Sigma;E_\Sigma,G)$ denote the space of {\em
classical pseudodifferential operators of order} $0$, acting from
sections of $E_\Sigma$ to sections of another smooth Hermitian vector
bundle $G$ over $\Sigma$.

\noi {\rm (b)} Let $P\in \CL^0(\Sigma;E_\Sigma,G)$. We denote by $A_P$ the operator $A$ acting on the domain
\begin{equation}\label{eq3.13}
\cD(A_P) := \bigsetdef{f\in L^2_1(M,E)}{P(\varrho f) = 0},
\end{equation}
and by $A_{\max,P}$ the operator $A$ acting on the domain
\begin{equation}\label{eq3.14}
\cD(A_{\max,P}) := \bigsetdef{f\in L^2(M,E)}{Af\in L^2(M,F),
P(\varrho f) = 0}.
\end{equation}

\noi {\rm (c)} The boundary condition $P$ for $A$ is called
\textit{regular} if $A_{\max,P} = A_P$, i.e., if $f,Af\in L^2,
P(\varrho f) =0$ already implies that $f\in L_1^2(M,E)$.

\noi {\rm (d)} The boundary condition $P$ is called \textit{strongly regular} if $f\in L^2, Af\in
L^2_k$, $P(\varrho f)=0$ already implies $f\in L^2_{k+1}(M,E)$,
$k\ge 0$.
\end{dfn}

\eqref{eq3.12} shows that $\cD(A_P)$ is in any case a closed subspace of $L^2_1(M,E)$.

\begin{prop}\label{prop3.2}
Let $P$ be regular for $A$. Then $A_P$ is a closed semi-Fredholm operator with finite--dimensional kernel.
\end{prop}
\begin{proof}
Let $(f_n)\subset \cD(A_P)$ be a sequence with $f_n\to f$ and
$Af_n\to g\in L^2(M,F)$. Then $Af=g$ weakly, hence $f\in\cD(A_{\max})=\cD(A_{\max,0})$ and in view of
\eqref{eq3.11}, \eqref{eq3.12} we have $P(\varrho f)=0$ and the regularity of $P$
implies $f\in L_1^2(M,E)$, thus $f\in\cD(A_P)$.

Hence $A_P$ is closed and thus $\cD(A_P)$ is complete in the graph
norm. In view of the Closed Graph Theorem 
the previous argument shows that the inclusion $\iota: \cD(A_P)\hookrightarrow L_1^2(M,E)$
is bounded. $\iota$ is thus an injective bounded linear map from the
Hilbert space $\cD(A_P)$ (equipped with the graph norm) onto a
closed subspace of $L_1^2(M,E)$; the closedness is also a consequence of the 
argument at the beginning of this proof. Consequently, on $\cD(A_P)$ the
graph norm and the $L_1^2$-norm are equivalent. I.e., for $f\in
\cD(A_P)$ we have
\begin{equation}\label{eq:Garding-inequ}
\frac{1}{C}\|f\|_1 \leq \|f\|_0 + \|Af\|_0 \leq C\|f\|_1.
\end{equation}
Since the inclusion $L_1^2(M,E)\hookrightarrow L^2(M,E)$ is compact,
the inclusion $\cD(A_P)\hookrightarrow L^2(M,E)$ is compact, too. Consequently, $A_P$ is a
semi-Fredholm operator with finite--dimensional kernel.
\end{proof}
\index{dual|see{boundary condition}}

\begin{remark}\label{rem3.3}\indent\par
\enum{1} $P=\Id$ is strongly regular and its domain $\cD(A_P) =
L_{1,0}^2(M,E)$ equals the closure of $\ginfz{M\setminus\Sigma;E}$
in $L_1^2(M,E)$. This is seen by induction. Namely, if $f\in L^2,
Af = g\in L^2$ and $\varrho f=0$ we may extend\footnote{We think of $M$ as being a subset of an open manifold
$\widetilde M$ to which $A$ can be extended as an elliptic operator.} 
$f$ by $0$ to obtain $\widetilde{f}, \widetilde{A}\widetilde{f}\in
L^2_{\loc}$ and hence $f\in L_1^2$.

For the induction we first note that similarly as in \cite[Cor.
2.14]{BruLes:BVP} one shows that to the map
\begin{equation} \label{eq3.15}
  \varrho^{(k+1)}:L^2_{k+1}(M,E)\longrightarrow \bigoplus_{j=0}^k
  L^2_{k-j+\half}(\Sigma,E_\gS),\quad    f\longmapsto (\varrho A^jf)_{j=0}^k
\end{equation}
there exists a continuous linear right--inverse
\begin{equation}\label{eq3.16}
   e^{(k+1)}: \bigoplus_{j=0}^k  L^2_{k-j+\half}(\Sigma,E_\gS)\longrightarrow
   L^2_{k+1}(M,E).
\end{equation}

To complete the induction consider $f\in L^2,\, Af\in L^2_k,\, \varrho
f=0$. By induction we may assume $f\in L^2_k$. Put
\begin{equation}\label{eq3.17}
   f_1:=f-e^{(k+1)}(0,\varrho Af,\ldots,\varrho A^kf).
\end{equation}
Then $f-f_1\in L^2_{k+1}, f_1\in L^2_k$ and $\varrho A^jf_1=0,
j=0,\ldots,k$. Hence we may extend $f$ by $0$ to obtain $\widetilde
f\in L^2$ with $\widetilde A^j \widetilde f\in L^2$, $j=0,\ldots,k+1$.
From local elliptic regularity we infer $\widetilde f\in
L^2_{k+1,\loc}$ and thus $f\in L^2_{k+1}$.

$(\im A_{\Id})^\perp = \bigsetdef {f\in L^2(M,F)}{A^tf=0}$ which is
known to be (or see Proposition \plref{prop3.19} and Theorem \plref{t:cobord} below)
infinite--dimensional if $\dim M>1$. Hence we cannot
expect regularity to imply that $A_P$ is Fredholm. However, if $P$ and the dual boundary condition
for $A^t$ are regular then $A_P$ is Fredholm.
\index{boundary condition!dual}

\enum{2} Regular boundary conditions are closely related to the
well-posed boundary conditions of Seeley \cite{See:TPO, BooWoj:EBP}.
One of the main results in \cite{BruLes:BVP} states that for symmetric
Dirac operators and  symmetric boundary conditions (given by
operators $P$ with closed range) regularity and well--posedness are
equivalent.

\enum{3} It is well-known that if $P$ has closed range then $P$ may
be replaced by an orthogonal projection with the same kernel.
In this setting the dual boundary condition can easily be computed:
\end{remark}
\begin{prop} \label{prop3.4}
Let $P\in \CL^0(\Sigma;E_\Sigma),\,  P=P^2=P^\ast$. Then
\begin{equation}
(A_P)^\ast = A^t_{\max,(\Id-P)J_0^t}.
\end{equation}
\end{prop}
\begin{proof} This follows easily from Green's formula Lemma \plref{l:Green-formula}.
\end{proof}
\index{Green's formula}
%\square

We recall from \cite[Definition 20.1.1]{Hor:ALPDOIII} (see also \cite[Remark 18.2d]{BooWoj:EBP}):

\index{boundary condition!local elliptic}
\index{boundary condition!{\v S}apiro-Lopatinski{\v i}} 
\glossary{$M^+_{y,\zeta}$}

\begin{dfn}\label{def3.2}Let $P\in \CL^0(\Sigma;E_\Sigma,G)$. We say that
$P$ defines a {\em local elliptic boundary condition} for our 
operator $A$ (or, equivalently, we say $P$ satisfies the {\v
S}apiro-Lopatinski{\v i} condition for $A$), if and only if the
leading symbol $\sigma^0_P$ of $P$ maps the space $M^+_{y,\zeta}$
isomorphically onto the fibre $G_y$ for each point $y\in \Sigma$ and
each cotangent vector $\zeta\in T^*_y(\Sigma)$, $\zeta\ne 0$. Here
$M^+_{y,\zeta}$ denotes the space of boundary values of
bounded solutions $u$ on the
positive real line of the linear system
$\frac d{dt}u+\gs^1_{B(0)}(y,\zeta)u=0$ of ordinary differential equations.
%where $\gs^1_{B(0)}$ denotes the leading symbol of the tangential
%operator $B(0)$\,.
\end{dfn}

\begin{remark}\label{rem-def-local-elliptic}
Note that a solution of the ordinary
differential equation $\frac d{dt}u+\gs^1_{B(0)}(y,\zeta)u=0$
is bounded if and only if the initial value $u_0$
belongs to the range of the positive
spectral projection $P_+(\sigma^1_{B(0)}(y,\zeta))$ (cf. Section \plref{s:sectorial} below)
of the matrix $\sigma^1_{B(0)}(y,\zeta)$.
Hence $M^+_{y,\zeta}=\im P_+(\sigma^1_{B(0)}(y,\zeta))$ and
local ellipticity means that $\sigma^0_P$ maps $\im P_+(\sigma^1_{B(0)}(y,\zeta))$
isomorphically onto $G_y$.
\end{remark}

We obtain from \cite[Theorem 20.1.2, Theorem 20.1.8]{Hor:ALPDOIII}
(differently also along the lines of \cite[Theorem 19.6]{BooWoj:EBP}):

\begin{prop}\label{prop3.5}Any $P$ satisfying the {\v S}apiro-Lopatinski{\v i} condition
for $A$ is strongly regular and the corresponding $A_P$ is a Fredholm operator.
\end{prop}

\section{Sectorial projections of an elliptic operator}\label{s:sectorial}

\subsection{Parameter dependent ellipticity}\label{ss:parameter-dependent-ellipticity}
Regarding properties of the tangential operator $B_0$ on $\gS$, it\index{operator!tangential}
is natural to distinguish three situations of increasing generality:

\begin{enumerate}
\item[(i)] $B_0$ is formally selfadjoint,
\item[(ii)] $B_0-B_0^t$ is an operator of order zero, and
\item[(iii)] $B_0$ is the tangential operator of an elliptic
operator over the whole manifold $M$.
\end{enumerate}
Whereas (i) implies that the spectrum $\spec(B_0)$ of $B_0$ is
contained in the real axis and (ii) that for all $p\in \gS$ and
$\xi\in T_p^*\gS$ the leading symbol $\gs^1_{B(0)}(p,\xi)\in
\End(E_p)$ is selfadjoint, the general case (iii) a priori only implies that
$\gs^1_{B(0)}(p,\xi)$ has no eigenvalues on the imaginary axis $i\R$
for all $p\in \gS$ and $\xi\in T_p^*\gS\setminus \{0\}$, as
explained above after \eqref{eq3.8}.

One may ask, what consequences can be drawn from the general
property (iii) for the spectrum of $B_0$? A first answer is
Proposition \ref{p:parameter-dependent-general} below.
In fact, (iii) contains more information than just that
the leading symbol $\gs^1_{B(0)}(p,\xi)$
has no eigenvalues on $i\R$.

\index{pseudodifferential|see{operator}}
\index{operator!pseudodifferential with parameter}

For the convenience of the reader let us briefly recall the notion
of a (pseudo)-differential operator with parameter, cf. Shubin
\cite[Section II.9]{Shu:POST}.

\index{conic set}
\glossary{$\Lambda$}
\glossary{$S^m(U,\R^n\times\gL)$}
 
Let $\gL\subset\C$ be an open conic subset, i.e., $z\in\gL,r>0 \Rightarrow rz\in\gL$.
For an open subset $U\subset\R^n$ let $S^m(U,\R^n\times\gL)$ denote the space of smooth functions
\[
a  : U\times\R^n\times\gL \too \C,\quad (x,\xi,\gl)\longmapsto a(x,\xi,\gl),
\]
such that for multi-indices $\ga,\gb\in\Z_+^n,\, \gamma\in\Z_+^2$ and each compact subset $K\subset U$
we have
\[
\bigl|\pl_x^{\ga}\pl_\xi^\beta\pl_\lambda^\gamma a(x,\xi,\gl)\bigr|
\le C_K\bigl(1+|\xi|+|\gl|\bigr)^{m-|\gb|-|\gamma|}\,.
\]
We emphasize that $\pl_\lambda^\gamma$ denotes real partial derivatives -- we do not require holomorphicity in $\gl$\,.

In other words, $S^m(U,\R^n\times \gL)$ are the symbols of H{\"o}rmander type (1,0).

\index{symbol!classical}
\glossary{$\CS^m(U,\R^n\times \gL)$}
We shall call a symbol $a\in S^m(U,\R^n\times \gL)$ {\em classical} if it has an asymptotic expansion
\begin{equation}\label{e:classical}
a \sim \sum_{j= 0}^\infty a_{m-j}\,,
\end{equation}
where $a_{m-j}\in S^{m-j}(U,\R^n\times \gL)$ with homogeneity
\[
a_{m-j}(r\xi,r\gl)=r^{m-j}a_{m-j}(\xi,\gl) \text{ for } r\ge 1,\ |\xi|^2+|\gl|^2\ge 1.
\]
We denote the classical symbols by $\CS^m(U,\R^n\times \gL)\subset S^m(U,\R^n\times \gL)$ .

\index{operator!classical pseudodifferential}
\begin{dfn}\label{d:pseudo-param}
Let $E_{\gS}$ be a complex vector bundle of finite fibre dimension $N$ over a smooth closed
manifold $\gS$ and let $\gL\subset\C$ be open and conic. A {\em classical pseudo\-differential
operator of order $m$ with parameter} $\gl\in\gL$ is a family $B(\gl)\in\CL^m(\gS;E_{\gS}),\, \gl\in\gL$, such that locally
$B(\gl)$ is given by
\[
B(\gl)u(x)=(2\pi)^{-n}\int_{\R^n}\int_Ue^{i\lla x-y,\xi\rra} b(x,\xi,\gl)u(y)dyd\xi
\]
with $b$ an $N\times N$ matrix of functions belonging to
$\CS^m(U,\R^n\times \gL)$.
\end{dfn}

\begin{remark}
\enum{1} A {\em pseudodifferential operator with parameter} is more than just a map from
$\gL$ to the space of pseudodifferential operators.

\enum{2} Our definition of a {\em pseudodifferential operator with parameter} is slightly different from that of Shubin \cite[Section II.9]{Shu:POST};
however, the main results of loc. cit. do also hold for this class of operators.
\end{remark}

\index{operator!pseudodifferential with parameter}
\index{operator!pseudodifferential with parameter!symbol}
\index{symbol!parametric leading symbol}

The leading symbol of a classical pseudodifferential operator $B$ of order $m$
with parameter is
now a smooth function $\gs^m_B(x,\xi,\gl)$ on $T^*\gS\times\gL\,\setminus\, \{(x,0,0)\mid
x\in\gS\}$ which is homogeneous in the following sense
\[
\gs^m_B(x,r\xi,r\gl)=r^m\gs^m_B(x,\xi,\gl) \text{ for } (\xi,\gl)\ne (0,0),\, r>0.
\]

\index{operator!parameter dependent elliptic}
{\em Parameter dependent ellipticity} is defined as invertibility of this homogeneous leading symbol.
The basic example of a pseudodifferential operator with parameter is the resolvent of an
elliptic differential operator.

\index{conic set}
\begin{prop}\label{p:parameter-dependent-general}
Let $\gS$ be a closed manifold and let $B\in\Diff^1(\gS;E_{\gS})$ be a first order differential
operator. Let $\gL\subset\C$ be an open conic subset such that $B-\gl, \gl\in\Lambda,$ is parameter
dependent elliptic, i.e., for each $(p,\xi)\in T^*\gS,\,
\xi\ne 0$, and each $\gl\in\gL$ the homomorphism
\[
\gs^1_{B}(p,\xi)-\gl:E_p \too E_p
\]
is invertible. Then there exists $R>0$ such that $B-\gl$ is
invertible for $\gl\in\gL,\, |\gl|\ge R$, and we have
\begin{equation}\label{ML-revision-20090314-1}
\norm{(B-\gl)\ii}_{s,s+\ga}\le C_\ga|\gl|^{-1+\ga}
\end{equation}
for such $\gl$ and $0\le \ga\le 1$.

%\item \dots parametrix $F(\gl)$ for $(B-\gl)$ and error estimate
%$\norm{(B-\gl)\ii-F(\gl}\sim |\gl|^{-2}$ for $|\gl|\to\infty$
\end{prop}

For the proof see \cite[Theorem 9.3]{Shu:POST}. In our situation
Proposition \plref{p:parameter-dependent-general} has the following
consequences:

\begin{figure}
\ifarxiv{
\centerline{\includegraphics[height=5cm]{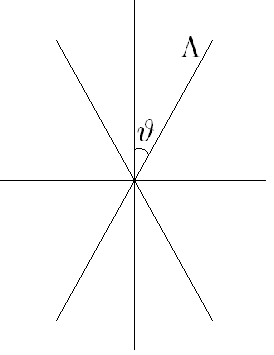}}}{% arxiv
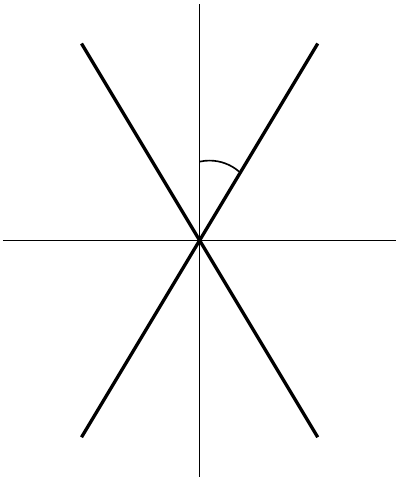}%endif
\caption{\label{f:spec-section}
Construction of a closed cone $\gL$ such that $B_0-\gl$ is elliptic with parameter
$\gl\in\gL$.}
\end{figure}

\begin{prop}\label{p:parameter-dependent}
Let $\gS$ be a closed manifold and let
\[
D=J_x\Bigl(\dd x+B_x\Bigr)
\]
be a first order elliptic differential operator on the collar $[0,\gve)\times\gS$. Then
\begin{thmenumm}
\item $B_0-\gl$ is parameter dependent elliptic in an open conic neighborhood $\gL$ of the imaginary axis $i\R$.
\item $B_0$ is an operator with compact resolvent, $\spec B_0$ consists of a discrete set of eigenvalues
of finite multiplicity. At most finitely many eigenvalues lie on the imaginary axis $i\R$.
\end{thmenumm}
\end{prop}
For an eigenvalue $\lambda$ even the \emph{generalized eigenspace}
$\bigcup_N \ker (B_0-\gl)^N$ is finite--dimensional; note that $B_0$ is not necessarily selfadjoint.
\index{generalized eigenspace}

\begin{proof}
From the ellipticity of $D$ we infer that $\gs^1_{B(0)}(p,\xi)-it$ is invertible for $(p,\xi,t)\in T^*\gS\times\R,
\, (\xi,t)\ne (0,0)$. Since
\[
\bigcup_{(p,\xi)\in T^*\gS,\, |\xi|=1} \spec\gs^1_{B(0)}(p,\xi)
\]
is bounded in $\C$ and in view of the homogeneity we find an angle $\vartheta>0$ such that
\[
\spec\gs^1_{B(0)}(p,\xi)\cap\gL=\emptyset\,.
\]
Here $\gL$ is as in Figure \ref{f:spec-section}.

(a) now follows from the previous proposition. Since $B_0$ is elliptic, its spectrum is either discrete or equals $\C$.
The previous lemma implies that $B_0-\gl$ is invertible for $\gl\in\gL$ large enough.
Hence we conclude that $\spec B_0$ is discrete and that (b) holds.
\end{proof}

\subsection{Sectorial operators: abstract Hilbert space framework}\label{ss:sectorial-abstract}

We shall now discuss the {\em positive} respectively {\em negative sectorial spectral projections}
of an elliptic differential operator $B$ of first order on a closed manifold $\gS$.
We start with a purely functional analytic discussion.

\subsubsection{Idempotents in a Hilbert space}\label{sss:idempotents-hilbert-space}
Let us briefly summarize some facts about (not necessarily bounded) idempotents
in a Hilbert space. A densely defined operator $P$ in the
Hilbert space $H$ is an \emph{idempotent} if $P\circ P=P$; as an identity between
unbounded operators $P\circ P=P$ means $\im P\subset \cD(P)$ \emph{and}
$P(Px)=Px$ for $x\in\cD(P)$.

\index{idempotent}
\index{operator!unbounded!idempotent}
\index{idempotent!unbounded}
\glossary{$P_{U,V}$}

Given subspaces $U,V\subset H$ with
\begin{align}
 &U\cap V=\{0\},\label{eq:ML20090330-1}\\
 & U+ V \text{ dense in } H,\label{eq:ML20090330-2}
\end{align}
the projection $P_{U,V}$ along $U$ onto $V$ is an (not necessarily bounded) idempotent
and every idempotent $P$ in $H$ is of this form with $\cD(P_{U,V})=U+V, U=\ker P$ and $V=\im P$.

It is easy to see that $P_{U,V}^*=P_{V^\perp,U^\perp}$ is also an (not necessarily densely
defined) idempotent. Thus $P_{U,V}$ is closable iff $U^\perp+V^\perp$ is dense
or, equivalently $\ovl{U}\cap \ovl{V}=\{0\}$.
In that case, the closure of $P_{U,V}$ is $\ovl{P_{U,V}}=P_{\ovl{U},\ovl{V}}$.
Consequently, $P_{U,V}$ is a closed operator if and only if $U,V$ are closed subspaces
of $H$.

\glossary{$P_{\ort}$}
\index{orthogonalization}

\begin{lemma}\label{l:idempotent-invertible}
\textup{(a)} Let $P_{U,V}$ be an idempotent in the Hilbert space $H$,
where $U,V$ are closed subspaces satisfying \eqref{eq:ML20090330-1}, \eqref{eq:ML20090330-2} above.
Then $P_{U,V}$ is bounded if and only if $U+V=H$.

\noi \textup{(b)} Let $P=P_{U,V}$ be a bounded idempotent in the Hilbert space
$H$. Then $P+\Id-P^*$ is an invertible operator.

Denote by $P_{\ort}$ the \emph{orthogonalization} of $P$, i.e., $P_{\ort}=P_{V^\perp,V}$
is the orthogonal projection onto $\im P$. Then we have
\begin{align}
         P_{\ort}&=P(P+\Id-P^*)\ii \label{e:p-ort},\\
         (P^*)_{\ort}&=(P+\Id-P^*)\ii P.\label{e:pstar-ort}
\end{align}
\end{lemma}
\begin{proof} \noi \enum{a} is a consequence of the Closed Graph Theorem.

\noi \enum{b} By \enum{a} $U,V$ are closed subspaces of $H$ satisfying $U\cap V=\{0\}, U\oplus V=H$.
Then $\Id-P^*=P_{U^\perp,V^\perp}$. Since bounded idempotents are bounded below
$P\rrr U^\perp$ maps $U^\perp=\ker P^\perp$ bijectively onto $V$ and
$\Id-P^*$ maps $U=\ker(\Id-P^*)^\perp$ bijectively onto $V^\perp$. Hence $P+\Id-P^*$
is invertible. Moreover, this description gives $(P+\Id-P^*)\ii$ explicitly:
given $v\in V=\im P$ then $(P+\Id-P^*)\ii v$ is the unique element $\xi\in U^\perp$
with $P\xi=v$ and thus $P(P+\Id-P^*)\ii v=v$. Furthermore, if $v\in V^\perp$ then
$(P+\Id-P^*)\ii v$ is the unique element $\eta\in U=\ker P$ with $(\Id-P^*)\eta=v$.
This proves $P(P+\Id-P^*)\ii=P_{V^\perp,V}=P_{\ort}$. The equality \eqref{e:pstar-ort}
is proved similarly. 

Alternatively, one may apply \eqref{e:p-ort} to $P^*$ to find
$(P^*)_{\ort}=P^*(P^*+\Id-P)\ii$. Then \eqref{e:pstar-ort} follows from
$(P+\Id-P^*)P^*=PP^*=P(P^*+\Id-P)$.
\end{proof}
Our construction of $P_{\ort}$ is a slight modification of the construction given by
M. Birman and A. Solomyak and disseminated in \cite[Lemma 12.8]{BooWoj:EBP}.

Lemma \plref{l:idempotent-invertible} \enum{a} shows that unbounded idempotents in a Hilbert space
are abundant. See also Example \ref{ex:sectorial-operator} below.

\subsubsection{The semigroups $Q_\pm(x)$ of a sectorial operator}
\begin{figure}
\ifarxiv{\centerline{\includegraphics[height=6cm]{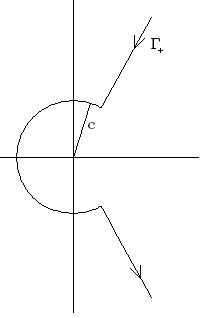}\hspace*{1cm}% arxiv
\includegraphics[scale=0.6]{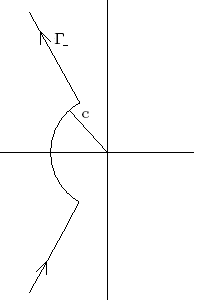}}}{%arxiv
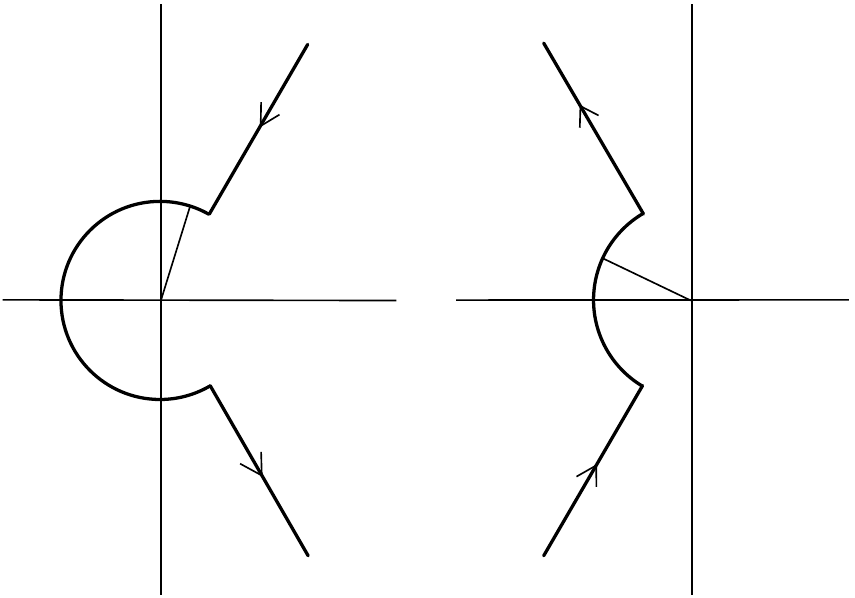}
\caption{\label{f:gamma_pm} The contours $\gG_{\pm}$ in the plane defining the semigroups
$Q_{\pm}$.}
\end{figure}

In this subsection let $H$ be a separable Hilbert space and $B$ a closed operator in $H$.
\index{operator!sectorial}
\index{sectorial|see{operator, spectral projection}}
\index{operator!weakly sectorial}

\begin{dfn}\label{def:sectorial-operator} We call $B$ a \emph{weakly sectorial} operator if
\begin{enumerate}
\item $B$ has compact resolvent.
\item There exists a closed conic neighborhood
$\Lambda$ of $i\R$ such that $\spec B\cap \Lambda$ is finite
and
\begin{equation}\label{e:weak-resolvent-estimate}
\|(B-\lambda)^{-1}\|={\mathcal O}(|\lambda|^{-\ga}),\quad |\lambda|\to\infty,\lambda\in\Lambda,
\end{equation}
for some $0<\ga\le 1$.
\end{enumerate}
If $\alpha=1$ then we call $B$ a \emph{sectorial} operator.
\end{dfn}

We fix a weakly sectorial operator $B$ in the sense of Definition \ref{def:sectorial-operator}.
%and choose contours $\Gamma_\pm$ as in Figure \ref{f:gamma_pm}.

\begin{convent}\label{fix_c_gamma_omega}
\begin{thmenumm}
\item $c>0$ is chosen large enough such that
\begin{equation}\label{e:spectral-gap}
\spec B\cap \bigsetdef{z\in\C}{ |z|= c}=\emptyset,
\end{equation}
and such that $\bigsetdef{z\in\C}{ |z|= c}$ contains all eigenvalues on the imaginary axis.

\item We specify two complementary contours $\gG_{\pm}$ in the plane as sketched in Figure \ref{f:gamma_pm}
with $\Gamma_+$ encircling, up to finitely many exceptions, the eigenvalues of $B$ with real part $\ge 0$
and $\Gamma_-$ encircling the remaining eigenvalues.
Of course, for this to be possible $c$ has to be large enough.
%; respectively if there are eigenvalues on
%the imaginary axis outside the disk of radius $c$ then $\Gamma_-$ will look slightly different than sketched.
\end{thmenumm}
\end{convent}

\glossary{$Q_\pm(x)$} 
\begin{dfn}\label{d:gen-spec-proj}
\begin{align}
Q_+(x)&:=\frac 1{2\pi i}\int_{\gG_+} e^{-\gl x}(\gl-B)\ii\, d\gl\,, \qquad x> 0,\label{e:qplus}\\
      &= \Id +\frac 1{2\pi i}\int_{\gG_+}e^{-\gl x} \gl^{-1}B(\gl-B)\ii\, d\gl,\label{e:qplus-b}\\
Q_-(x)&:=\frac 1{2\pi i}\int_{\gG_-} e^{-\gl x}(\gl-B)\ii\, d\gl\,,  \qquad x < 0,\label{e:qminus}\\
      &= \frac 1{2\pi i}\int_{\gG_-}e^{-\gl x} \gl^{-1}B(\gl-B)\ii\, d\gl.\label{e:qminus-b}
\end{align}
When the dependence on $B$ matters we will write $Q_\pm(x,B)$.
\end{dfn}
Formulas \eqref{e:qplus-b}, \eqref{e:qminus-b} are obtained by adding and subtracting $\gl^{-1}$ inside
the integral and taking into account that $0$ lies inside $\gG_+$ but outside $\gG_-$.

$Q_\pm(x)$ are certainly bounded operators for $x>0$ $(x<0)$. Heuristically, $Q_\pm(0)$
should be the positive/negative sectorial spectral projection of $B$, obtained from
holomorphic functional calculus. However, $Q_\pm(0)$ is not defined everywhere.
To avoid ambiguities, we shall keep to the following two
rigorous definitions instead of dealing directly with $Q_\pm(0)$.

\glossary{$P_{\pm,0}$} 
\begin{dfn}\label{def:P-pm}
We put $\cD(P_{+,0}):=\bigsetdef{\xi\in H}{\lim\limits_{x\to 0+} Q_+(x)\xi \text{ exists}}$
and $P_{+,0}\xi:=\lim\limits_{x\to 0+} Q_+(x)\xi$ for $\xi\in\cD(P_{+,0})$.
$P_{-,0}$ is defined analogously using $Q_-(x)$.
\end{dfn}

\eqref{e:qplus-b}, \eqref{e:qminus-b} and the estimate \eqref{e:weak-resolvent-estimate}
imply that $\cD(B)\subset\cD(P_{+,0})$ and
for $\xi\in\cD(B)$ we have
\begin{align}
P_{+,0}\xi &= \xi +\frac 1{2\pi i}\int_{\gG_+} \gl^{-1}(\gl-B)\ii\, d\gl\; (B\xi)\label{e:pplus}\\
P_{-,0}\xi&=\frac 1{2\pi i}\int_{\gG_-} \gl^{-1}(\gl-B)\ii\, d\gl \;(B\xi),\label{e:pminus}
\end{align}
thus $P_{\pm,0}$ is densely defined ($\cD(B)$ is indeed
a core for $P_{\pm,0}$). Note that
$Q_\pm(x,B)^*=Q_\pm(x,B^*)$ (cf. Prop. \ref{p:properties-Qpm}), hence
the densely defined operator $P_{\pm,0}(B^*)$ is contained in
$P_{\pm,0}(B)^*$. Thus $P_{\pm,0}$ is closable:
%Its closure will be denoted by $P_\pm$.

\glossary{$P_\pm$} 
\index{spectral projection!sectorial}
\begin{dfn}\label{d:p-plus} The closure of $P_{\pm,0}$
will be called the {\em positive/negative sectorial spectral projection} $P_\pm$ of
$B$.
\end{dfn}

\begin{prop}\label{p:properties-Qpm} For $x,y>0$ we have
\begin{thmenumm}
\item $Q_+(x,B)^*=Q_+(x,B^*), \quad Q_-(-x,B)^*=Q_-(-x,B^*)$.
\item $Q_+(x)Q_+(y)=Q_+(x+y)$.
\item $Q_+$ is differentiable and $\frac{dQ_+}{dx}(x)=-BQ_+(x)$.
\item $Q_+(x)Q_-(-y)=Q_-(-y)Q_+(x)=0$.
\item $P_+Q_+(x)\subset Q_+(x)P_+$, $P_+Q_-(-x)=0$.
\end{thmenumm}
\end{prop}

\begin{proof} The proof is straightforward and analogous to the proof in \cite{Pon:SAZ} of the
fact that $P_+$ is an idempotent.
\end{proof}

\begin{figure}
\ifarxiv{\centerline{\includegraphics[height=6cm]{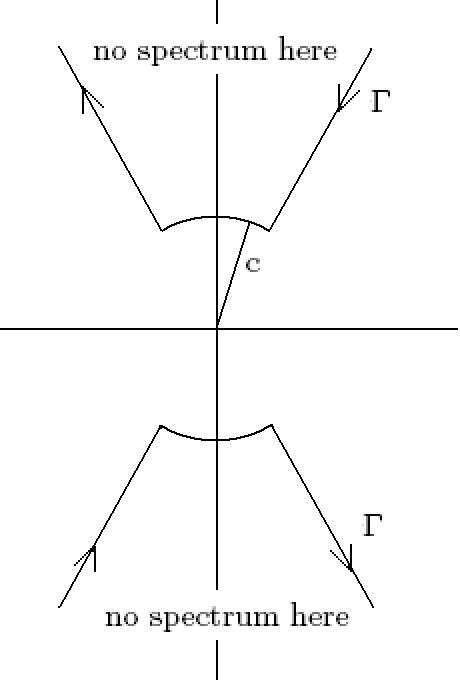}}}{%arxiv
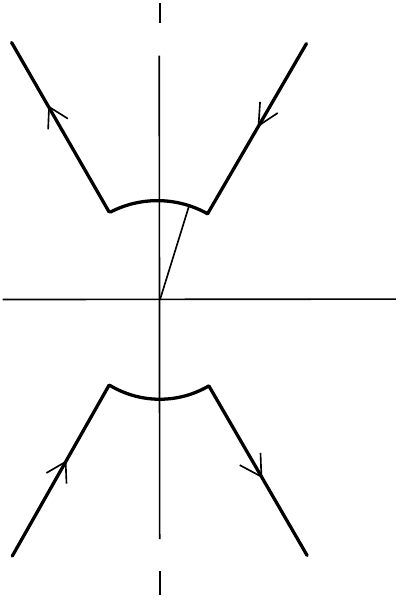}
\caption{\label{f:no-spec-here}
A two-component contour $\gG$\,, separating an inner sector around the real axis where all eigenvalues
of $B_0$ show up, from two outer sectors which totally belong to
the resolvent set of $B_0$.}
\end{figure}

\begin{cor}\label{c:properties_Ppm}
$P_\pm$ are complementary, i.e., $P_+=\Id-P_-$, (possibly unbounded) idempotents in $H$.
\end{cor}
\begin{proof}
Since $\cD(B)$ is a core for $P_\pm$ it suffices to check that
for $\xi\in\cD(B)$ we have $P_\pm^2\xi=P_\pm\xi$ and $(P_++P_-)\xi=\xi$.

If $\xi\in\cD(B)$ then using Proposition \ref{p:properties-Qpm}
we find
\begin{equation}\begin{split}
      Q_+(x)P_+\xi&=\lim_{y\to 0+}Q_+(x)Q_+(y)\xi=\lim_{y\to 0+}Q_+(x+y)\xi=Q_+(x)\xi,
    \end{split}
\end{equation}
hence $P_+\xi\in\cD(P_{+,0})\subset \cD(P_+)$ and $P_+^2\xi=P_+\xi$.

Secondly, we take a $\xi\in\cD(B)$ and find
\[
(P_++P_-)\xi=\xi + \frac 1{2\pi i}\int_{\gG} \gl\ii (\gl-B)\ii\, d\gl\; (B\xi)\,,
\]
where $\gG$ is chosen as in Figure \ref{f:no-spec-here}.
Pushing the radius of the circle arches to $\infty$
shows $(P_++P_-)\xi=\xi$.
\end{proof}

The fact that the sectorial projections are a priori unbounded operators
may seem strange. The following example shows that the phenomenon really occurs:

\begin{example}\label{ex:sectorial-operator}
Let $D$ be a discrete selfadjoint positive definite operator in $H$.
I.e., there is an orthonormal basis $(e_n)_{n\in\N}$ of $H$ such
that $De_n=\lambda_n e_n$, where $0<\lambda_1\le \lambda_2\le\ldots\to \infty$.

Pick a parameter $0\le\ga\le 1$ and define the operator $B$ in $H\oplus H$
as follows:
\begin{equation}\begin{split}
      \cD(B)&:=\bigsetdef{(u,v)\in H\oplus H}{v\in\cD(D), Du-D^{2-\ga}v\in H},\\
       B(u,v)&:=(Du-D^{2-\ga}v,-Dv).
     \end{split}
\end{equation}
One immediately checks that for $\lambda\not\in\spec D\cup -\spec D$ the resolvent
of $B$ is given by
\begin{equation}\label{resolv-secexample}
     (B-\lambda)^{-1}(\xi,\eta)=((D-\lambda)^{-1}\xi-2 (D-\lambda)^{-1}D^{2-\ga}(D+\lambda)^{-1}\eta,
                  -(D+\lambda)^{-1}\eta).
\end{equation}
Because of $0\le\ga\le 1$ the resolvent is indeed bounded. Furthermore, \eqref{resolv-secexample}
shows that outside a conic neighborhood of the real axis, equivalently in a conic neighborhood
of the $i\R$, we have an estimate
\begin{equation}
     \| (B-\lambda)^{-1}\|={\mathcal O}(|\lambda|^{-\ga}), \quad |\lambda|\to \infty.
\end{equation}
Hence, if $0<\ga\le 1$ then
$B$ is a weakly sectorial operator in the sense of Definition \ref{def:sectorial-operator},
$\spec B=\spec D\cup-\spec D$
and the positive/negative spectral subspaces of $B$ are given by
\begin{equation}
\begin{split}
       \im P_+(B)&=H\oplus 0,\\
       \ker P_+(B)&=\bigsetdef{(u,D^{\ga-1}u)}{u\in H}=\text{Graph}(D^{\ga-1}).
\end{split}
\end{equation}
Consequently, if $0<\ga< 1$ then $\im P_+(B)\oplus \ker P_+(B)$ is not closed  and hence
the positive sectorial projection $P_+(B)$ is not bounded.
\end{example}
We leave it as an intriguing problem to find an example of a sectorial
operator with decay rate $\ga=1$ in \eqref{e:weak-resolvent-estimate}
such that $P_+$ is unbounded.

\subsection{Sectorial operators: parametric elliptic differential operators}\label{ss:sectorial-geometric}

\subsubsection{The geometric situation}\label{sss:geometric-situation}

We return to our geometric situation and consider the tangential\index{operator!tangential}
operator $B$ (previously denoted by $B(0)$ or $B_0$, for convenience we omit
$(0)$ as long as we do not need $B(x)$)
of an elliptic differential operator $A$ on
a compact manifold with boundary, cf. Section \plref{s:sectwo}, in particular
\eqref{eq:operator-collar}.

Then it is known that the positive sectorial projection is bounded:

\index{operator!parameter dependent elliptic}
\begin{theorem}\label{t:boundedness-positive-sectorial}
Let $B$ be a first order elliptic differential operator
on the closed manifold $\Sigma$. Furthermore, assume that $B-\gl$
is parametric elliptic in an open conic neighborhood $\Lambda$ of
$i\R$. Then the positive/negative sectorial projections $P_\pm$
of $B$ are pseudodifferential operators of order $0$.
In particular $P_\pm$ acts as bounded operator in each Sobolev
space $L_s^2(\Sigma,E_\gS)$.
\end{theorem}

The proof is an adaption of the classical complex power construction
of Seeley \cite{See:CPE}. See Burak \cite{Bur:SPE}, Wojciechowski
\cite{Woj:SFG}, and recently Ponge \cite{Pon:SAZ}.

We also note that it follows from Proposition \plref{p:parameter-dependent-general}
that $B$ is a sectorial operator in the sense of Definition \plref{def:sectorial-operator}.
Also, recall from \eqref{ML-revision-20090314-1} the resolvent estimate:

For all $s\in \R,\ 0\le \ga\le 1$, we have
\begin{equation}\label{e:resolvent-estimate}
\sup_{\gl\in\gG_{\pm}} |\gl|^{1-\ga}\, \norm{(\gl-B)\ii}_{s,s+\ga} \le C(s,\ga),
\end{equation}
where $\norm{\cdot}_{s,s+\ga}$ denotes the operator norm between the Sobolev spaces
$L^2_s(\gS,E_{\gS})$ and $L^2_{s+\ga}(\gS,E_{\gS})$\,,
see also the following remark.

Here and in the following we shall denote the closed interval $[0,\infty)$
by $\R_+\,$. Similarly $\Z_+:=\{0,1,2,3,\ldots\}$.

\begin{remark}%\label{r:sobolev-norms}
\enum{1} We recall that (cf. e.g. \cite[Cor. 2.20]{BruLes:BVP})
\begin{equation}\label{eq4.2}
L^2_s(\R_+\times\Sigma,E_{\gS})=L^2_s(\R_+,L^2(\Sigma,E_{\gS}))\cap
L^2(\R_+,L^2_s(\Sigma,E_{\gS})), \quad s\ge 0.
\end{equation}

In particular, if $s\in\Z_+$ then a Sobolev norm for
$L^2_s(\R_+\times\Sigma,E_{\gS})$ is given by
\begin{equation}\label{eq4.3}
  \|f\|_{L^2_s(\R_+\times\Sigma,E_{\gS})}^2
    =\int_0^\infty \|\pl_x^s f(x)\|^2_{L^2(\Sigma,E_{\gS})}+\|(\Id+|B|)^s
    f(x)\|^2_{L^2(\Sigma,E_{\gS})} dx.
\end{equation}
Since the spaces $L^2_s(\ldots)$ have the interpolation property
(\cite[Sec. 4.2]{Tay:PDEI}, \cite[Sec. 2]{BruLes:BVP}) for $s\ge
0$, it will be sufficient in most cases to deal with integer $s\in
\Z_+$.

\enum{2} Note that since $B$ is elliptic, the Sobolev norms on sections of
$E_{\gS}$ can be defined using $B$, i.e.,
\begin{equation}\label{eq4.4}
   \|\xi\|_{L^2_s(\Sigma,E_{\gS})}^2=\|(\Id+|B|)^s\xi\|^2_{L^2(\Sigma,E_{\gS})}.
\end{equation}

\enum{3} Whenever it is clear from the context whether we are taking norms of
sections over $\R_+\times \Sigma$ or over $\Sigma$ we will, as before, denote
Sobolev norms of order $s$ by a subscript $s$. 
%Furthermore, norms of
%operators from $L^2_s$ to $L^2_{s'}$ will be denoted by $\| \cdot
%\|_{s,s'}$, and $\| \cdot \|_{\infty}$ denotes the sup-norm of a
%function.

\end{remark}

\subsubsection{Mapping properties of $Q_+$.}

The following Proposition will be useful for the study of the mapping properties 
of the invertible double and of the remainder terms in the construction of the Poisson
operator and the Calder{\'o}n projection, see Subsections \plref{ss:continuity-pseudo-inverse}
and \plref{ss:calderon-projector}.
Proposition \plref{prop4.1} establishes a weak convergence of $Q_+(x)\to P_+, x\to 0+$,
in compensation for the generally not valid convergence in the operator norm.

\begin{prop}\label{prop4.1}
  Let $\varphi\in\cinfz{\R_+}, m\in\Z_+$.
\begin{thmenumm}
\item For $s\in\R$ the operator
\begin{equation}
\id_{\R_+}^m \varphi Q_+: \xi\longmapsto \bigl( x \mapsto x^m\/ \varphi(x)\/ Q_+(x)\xi\bigr)
\end{equation}
maps $L^2_s(\Sigma,E_{\gS})$
continuously to $L^2_{\comp}(\R_+,L^2_{s+m+\half}(\Sigma,E_{\gS}))$.

\item For $s\ge -\half$ it maps continuously to
  $L^2_{s+m+\half,\comp}(\R_+\times\Sigma,E_{\gS})$.
\end{thmenumm}
\end{prop}
\begin{proof} \begin{sloppypar}Let us first prove the claim (b).
It is fairly easy to see that  $\id_{\R_+}^m \varphi Q_+$ maps $L^2_s(\Sigma,E_\Sigma)$
continuously to $L^2_{\comp}(\R_+,L^2_{s'}(\Sigma,E_{\gS}))$ for some $s'$.
Thus once we have proved that the range of
$\id_{\R_+}^m \varphi Q_+$
is contained in the space
$L^2_{\comp}(\R_+,L^2_{s+m+1/2}(\Sigma,E_{\gS}))$ the continuity
will follow from the Closed Graph Theorem.\end{sloppypar}

Furthermore, since $\id_{\R_+}^m \varphi Q_+$ commutes with $B$ it suffices to prove
the claim for $s$ large enough: namely, we pick a $\lambda_0$ in the resolvent set of $B$.
Then for arbitrary $s\ge -1/2$ we choose $k$ large enough such that the
claim holds for $s+k$. The claim for $s$ now follows from the identity
\begin{equation}
\id_{\R_+}^m \varphi Q_+|L^2_s=(\lambda_0-B)^k
\bigl(\id_{\R_+}^m \varphi Q_+|L^2_{s+k}\bigr)\bigl((\lambda_0-B)^{-k}|L^2_s\bigr).
\end{equation}

Finally, by complex interpolation (cf. e.g. \cite[Sec. 4.2]{Tay:PDEI})\index{complex interpolation theory}
it suffices therefore to consider $s=n+1/2$, $n\in\Z_+$.
Now pick $\xi\in L^2_{n+1/2}(\Sigma,E_\Sigma)$ and put $f(x):=x^m\varphi(x)Q_+(x)\xi$. It is
straightforward to check that $f$ is smooth on $(0,\infty)\times\Sigma$.
From
\begin{equation}
     \bigl(\frac{d}{d x}+B\bigr)^j f(x) = \Bigl(\bigl(\frac{d}{d x}\bigr)^j x^m\varphi(x)\Bigr) Q_+(x)\xi
\end{equation}
we infer by the boundedness of $P_+(B)$ (according to Theorem \plref{t:boundedness-positive-sectorial}) that
\begin{equation}\label{eq:ML20090330-3}\begin{split}
        (\frac{d}{d x}&+B)^j f\big|_{x=0}\\
                      &=\begin{cases}
                        0, & j=0,\ldots,m-1,\\
                   \bigl(\frac{d}{d x}\bigr)^j|_{x=0} x^m\varphi(x) P_+(B)\xi\in L^2_{n+1/2}(\Sigma,E_\Sigma),& j\ge m,
              \end{cases}\\
                   &\in L^2_{s+m+1/2-j-1/2}(\Sigma,E_\Sigma).
                 \end{split}
\end{equation}
From \eqref{eq:ML20090330-3} and (an obvious adaption of) \cite[Cor. 2.17]{BruLes:BVP} (cf. also Remark \ref{rem3.3})
we infer that $f\in L^2_{s+m+1/2}(\R_+\times\Sigma,E_\gS)$.
Hence (b) is proved.

For $s\ge -1/2$ the claim (a) follows from (b). For arbitrary $s$ we again conjugate
by $(\lambda_0-B)^k$ as above and we reach the conclusion.
\end{proof}
\begin{remark}
The claim of the previous proposition also follows by applying more sophisticated
pseudodifferential techniques (cf. Grubb \cite[Thm. 2.5.7]{Gru:FCP}).
Our proof only uses the basic trace results for Sobolev spaces, the ellipticity
of $B$, and the boundedness of the positive sectorial projection on Sobolev spaces.
The previous proposition can therefore be generalized to situations where
pseudodifferential techniques are not necessarily available.
An abstract version is as follows (see also Subsection \plref{ss:Sobolev-scale} where scales
of Hilbert spaces are recalled to some extent):
\end{remark}

\index{operator!sectorial}
\begin{prop}\label{prop4.1-abstract}
Let $B$ be a sectorial operator in a Hilbert space. Let $H_s:=\cD((B^*B)^{s/2})$,
$s\ge 0$, be the scale of Hilbert spaces of $B^*B$ and $\widetilde H_s:=\cD((BB^*)^{s/2})$
be the scale of Hilbert spaces of $BB^*$. For negative $s$ the spaces
$H_s$ and $\widetilde H_s$ are defined by duality (cf. \cite[Sec. 2.A]{BruLes:BVP}).
Furthermore, put for $s\ge 0$
\[
\cH_s(\R_+,H_\bullet):=\bigcap_{0\le t\le s} L^2_s(\R_+,H_{s-t})
\]
(cf. \cite[Sec. 2, Prop. 2.10]{BruLes:BVP} for other descriptions).

Assume that the positive sectorial projection $P_+$ of $B$ maps
$H_s$  continuously to $\widetilde H_s$ for all $s$.

Let $\varphi\in\cinfz{\R_+}, m\in\Z_+$. Then
\begin{thmenumm}
\item For $s\in\R$ the operator
\begin{equation}\label{e:4.1}
\id_{\R_+}^m \varphi Q_+: \xi\longmapsto \bigl( x \mapsto x^m\/ \varphi(x)\/ Q_+(x)\xi\bigr)
\end{equation}
maps $\cH_s(\R_+,H_{\bullet})$
continuously to $L^2_{\comp}(\R_+,\widetilde H_{s+m+\half})$.

\item For $s\ge -\half$ it maps continuously to
  $\cH_{s+m+\half,\comp}(\R_+, \widetilde H_{\bullet})$.
\end{thmenumm}
\end{prop}

For elliptic pseudodifferential operators the distinction between $H_s$ and $\widetilde H_s$
is, of course, unnecessary. For general unbounded operators,
however, we cannot expect $H_s$ to be equal to $\widetilde H_s$.

%\input{ml-sec4-5.tex}
% ml-sec4-5.tex copied back by ML 21.01.2008
\section{The invertible double}\label{s:invdoub}
We return to the set--up described at the beginning of Subsection \plref{ss:general-set-up}
and give a construction of the invertible double of a general first order elliptic
differential operator. 

\subsection{The construction of $\tilde A_{P(T)}$}\label{ss:construction-tilde-A}
We introduce the operator
\begin{equation}\label{eq3.19}
\widetilde{A} := A\oplus(-A^t): \ginf{M;E\oplus F}\too \ginf{M;F\oplus E}.
\end{equation}
We are going to consider a special class of boundary conditions
for $\widetilde{A}$:

\glossary{$T$}
\glossary{$\widetilde A$}
\glossary{$\widetilde A_{P(T)}$}
\glossary{$P(T)$}    
\begin{dfn}\label{def3.6}
Let $T\in \CL^0(\Sigma;E_\gS,F_\gS)$ be a classical
pseudodifferential operator of order 0, acting from sections of
$E_\gS$ to sections of $F_\gS$. We put
\begin{equation}\label{eq3.20}
P(T) = \begin{pmatrix}  -T & \Id \end{pmatrix}
     \in \CL^0(\Sigma;E_\gS\oplus F_\gS,F_\gS).
\end{equation}
Viewed as an operator in $L^2_s(\Sigma,E_\gS\oplus F_\gS)$ the operator $P(T)$ has closed range which equals
\begin{equation}\label{eq3.21}
\im P(T) = L_s^2(\Sigma,F_\gS)\subset L^2_s(\Sigma,E_\gS\oplus F_\gS).
\end{equation}
Since this is a closed subspace of $L^2_s(\Sigma,E_\gS\oplus F_\gS)$, the boundary condition for $\widetilde{A}$ given by $P(T)$ can be realized by a
pseudodifferential orthogonal projection, as noted in Remark \plref{rem3.3}.3.
\end{dfn}
To be more specific we recall that the realization $\widetilde A_{P(T)}$ of $\widetilde A$ with respect
to the boundary condition $P(T)$ has domain
\begin{equation}\label{eq:dom-inv-double}
\cD(\widetilde A_{P(T)}):=\bigsetdef{\binom{f_+}{f_{-}}\in
L^2_1(M,E\oplus F)}{\varrho f_-=T\varrho f_+}.
\end{equation}

\begin{lemma}\label{lem3.7}
If $T$ is invertible, the dual of the boundary condition $P(T)$ for
$\widetilde{A}$ is
$$P(-J_0^{-1}(T^t)^{-1}J_0^t),$$
i.e.,
\begin{equation}
\bigl(\widetilde{A}_{P(T)}\bigr)^\ast = \widetilde{A}^t_{\max,P(-J_0^{-1}(T^t)^{-1}J_0^t)}\,.
\label{eq:4.4a}
\end{equation}

\end{lemma}

\begin{proof}
Let $f=\binom{f_+}{f_-}\in \cD(\widetilde{A}_{P(T)})$ and
$g=\binom{g_+}{g_-}\in \cD\bigl((\widetilde{A}_{P(T)})^*\bigr)$. Note
that $g_+\in L^2(M,F), g_-\in L^2(M,E)$. Green's formula Lemma \plref{l:Green-formula} yields\index{Green's formula}
\begin{equation}\label{eq3.22}
\begin{split}
   0 &=    \scalar{\widetilde{A}f}{g} - \scalar{f}{\widetilde{A}^t g} \\
     &=  \scalar{Af_+}{g_+} - \scalar{f_+}{A^t g_+} -  \scalar{A^tf_-}{g_-}
             + \scalar{f_-}{A g_-}\\
     &=   - \scalar{J_0\varrho f_+}{\varrho g_+} - \scalar{\varrho f_-}{J_0\varrho g_-}\\
     &=  -  \scalar{\varrho f_+}{J_0^t\varrho g_++ T^tJ_0\varrho g_-}.
\end{split}
\end{equation}
This holds for all $f\in \cD(\widetilde{A}_{P(T)})$ if and only
if
%
%\begin{equation}%\label{eq3.23}
$J_0^t\varrho g_++T^tJ_0\varrho g_- = 0$
%\end{equation}
%
and we reach the conclusion.
\end{proof}
\subsection{The local ellipticity of $P(T)$ for $\tilde A$}\label{ss:local-ellipticity}
\begin{prop}\label{prop-reg}
Let $T$ be an invertible bundle homomorphism from $E_\gS$ to
$F_\gS$ with $J_0^tT>0$. Then the boundary
condition defined by $P(T)$ satisfies the 
{\v S}apiro-Lopatinski{\v i} condition for $\widetilde A$.
\end{prop}
\begin{remark}\label{rem3.10}\indent\par
\enum{1} Obvious candidates for $T$ with $J_0^tT >0$
are $J_0$ and $(J_0^t)^{-1}$. We can in
addition choose $T$ to be unitary by putting $T :=
(J_0J_0^t)^{-1/2}J_0$\,.

\enum{2} Note that if $J_0^tT$ is positive definite then it is
in particular selfadjoint and hence the dual condition for $\widetilde A^t$
(cf. Lemma \plref{lem3.7}) is given by
\begin{equation}\label{eq3.28}
T^{\text{dual}}  =  -J_0^{-1}(T^t)^{-1} J_0^t
=-  (T^tJ_0)^{-1} J_0^t=-(J_0^tT)\ii J_0^t= -T^{-1}.
\end{equation}
In particular we find that fulfilling the assumption $J_0^tT>0$ 
of the preceding proposition implies that the boundary condition 
for $\widetilde A^t$ defined by $P(T^{\text{dual}})$ also satisfies the 
{\v S}apiro-Lopatinski{\v i} condition. 
To see this we recall from \eqref{eq:formal-adjoint} 
that the front bundle endomorphism $J_0(A^t)$ of $A^t$ is given by
$J_0(A^t) = -J_0^t$, so
\begin{equation}\label{eq3.29}
(-J_0^t)^t T^{\text{dual}} = J_0T\ii= J_0 (J_0^t T)\ii J_0^t>0.
\end{equation}
\end{remark}
\begin{proof}
We refer to Remark \plref{rem-def-local-elliptic} and use the
language of idempotents. During this proof,
for an endomorphism $b$ of a finite--dimensional
vector space, $P_\pm(b)$ will denote the spectral projection corresponding
to a closed contour encircling all eigenvalues $\lambda$ with $\Re \lambda\ge 0 $
(respectively $<0$),
cf. Section \plref{s:sectorial}.

From \eqref{eq:operator-collar}, \eqref{eq:formal-adjoint} we see that the tangential
operator of $\widetilde A$ has leading symbol $b_0\oplus -(J_0^t)\ii b_0^* J_0^t$,
$b_0:=\sigma^1_{B(0)}$. Consequently the positive spectral projection of
$b_0\oplus -(J_0^t)\ii b_0^* J_0^t$ is given by $P_+(b_0)\oplus (J_0^t)\ii
P_-(B^t) J_0^t$.
In each $y\in \gS$ and $\zeta\in T^*_y(E_\gS)$, $\zeta\ne 0$, we
consider the {\v S}apiro-Lopatinski{\v i} mapping from
$\im P_+(b_0(y,\zeta))\oplus (J_0^t)\ii \im P_-(b_0(y,\zeta)^*)$ to $F_y$,
given by $\gs^0_{P(T)}=(-T\, \Id)$:
\[
\begin{split}
\im P_+(b_0(y,\zeta))\oplus (J_0^t)\ii \im P_-(b_0(y,\zeta)^*)&\too  F_y\\
\bigl(e_+,(J_0^t)\ii e_-\bigr)  &\longmapsto  -T e_+ + (J_0^t)\ii
e_-.
\end{split}
\]
Multiplying by $J_0^t$ we see that
this map is bijective if and only if the map
\begin{equation}\label{eq:shap-lop-map}
\begin{split}
E_y=\im P_+(b_0(y,\zeta))\oplus \im P_-(b_0(y,\zeta)^*)&\too  E_y\\
    (e_+,e_-)  &\longmapsto  -J_0^tT e_+ +e_-
  \end{split}
\end{equation}
is bijective. To explain why
$E_y=\im P_+(b_0(y,\zeta))\oplus \im P_-(b_0(y,\zeta)^*)$ we note
that $\im P_+(b_0(y,\zeta))^\perp=\ker P_+(b_0(y,\zeta))^*=\ker P_+(b_0(y,\zeta)^*)
=\im P_-(b_0(y,\zeta)^*)$, so the sum on the left of \eqref{eq:shap-lop-map}
is indeed an orthogonal decomposition (cf. also Lemma \plref{l:idempotent-invertible}).

%\begin{sloppypar}
Since the dimensions on the left and on the right side of \eqref{eq:shap-lop-map} 
coincide it suffices to show that the map in \eqref{eq:shap-lop-map} is injective: so let
$-J_0^tT e_++e_-=0,$
$e_+\in \im  P_+(b_0(y,\zeta)),$ $e_-\in \im P_-(b_0(y,\zeta)^*)
=\im P_+(b_0(y,\zeta))^\perp$.
Taking scalar product with $e_+$ we find $0=-\scalar{J_0^tT e_+}{e_+}$.
This implies, since by assumption $J_0^tT>0$, that
$e_+=0$. But then $e_-=0$ as well.
%\end{sloppypar}
\end{proof}

\subsection{The solution space $\ker \tilde A_{P(T)}$} \label{ss:solution-space}
Next we indicate why the boundary conditions of Definition \plref{def3.6} are significant. Before doing
that we recall the various solution spaces and Cauchy data spaces associated to $A$.

\glossary{$Z_\pm^s(A)$}
\glossary{$Z_{+,0}(A)$} 
\glossary{$N_\pm^s(A)$}

\index{Cauchy data space}
\begin{dfn}\label{def3.8}
\noi (a) Put
\[Z^s(A) := \bigsetdef{f\in L^2_s(M,E)}{Af=0}, \quad s\geq 0.\]
$Z^s(A^t) \subset L_s^2(M,F)$ is defined analogously. For brevity we
often write
\begin{equation}\label{eq3.24}
Z_+^s := Z^s(A),\quad Z^s_{-} := Z^s(A^t).
\end{equation}
It follows from \eqref{eq3.10} and \eqref{eq3.11} that the trace map sends $Z_{\pm}^s$
continuously to $L_{s-\half}^2(\Sigma,E_\Sigma)$ respectively
$L^2_{s-\half}(\Sigma,F_\Sigma), s\ge 0$. 

\noi (b) We define the {\em
Cauchy data spaces} by
\begin{equation}\label{eq3.25}
\begin{split}
N_{{\pm}}^s & :=  \varrho(Z_{\pm}^{s+\half}), \quad s\geq -\half,\\
N_{{\pm}}  & :=  N_{\pm}^0.
\end{split}
\end{equation}

\noi (c) Finally let
\begin{equation}\label{eq3.26}
Z_{+,0}(A) := \bigsetdef{f\in L_1^2(M,E)}{Af=0, \varrho f=0}
\end{equation}
denote the space of all {\em inner solutions}. It is the
finite--dimensional kernel of $A_{\Id}$ (cf. Proposition
\ref{prop3.2}). $Z_{-,0}(A):=Z_{+,0}(A^t)$ denotes the corresponding kernel of
$A^t_{\Id}$.

\noi (d)
We say that $A$ has the {\em weak inner unique continuation
property} (\UCP) if $Z_{+,0}(A)=\{0\}$.
\end{dfn}
\index{UCP|see{unique continuation property}}
\index{unique continuation property!weak inner}

\begin{prop}\label{prop3.9}
Let $T$ be as in Definition \plref{def3.6}. Then there is a
canonical inclusion
\[ Z_{+,0}(A) \oplus Z_{-,0}(A)\subset \ker \widetilde{A}_{P(T)}.\]
If, in addition,  $J_0^tT$ is positive definite, then the
inclusion is an equality.
\end{prop}
\begin{proof}If $f_\pm\in Z_{\pm,0}$ then
$\binom{f_+}{f_-}=f \in \ker \widetilde{A}_{P(T)}$ since
$\varrho f_-=0=T \varrho f_+$, cf. \eqref{eq:dom-inv-double}.

Now assume that $J_0^tT$ is positive definite and let
$\binom{f_+}{f_-}=f \in \ker \widetilde{A}_{P(T)}$. Then certainly
$f_{{\pm}}\in Z_{{\pm}}^1$ and $\varrho f_{-} = T \varrho
f_+$. Since $J_0^tT$ is nonnegative and invertible, the
operator $W:= (J_0^tT)^{\half}$ exists and is invertible.
Now Green's formula Lemma \plref{l:Green-formula} yields\index{Green's formula}
\begin{equation}\label{eq3.27}
\begin{split}
\|W\varrho f_{+}\|^2   &=   \scalar{\varrho f_{+}}{J_0^tT\varrho f_{+}}\\
                &= \scalar{J_0\varrho f_+}{\varrho f_{-}}\\
                &= -\scalar{Af_+}{f_{-}}+\scalar{f_+}{A^tf_{-}}\\
                &= 0,
\end{split}
\end{equation}
and since $W$ is invertible we find $\varrho f_{-} =0$. Thus
$\varrho f_+ = T^{-1}\varrho f_- =0$ and hence $f_\pm\in
Z_{\pm,0}$.
\end{proof}

\subsection{The main result}\label{ss:main-result}
Recall from Proposition \ref{prop3.5} that by checking the {\v
S}apiro-Lopatinski{\v i} condition we have not only proved
regularity of $P(T)$ for $\widetilde A$, but also that $\widetilde A_{P(T)}$ is a Fredholm
operator.
Let us summarize the results of Proposition \ref{prop3.2},
Proposition \ref{prop3.5}, Proposition \ref{prop-reg}, Remark \plref{rem3.10},
and Proposition \ref{prop3.9}:

\begin{theorem}\label{thm3.14}
Let $M$ be a compact Riemannian manifold with boundary and
$$A:  \ginf{M;E} \too \ginf{M;F}$$
a first order elliptic differential operator. Write
$$D  =  \Phi^F A(\Phi^E)^{-1}  =:  J_x\Bigl(\frac{d}{dx} + B_x\Bigr)$$
as in \eqref{eq3.5}, \eqref{eq3.6} for suitable isometries
$\Phi^E,\, \Phi^F$.

Let $T$ be an invertible bundle homomorphism from $E_\gS$ to
$F_\gS$ and consider the boundary condition for $\widetilde A:=A\oplus
(-A^t)$ given by
\begin{equation}
P(T) =
\begin{pmatrix}  -T & \Id
\end{pmatrix}.
\end{equation}
%i.e.,
%\begin{equation}\label{eq3.49}
%\cD(\widetilde A_{P(T)}):=\bigsetdef{\binom{f_+}{f_{-}}\in
%L^2_1(M,E\oplus F)}{\varrho f_-=T\varrho f_+}.
%\end{equation}
Assume furthermore that
\begin{equation}\label{eq3.45}
 J_0^tT \text{ is positive definite, in particular
        selfadjoint. }
\end{equation}
Then

\noi {\rm (a)} $P(T)$ is
strongly regular for $\widetilde{A} := A\oplus (-A^t)$.

\noi {\rm (b)} The dual condition is given by
$P(T^{\text{dual}})=P(-T^{-1})$. It is strongly regular
for $\widetilde A^t$.
%\begin{equation}\label{eq3.47}
%\cD\Bigl(\bigl(\widetilde A_{P(T)}\bigr)^*\Bigr)
%=\bigsetdef{\binom{f_+}{f_{-}}\in L^2_1(M,F\oplus E)}{\varrho
%f_-=-T^{-1}\varrho f_+}.
%\end{equation}

\noi {\rm (c)} The operator $\widetilde{A}_{P(T)}$ is Fredholm with compact
resolvent and
\begin{equation*}
\begin{split}
\ker\widetilde{A}_{P(T)} & =  Z_{+,0}(A) \oplus Z_{-,0}(A),\\
\coker \widetilde{A}_{P(T)} & \simeq Z_{-,0}(A) \oplus Z_{+,0}(A).
\end{split}
\end{equation*}

\noi {\rm (d)} Finally, if $A$ and $A^t$ satisfy weak inner \UCP\ then
$\widetilde{A}_{P(T)}$ is invertible. Moreover in this case the
inverse $\bigl(\widetilde{A}_{P(T)}\bigr)^{-1}$ maps $L^2_s(M,F\oplus E)$
continuously to $L^2_{s+1}(M,E\oplus F)$ for all $s\ge 0$.
\end{theorem}
\begin{proof} We only have to comment on the very last statement.
As in the proof of Proposition \plref{prop3.2} we infer from the
strong regularity that on $\bigsetdef{f\in L^2_k}{P(T)(\varrho f)=0}$
we have estimates
\begin{equation}
    \frac 1C \|f\|_k\le \|f\|_{k-1}+\|\widetilde Af\|_{k-1}\le C\|f\|_k.
\end{equation}
Hence $\widetilde A_{P(T)}^{-1}$ maps $L^2_k$ continuously to
$L^2_{k+1}$, $k\in\Z_+$. Now the claim follows from complex
interpolation.
\end{proof}
%% \commentary{TODO: introduce comment on complex interpolation at
%% its first occurrence. It is somewhere below!
%%
%% In view of the notion of strong regularity the sequel can be
%% improved considerably.}
\begin{remark}\label{r:commment-JLambda}
We emphasize that the condition \eqref{eq3.45} holds for 
\begin{equation}\label{eq:T-choices}
T\in\bigl\{J_0, (J_0^t)\ii, (J_0J_0^t)^{-1/2}J_0\bigr\}\,.
\end{equation}
\end{remark}

\section{Calder\'{o}n projection from the invertible double}\label{s:calderon}
\subsection{Sobolev scale}\label{ss:Sobolev-scale} Next we recall the
purely functional analytic notion of the {\em Sobolev scale} of an
operator (cf. \cite{LioMag:NHBVPAI}, \cite{BruLes:BVP}). For the moment let $D$ be a
closed operator in the Hilbert space $H$. For $s\in\R$ let $H_s(D)$
be the completion of
\begin{equation}
     H_\infty(D):=\bigcap_{s\ge 0} \cD((D^*D)^{s/2})
\end{equation}
with respect to the scalar product
\begin{equation}
    \scalar{x}{y}_s:=\scalar{(\Id+D^*D)^sx}{y}.
\end{equation}
Obviously $H_1(D)=\cD(D)$ and the scalar product $\scalar{.}{.}_0$
extends to a perfect pairing between $H_s(D)$ and $H_{-s}(D)$.
Furthermore, the spaces $H_s(D)$ have the interpolation property, that
is for $s<t$ and $0\le\Theta\le 1$ we have 
\begin{equation}\label{eq3.48}
     H_{\Theta t+(1-\Theta)s}(D)=[H_s(D),H_t(D)]_{\Theta}
\end{equation}
in the sense of complex interpolation theory 
(cf. e.g. \cite[Sec. 4.2]{Tay:PDEI}).\index{complex interpolation theory}
\glossary{$H_s(D)$}  

Note that $D$ induces bounded linear maps $H_s(D)\to H_{s-1}(D^*)$.
We will mostly use the case $|s|\le 1$.

Since $H_{-1}(D)$ is canonically ($\C$--anti--)isomorphic to the dual
of $H_1(D)$ it follows from \eqref{eq3.48} that $H_s(D), |s|\le 1$,
depends only on the spaces $H_0(D)$ and $H_1(D)$; it does not depend on the
particular operator $D$ generating the scale. This independence, of
course, is not true for $|s|>1$.

If the condition \eqref{eq3.45}  ($J_0^tT>0$)
is fulfilled, then in view of
\eqref{eq:dom-inv-double} it is appropriate to put\footnote{Lateron the following considerations will always be used
for the dual boundary condition $T^{\text{dual}}=-T^{-1}$, see \eqref{eq3.28}. 
Therefore we present them already here for
$-T\ii$ instead of $T$.}
\begin{equation}\label{eq3.48.5}
   L^2_{s,-T\ii}(M,F\oplus E):=H_s((\widetilde A_{P(T)})^*),\quad -1\le s\le 1.
\end{equation}

Obviously, for $\half<s\le 1$ we have by \eqref{eq3.28}
\begin{equation}
  L^2_{s,-T\ii}(M,F\oplus E)=\bigsetdef{\binom{f_+}{f_{-}}\in
L^2_s(M,E\oplus F)}{\varrho f_-=-T\ii\varrho f_+}.
\end{equation}
The latter definition makes sense also for $s>1$ but the equality
\eqref{eq3.48.5} is limited to $|s|\le 1$.

We have by construction $L^2_{s,T}\subset L^2_s$, hence
$\varrho$ induces bounded linear maps
\begin{equation}\label{eq:trace-map}
\begin{split}
    L^2_{s,-T\ii}(M,F\oplus E) &\longrightarrow
    L^2_{s-\half}(\Sigma,F_\gS),\\
    \binom{f}{g}&\longmapsto f\rrr\Sigma,
\end{split}
\quad \half<s\le 1.
\end{equation}
Denote by $\varrho^*$ the $L^2$--dual of $\varrho$. I.e.,
$\varrho^*$ is a bounded linear map
\begin{equation}\label{eq:trace-map-dual}
     \varrho^*:L^2_s(\Sigma,F_\gS)\longrightarrow
      L^2_{s-\half,-T\ii}(M,F\oplus E), \quad -\half\le s<0,
\end{equation}
such that for $\xi\in L^2_s(\Sigma,F_\gS)$ and $f\in
L^2_{-s+\half,-T\ii}(M,F\oplus E)$ we have
$\scalar{\varrho^*\xi}{f}=\scalar{\xi}{\varrho f}$.
\index{trace map!$L^2$--dual}

\subsection{Induced Poisson type operators and inverses}\label{ss:continuity-pseudo-inverse}
We use the notation of the previous Section \ref{s:invdoub}.
Throughout the whole section we assume \eqref{eq3.45}
\begin{align}
     &J_0^tT \text{ is positive definite, in particular
        selfadjoint }\label{eq3.45-rep}\\
\intertext{and additionally}
     &[J_0^tT,B_0^t] \text{ is of order } 0.\label{eq:3.45-additional}
\end{align}
\begin{remark}\label{r:comment-commutator}
$[J_0^tT,B_0^t]=0$ for the choice $T=(J_0^t)\ii$.

If $A=A^t$ and $B_0-B_0^t$ is of order $0$ then
$[J_0^tT,B_0^t]$ is of order $0$ for all three choices of $T$
in \eqref{eq:T-choices}.
\end{remark}

Recall from Remark \ref{rem3.10} that condition \eqref{eq3.45-rep}
implies that the dual boundary condition for $\widetilde A^t$ is then given
by $-T^{-1}$.

\glossary{$Z_{\pm,0}(\widetilde A)$}
\glossary{$P_{Z_{\pm,0}(\widetilde A)}$} 

According to Theorem \ref{thm3.14} the boundary condition
$P(T)$ is regular for $\widetilde{A}$ and
\begin{equation}\label{eq3.50}\begin{split}
\ker \widetilde{A}_{P(T)} &  = Z_{+,0}(A) \oplus Z_{-,0}(A) =Z_{+,0}(\widetilde A),\\
\coker\widetilde{A}_{P(T)} & \simeq Z_{-,0}(A) \oplus Z_{+,0}(A) =Z_{-,0}(\widetilde A).
                  \end{split}
\end{equation}
The orthogonal projections onto $Z_{+,0}(\widetilde A), Z_{-,0}(\widetilde A)$ are denoted by
$P_{Z_{+,0}(\widetilde A)}, P_{Z_{-,0}(\widetilde A)}$, respectively.

In a collar of the boundary we expand $A$ in the following form,
omitting the explicit reference to $\Phi^E,\Phi^F$ etc. (cf. \eqref{eq3.6}, \eqref{eq:operator-collar}),
\begin{equation}\label{eq4.15}
   A=J_0\Bigl(\frac{d}{dx} + B_0\Bigr)+C_1 x+C_0
\end{equation}
with a first order differential operator $C_1$
and $x$--independent bundle morphism $C_0\in\ginf{\Sigma;\Hom(E_\gS,F_\gS)}$.
Here and in the sequel, by slight abuse of notation, $x$ will also denote the
operator of multiplication by the function $x\mapsto x$.

To see \eqref{eq4.15} we expand $J$ and $B$ near $x=0$
\begin{equation}
   \begin{split}
         J_x&=J_0+ J_x^{(1)}x=J_0+J_0' x+ J_x^{(2)} x^2\\
         B_x&=B_0+ B_x^{(1)} x.
\end{split}
\end{equation}
Noting that $[\frac{d}{dx},x]=1$ we find
\begin{equation}
   \begin{split}
        J_x\frac{d}{dx}&=J_0\frac{d}{dx}+J_x^{(1)}x\frac{d}{dx}\\
                       &=J_0\frac{d}{dx}+\bigl(J_x^{(1)}\frac{d}{dx}\bigr)x-J_x^{(1)}\\
                       &=J_0\frac{d}{dx}+\bigl(J_x^{(1)}\frac{d}{dx}-J_x^{(2)}\bigr)x-J_0'\\
        J_x B_x        &=J_0 B_0+\bigl(J_0 B_x^{(1)}+J_x^{(1)} B_x\bigr) x,
\end{split}
\end{equation}
thus
\begin{equation}
    \begin{split}
            C_1&=J_x^{(1)}\frac{d}{dx}-J_x^{(2)}+J_0 B_x^{(1)}+J_x^{(1)} B_x,\\
            C_0&=-J_0'.
\end{split}
\end{equation}

The formal adjoint can be written similarly as
\begin{equation}\label{eq:4.15-adjoint}
    A^t=\Bigl(-\frac{d}{dx} + B_0^t\Bigr)J_0^t+\widetilde C_1 x+\widetilde C_0,
\end{equation}
where again $\widetilde C_1$ is a first order differential operator and
$\widetilde C_0$ is an $x$--independent bundle morphism. More precisely,
\begin{equation}
   \begin{split}
       \widetilde C_1&=C_1^t+\frac{1}{x}\bigl([x,C_1^t]-[x,C_1^t]_{\bigl|x=0}\bigr),\\
       \widetilde C_0&=C_0^t+[x,C_1^t]_{\bigl|x=0}=-(J_0')^t+(J_0')^t=0,
     \end{split}
\end{equation}
thus $\widetilde C_0$ in fact vanishes. 
Note that in \eqref{eq4.15}, \eqref{eq:4.15-adjoint} $x$ is intentionally on the right of $C_1$.
From this it also becomes clear that $\widetilde C_0=0$ because in the expansion of 
$A^t=\Bigl(-\frac{d}{dx}+B_x^t\Bigr)J_x^t$ near $x=0$
the commutator $[x,\frac{d}{dx}]$ does not show up.

\begin{remark}\label{r:self-adjoint-leading}
We note that if the tangential operator $B_0$ has a selfadjoint
leading symbol we may replace $B_0$ by $\frac 12(B_0+B_0^t)$
and still write $A, A^t$ in the form \eqref{eq4.15},\eqref{eq:4.15-adjoint}.
This changes, of course, $C_0$ and $\widetilde C_0$.
\end{remark}

We fix a real number $c>0$ as in Section \plref{s:sectorial}, Convention \ref{fix_c_gamma_omega}
and consider the corresponding family of operators $Q_\pm(x)$ of Definition \plref{d:gen-spec-proj}.
Then we define the following operators mapping (distributional) sections of $E\rrr\Sigma$
into (distributional) sections of $E\rrr\R_+\times \Sigma$:
\begin{equation}\label{eq4.20}
    (R\xi)(x):=\begin{pmatrix} Q_+(x) \xi \\
                          Q_-(-x)^*\xi
           \end{pmatrix}, \quad R_T\xi:=\begin{pmatrix} \Id& 0\\ 0 &
    -T \end{pmatrix} R\xi.
\end{equation}
\glossary{$R,R_T$} 

$R_T$ will allow us to study the regularity properties of the (generalized)
inverse of $A_{P(T)}$ (cf.  \eqref{eq4.22} below) and of the Poisson operator
(cf. Definition \plref{def3.18} below). The Poisson operator is a map sending
sections on the boundary into the kernel of $\widetilde A$ in the interior. $R$ and
$R_T$ do almost have this property. For the constant coefficient operator $A=J_0\Bigl(\frac{d}{dx}+B_0\Bigr)$
one has indeed $\widetilde A Q_+=0$ by Proposition \plref{p:properties-Qpm} a,c. 
Even in the constant coefficient case $\widetilde A R$ is not necessarily $0$
but, thanks to \eqref{eq:3.45-additional}, small in a certain sense. This
will become clear below. Note that $R_T$ does \emph{not} map into the
domain of $A_{P(T)}$. Its role will become transparent in formula 
\eqref{eq4.21} below.

\begin{sloppy}
For a cut--off function $\varphi\in\cinfz{\R_+}$ we consider
$\varphi R_T$ as an operator from sections of $E_\Sigma=E\rrr
\Sigma$ to sections of $E\oplus F$ over $M$; note that the range
of $\varphi R_T$ consists of sections vanishing outside a collar of $\gS$.
From Proposition \plref{prop4.1} we infer that $\varphi
R_T$ maps $L^2_s(\Sigma,E_{\gS})$ continuously to
$L^2_{s+\half,\comp}(M,E\oplus F)$, $s\ge -\half$.
\end{sloppy}

To calculate $\widetilde A \varphi R_T$ we proceed by
component:
\begin{equation}\label{eq4.18}
   A\varphi Q_+(x) \xi = \Bigl((C_1 x+C_0)\varphi(x) +J_0
   \varphi'(x)\Bigr)Q_+(x) \xi
\end{equation}
and
\begin{equation}\label{eq4.19}
   \begin{split}
   A^t \varphi T Q_-(-x)^*\xi&=
      \Bigl((\widetilde C_1 xT+\widetilde
      C_0T+[B_0^t,J_0^tT])\varphi(x)-\\
      &\quad  -J_0^t\varphi'(x)\Bigr)Q_-(-x)^* \xi.
\end{split}
\end{equation}
The mapping properties of the right hand sides with respect to
Sobolev spaces can be deduced from Proposition \plref{prop4.1}.

\glossary{$S(A,T)$} 
\begin{dfn}\label{def4.3}
We write $S(A,T)\xi$ for the
differential expression $\widetilde A$ applied to $\varphi R_T\xi$. 
\end{dfn}

\begin{remark}
The \textqm{differential expression}  is emphasized here since
$\varphi R_T$ does \emph{not} map into the domain of $\widetilde A_{P(T)}$.
However, by duality (cf. Subsection \plref{ss:Sobolev-scale})
$\widetilde A_{P(T)}$ may also be viewed as a bounded operator
$L^2(M,E\oplus F)\longrightarrow L^2_{-1,-T\ii}(M,F\oplus E)$.
This should be viewed as applying $\widetilde A$ \emph{in the distributional sense}.

The distinction between $S(A,T)$ and $\widetilde A_{P(T)}$ acting on
$L^2(M,E\oplus F)$ is crucial. The difference between the two
operators is (see \eqref{eq4.21} below) $\varrho^* J_0 (P_++P_-^*)$.

$S(A,T)$ will allow us to control the error (in terms of regularity, not in terms of size)
between the approximate Poisson operator
constructed from $R_T$ and the true Poisson operator.

$S(A,T)$ also depends on the choice of $c$ in Convention \plref{fix_c_gamma_omega}, 
but this will be suppressed in the notation.
$S(A,T)$ is a $2\times 1$ column consisting (up to sign) of
the right hand sides of \eqref{eq4.18} and \eqref{eq4.19}. 
\end{remark}

We single out an immediate but important consequence of
Proposition \plref{prop4.1}:

\begin{prop}\label{p:mapping-SAT}\indent\par\noindent
$S(A,T)$ maps $L^2_s(\Sigma,E_{\gS})$ continuously to
$L^2_{s+\half,\comp}(M,F\oplus E)$, $s\ge -\half$.
\end{prop}

From the mapping properties of $\varphi R_T$ we conclude
in particular that for $\xi\in L^2_s(\Sigma,E_{\gS}), s\ge -\half,$
we have $\varphi R_T\xi\in L^2(M,E\oplus F)=H_0(\widetilde
A_{P(T)})$. Hence we may apply $\widetilde A_{P(T)}\in
\cB(L^2(M,E\oplus F),L^2_{-1,-T\ii}(M,F\oplus E))$ (cf. Subsection
\ref{ss:Sobolev-scale}) to $\varphi R_T$. Recall that this
is defined by duality and the result will be different from the
differential expression $\widetilde A$ applied to $\varphi
R_T\xi$. Indeed for 
$f\in \cD((\widetilde A_{P(T)})^*)=L^2_{1,-T^{-1}}(M,F\oplus E)$ we find using
Green's formula, \eqref{eq3.45-rep} and the boundary condition\index{Green's formula}
$\varrho_-f=-T\ii \varrho_+ f;\, \varrho_\pm f:=\varrho(f_\pm)$:
\begin{equation}
\begin{split}
\scalar{&\varphi R_T\xi}{(\widetilde A_{P(T)})^*f}\\
     &=\scalar{J_0 P_+\xi}{\varrho_+ f}-
            \scalar{J_0^tT P_-^*\xi}{\varrho_- f}+\scalar{S(A,T)\xi}{f}\\
     &=\scalar{J_0 (P_++P_-^*)\xi}{\varrho f}+\scalar{S(A,T)\xi}{f}\\
     &= \scalar{\bigl(\varrho^*J_0(P_++P_-^*)+S(A,T)\bigr)\xi}{f}.
\end{split}
\end{equation}
Here $P_\pm$ denote the positive/negative sectorial spectral projections of $B_0$
in the sense of Definition \plref{d:p-plus}. Recall $P_+(B_0^t)=P_+(B_0)^*$.

Thus as an identity in 
$H_{-1}((\widetilde A_{P(T)})^*)=L^2_{-1,-T^{-1}}(M,F\oplus E)$ we arrive at
\begin{equation}\label{eq4.21}
     \widetilde A_{P(T)}\varphi R_T\xi=
     \bigl(\varrho^*J_0(P_++P_-^*)+S(A,T)\bigr)\xi.
\end{equation}
It is important to note that, even if $\xi$ has better regularity
than $L^2_{-1/2}$, this is just an identity in
$H_{-1}((\widetilde A_{P(T)})^*)$ since $\varphi R_T\xi$ does not fulfill
the boundary condition for $A_{P(T)}$. The boundary condition plays no
role as long as we view $\varphi R_T\xi $ as an element of $L^2$. 

%The reader should be warned that from this equality one should not
%draw false conclusions about the mapping properties of $\varrho^*$.
%$\varphi R_T$ maps $L^2$ to $L^2_\half$ but not to
%$H_{\half}(\widetilde A_{P(T)})$. So we cannot conclude (and it is
%indeed not true in general) that the right hand side of
%\eqref{eq4.21} lies in $H_{-\half}(\widetilde A_{P(T)})$.

With some care we can now
basically proceed as in \cite[Sec. 3.2 and 3.3]{HimKirLes:CPH}:
Firstly we note that by Lemma \plref{l:idempotent-invertible}
$P_++P_-^*$ is invertible.
\index{pseudo--inverse}
\glossary{$\widetilde G$} 
Secondly we introduce the (Hilbert space) pseudo--inverse $\widetilde{G}$ of
$\widetilde{A}_{P(T)}$, namely,
\begin{equation}\label{eq3.52}
\widetilde{G}f := \begin{cases}
\widetilde{A}_{P(T)}^{-1} f, & f\in \im  \widetilde A_{P(T)}, \\
0, & f\in \im \bigl(\widetilde A_{P(T)}\bigr)^\perp,
\end{cases}
\end{equation}
taking account of the possible absence of weak unique continuation.
Here $\widetilde{A}_{P(T)}^{-1} f$ denotes the inverse image in $\bigl(\ker
\widetilde A_{P(T)}\bigr)^\perp$ of $f$ under
$\widetilde{A}_{P(T)}$.

Let\glossary{$U$} 
\begin{equation}\label{eq3.53}
U: \, L^2(M,E\oplus F) \too L^2(M,F\oplus E)
\end{equation}
be the partial isometry which sends $Z_{+,0}(\widetilde A)$ onto $Z_{-,0}(\widetilde A)$ by
interchanging the summands in \eqref{eq3.50} and which is zero on
the orthogonal complement. Then
\begin{equation}\label{eq3.54}
P_{Z_{+,0}(\widetilde A)} =  U^*U, \quad P_{Z_{-,0}(\widetilde A)}  =  UU^*,
\end{equation}
and
\begin{equation}\label{eq3.55}
\widetilde{G}  =  (\Id-U^*U)(\widetilde{A}_{P(T)}+U)^{-1}.
\end{equation}
We recall that $Z_{+,0}(\widetilde A), Z_{-,0}(\widetilde A)$ are finite--dimensional and
consist of sections which are smooth up to the boundary (cf.
Remark \ref{rem3.3} (1)).
Hence $U$ and $U^*$ are smoothing operators. Furthermore, from
$\varrho U = 0, \varrho U^* =0$ we immediately obtain
\begin{equation}\label{eq3.56}
U^*\varrho^*  =  0, \quad U\varrho^*  =  0.
\end{equation}

By construction $\ker \widetilde A_{P(T)}\subset H_\infty(\widetilde A_{P(T)})$. 
Hence $\widetilde A_{P(T)}+U$ induces
invertible bounded linear maps from $H_s(\widetilde A_{P(T)})$
onto $H_{s-1}((\widetilde A_{P(T)})^*)$. Consequently, $\widetilde G$
induces bounded linear maps from
 $H_s((\widetilde A_{P(T)})^*)$ onto  $H_{s+1}(\widetilde A_{P(T)})$.
Together with the mapping properties of $\varrho^*$ we conclude that
$\widetilde{G}\varrho^*$ maps $L_s^2(\Sigma,F_{\gS})$ continuously to
$L_{s+1/2}^2(M,E\oplus F)$ for $-1/2\leq s < 0$.

$G\varrho^*$ is the main building block of the Poisson operator
(Definition \plref{def3.18} below) which should act at least on $L^2$. Therefore
we have to improve the bound on $s$, which is now straightforward:

To apply the pseudo--inverse $\widetilde G$ to \eqref{eq4.21} 
it is enough that \eqref{eq4.21} is an identity in $H_{-1}$. Hence we find
\begin{equation}\label{eq4.22}
     \widetilde G \varrho^*=\Bigl((\Id-P_{Z_{+,0}(\widetilde A)})\varphi R_T - \widetilde G S(A,T)\Bigr) (J_0(P_++P_-^*))^{-1}.
\end{equation}

\begin{theorem}\label{thm3.16} For $-1/2\leq s\le 1/2$
the operator $\widetilde{G}\varrho^*$ maps $L_s^2(\Sigma,F_{\gS})$
continuously to $L_{s+1/2}^2(M,E\oplus F)$.
\end{theorem}
\begin{proof}
  This follows immediately from Proposition \plref{prop4.1},
\eqref{eq4.18}, \eqref{eq4.19}, \eqref{eq4.22}, and
Proposition \plref{p:mapping-SAT}.
\end{proof}

\subsection{The Calder{\'o}n projection}\label{sec5.3}\label{ss:calderon-projector}
From the invertible double the construction of the Calder\'{o}n
projection is straightforward. During the whole subsection we assume
that the conditions \eqref{eq3.45-rep} and \eqref{eq:3.45-additional}
are fulfilled.
\begin{dfn}\label{def3.15}
For a section $f = \binom{f_+}{f_{-}}$ of $E\oplus F$ we recall the
notation $\varrho_{\pm}(f) := \varrho(f_{\pm})$ and $r_{\pm}(f) :=
f_{\pm}$.
Furthermore, we put for sections $f,g$ of $E,F$
\begin{equation}\label{eq3.51}
e_+(f) := \binom{f}{0}, \quad e_{-}(g) := \binom{0}{g}.
\end{equation}
\glossary{$\varrho_{\pm}$}
\glossary{$r_\pm, e_\pm$}  
\end{dfn}
\begin{prop}\label{cor3.17}
$\widetilde{G}\varrho^*$ maps $L_s^2(\Sigma,F_{\gS})$ to
$Z^{s+1/2}_{+}\oplus Z_{-}^{s+1/2}, -1/2\le s \le 1/2$.
\end{prop}
\begin{proof} In view of Theorem \plref{thm3.16} it remains to be shown
that $\widetilde{G}\varrho^*$ maps into the kernel of $\widetilde{A}$.

Let $f\in L_s^2(\Sigma,F_{\gS}), -1/2\le s\le 1/2,$ and a test function $\varphi\in
\ginfz{M\setminus\Sigma;F\oplus E}$ be given. In view of \eqref{eq:trace-map-dual}
we choose a real number $s'$ with $-1/2\le s' <0, s'\le s$. Then, by
\eqref{eq:trace-map-dual}, we have $\varrho^*f\in L^2_{s'-1/2,-T\ii}(M,F\oplus E)$
and thus $\widetilde G \varrho^*f\in H_{s'+1/2}(\widetilde A_{P(T)})$; note $s'+1/2\ge 0$.
Since $\varphi$ has compact support away from $\gS$ we certainly have 
$\varphi\in H_{\infty}\bigl((\widetilde A_{P(T)})^*i\bigr)$.
Hence  $\scalar{\widetilde{A}\widetilde{G}\varrho^*f}{\varphi}=\scalar{\widetilde{G}\varrho^*f}{(\widetilde{A}_{P(T)})^* \varphi}.$
Viewing the rhs as the dual pairing between $H_{s'+1/2}(\widetilde A_{P(T)})$ and
$H_{-s'-1/2}(\widetilde A_{P(T)})$ we may also move $\widetilde G$ to the right and find
$\scalar{\widetilde{A}\widetilde{G}\varrho^*f}{\varphi}=\scalar{\varrho^*f}{\widetilde G^*(\widetilde{A}_{P(T)})^* \varphi}$, 
cf. Subsection \ref{ss:Sobolev-scale}. By construction of the generalized inverse we
have, note again that $\varphi\in H_{\infty}\bigl((\widetilde A_{P(T)})^*\bigr)\subset L^2(M,F\oplus E)$,
\begin{equation}
     \widetilde G^* (\widetilde A_{P(T)})^*\varphi= (I-U^*U)\varphi.
\end{equation}
However, $U^*U\varphi\in H_\infty(\widetilde A_{P(T)})$
and $\varrho U^*U \varphi=0$, cf. \eqref{eq3.56}. Thus
$ \scalar{\varrho^*f}{(\Id-U^*U)\varphi} =   \scalar{f}{\varrho \varphi} = 0.$

This calculation shows that if $\widetilde A$ is applied in the weak sense 
to $\widetilde G \varrho^* f$ one gets $0$. But then $\widetilde G\varrho^* f\in Z^{s+1/2}_{+}\oplus Z_{-}^{s+1/2}$.
\end{proof}

\index{Poisson operator}
\index{Calder{\'o}n projection}
\begin{dfn}\label{def3.18}\indent\par
\begin{thmenum}
\item Define the \emph{Poisson operator} by
$$K_{\pm}  :=  \pm\, r_{\pm}\widetilde{G}\varrho^*J_0.$$
$K_{\pm}$ maps $L_s^2(\Sigma,E_{\gS})$ continuously to
$L_{s+1/2}^2(M,E)$ $(L_{s+1/2}^2(M,F))$ for $-1/2\leq s\leq 1/2$.
\item $$C_{+} := \varrho_+K_+, \quad C_{-}  :=  T^{-1}\varrho_{-}K_{-}.$$
$C_+$ is called the Calder{\'o}n projection of $A$. Recall that
$T$ is the operator defining the boundary condition for
$\widetilde{A}$.
\end{thmenum}
$K_\pm$ and $C_\pm$ depend on the pair $(A,T)$.
\end{dfn}
\glossary{$C_\pm, K_\pm$} 
We summarize the result of the construction before Theorem \plref{thm3.16},
cf. in particular \eqref{eq4.22}:

\begin{prop}\label{prop4.5} Let $\varphi\in\cinfz{\R_+}$ and let
$R_T$ be defined as in \eqref{eq4.20}. Furthermore, let
$P_+:=P_+(B_0), P_-:=P_-(B_0)$ be the positive respectively negative
sectorial spectral projections of $B_0$ as introduced in Definition \plref{d:p-plus}, cf.
Theorem \plref{t:boundedness-positive-sectorial}.

Then the Poisson operators are given by
\begin{equation}\label{eq4.23}
    K_\pm= \pm r_{\pm}\Bigl( (\Id-P_{Z_{+,0}(\widetilde A)})\varphi R_T -\widetilde G S(A,T)\Bigr)(P_++P_-^*)^{-1},
\end{equation}
and the Calder{\'o}n projections are given by (see also \eqref{eq:trace-map} and \eqref{eq:trace-map-dual})
\begin{equation}\label{eq4.24}
   \begin{split}
      C_+&= \Bigl(P_+ - \varrho_+ \widetilde G S(A,T)\Bigr)(P_++P_-^*)\ii\\
      C_-&= \Bigl(P_- +T^{-1}\widetilde G S(A,T)\Bigr)(P_++P_-^*)\ii.
   \end{split}
\end{equation}
\end{prop}
\begin{proof} The Theorem follows immediately from \eqref{eq4.22}. 
%For the formula
%\eqref{eq4.24} one should note that $P_+(P_++P_-^*)=P_+$ (since $\im P_-^*=\ker P_-^\perp=\im P_+^\perp$)
%and thus $P_+(P_++P_-^*)\ii=P_+$.
\end{proof}

\begin{remark}\label{r:formula-calderon-poisson}
Note that in the formula for $K_+$  in \eqref{eq4.23}
$R_T$ can be replaced by $R$ (cf. \eqref{eq4.20}), hence the first summand
for $K_+$ is independent of $T$.

Note that \textqm{our} Calder{\'o}n projection differs from $P_{+,\ort}=P_+(P_++P_-^*)\ii$
(cf. Lemma \plref{l:idempotent-invertible}, \eqref{e:p-ort})
by an operator which regularizes by at least one Sobolev order.
So our construction of the invertible double yields naturally a version of the
\emph{orthogonalized} Calder{\'o}n projection; if $T=(J_0^t)^{-1}$ then $C_+$
is indeed an orthogonal projection, see the next Proposition.
\index{orthogonalization}
\index{Calder{\'o}n projection!Seeley's classical}

This is in contrast to the \emph{classical} Calder{\'o}n projection $\cP_+$ of Seeley \cite{See:TPO} which
is a pseudodifferential operator of order $0$ whose leading symbol coincides with
that of $P_+$. Hence our $C_+$ differs from the orthogonalized Calder{\'o}n projection
$\cP_{+,\ort}$ by an operator which regularizes by at least one Sobolev order.
With some more work one can indeed show that $C_+$ is a pseudodifferential operator
which differs from $\cP_{+,\ort}$ by an operator of order $-1$.

If $B_0=B_0^t$, or more generally if $B_0$ has a selfadjoint leading symbol,
then Proposition \plref{prop4.5} shows in particular the well--known fact (cf. %\cite[Eq. (3.7)]{Gru:TEP}) 
\cite[Corollary 14.3]{BooWoj:EBP}) 
that $C_+-P_+(B_0)$ is an operator of order $-1$.

Our approach also reproves a stronger result in the product situation:
namely assume that $A=J_0(\frac{d}{dx}+B_0)$ in a collar of the boundary with $B_0=B_0^t$.
Using the following modified version\footnote{The whole discussion after \eqref{eq4.20} can be
carried out with this modified $R'$, too.} of $R$
\begin{equation}\label{eq4.20-modified}
    (R'\xi)(x):=\begin{pmatrix} Q_+(x) \xi \\
                          (J_0^tT)^{-1}Q_-(-x)^*(J_0^tT)\xi
           \end{pmatrix}, \quad R_T'\xi:=\begin{pmatrix} \Id& 0\\ 0 &
    -T \end{pmatrix} R'\xi,
\end{equation}
one then has $\widetilde AR_T\xi=0$ and hence $S(A,T)\xi$
is supported in $\supp \varphi'$ away from the boundary. Then it is not difficult to see
that $S(A,T)$ is smoothing (cf. \cite[Prop 3.15 and Eq. (3.38)]{HimKirLes:CPH})
and thus $C_+-P_+(B_0)$ is a smoothing operator.

This result was proved by G. Grubb \cite[Prop. 4.1]{Gru:TEP}. 
Before it was shown by S. Scott \cite[Prop. 2.2]{Sco:DDB}
for selfadjoint Dirac operators on spin manifolds in the case
$\ker B_0=0$. 

In general, $C_+$ and $P_+(B_0)$ belong to different connected components  of the Grassmannian of pseudodifferential projections with fixed leading symbol even for symmetric $B_0$, see \cite[Remark 22.25]{BooWoj:EBP}. For operators of Dirac type, however, $C_+$ can be continuously deformed into a finite range perturbation of $P_+(B_0)$ in the $L^2$ and
the $L^2_{1/2}$ operator topology, see \cite[Corollary C.3]{Nic:GSG}.
\end{remark}

\begin{prop}\label{prop3.19}
$C_{\pm}$ are idempotents with $C_+ + C_{-} = \Id$ and
\begin{equation*}
\begin{split}
 C_+(L^2_s)  &=  N_+^s,\\
 C_-(L^2_s)  & = T^{-1} N_-^s,
\end{split}
               \quad  -1/2\leq s\leq 1/2.
\end{equation*}
Furthermore if $T := (J_0^t)^{-1}$ then $C_{\pm}^* =
C_{\pm}$, i.e., $C_{\pm}$ act as orthogonal projections on $L^2$. In
that case $C_{\pm}\rrr L^2_0$ are $L^2$ extensions of
pseudodifferential projections.
\end{prop}

\begin{remark}\label{rem3.20}
In view of the previous Proposition and Theorem \ref{thm3.14} we can
always construct $\widetilde{A}_{P(T)}$ in such a way that $C_{\pm}$
are orthogonal projections (as mentioned, by choosing $T :=
(J_0^t)^{-1}$). However, even if $A$ is symmetric it may happen
that $\widetilde{A}_{P(T)}$ is not selfadjoint.

If $A=A^t$ and if  $T = J_0(-J_0^2)^{1/2}$ satisfies \eqref{eq:3.45-additional}
we may construct a selfadjoint extension of $\widetilde{A}$
at the cost of a non-orthogonal Calder{\'o}n projection.

Only if $J_0^2 = -\Id$ and $T = (J_0^t)^{-1}$ then
$\widetilde{A}_{P(T)}$ is selfadjoint {\bf{and}} $C_{\pm}$ are
orthogonal projections.
\end{remark}
\begin{proof}[Proof of Proposition \ref{prop3.19}]
We already know from Proposition \ref{cor3.17} that
\begin{equation*}
\begin{split}
C_+(L_s^2)  & \subset  N_+^s,\\
C_-(L_s^2)  & \subset  T^{-1}N_-^s.
\end{split}
\end{equation*}
We show
\begin{enumerate}
\item[(i)] $N_+^s \cap T^{-1}N_-^s  =  \{0\}$,
\item[(ii)] $C_+ \, +\, C_-  =  \Id$.
\end{enumerate}
This easily implies the first claim.

\enum{i} Let $\xi\in N_+^s \cap T^{-1}N_-^s$. Then there are
$f\in Z_+^{s+1/2},  g\in Z_-^{s+1/2}$ with $\varrho f = \xi =
T^{-1}\varrho g$. Then
\begin{equation}
\binom{f}{g}\in\ker\widetilde{A}_{P(T)} = Z_{+,0}(\widetilde A).
\end{equation}
Since elements of $Z_{+,0}(\widetilde A)$ vanish on the boundary we infer $\xi =
0$.

\enum{ii} Let $\xi\in L^2_s(\Sigma,E_{\gS}), -1/2\leq s\leq 1/2$,
and $f\in \cD\bigl((\widetilde{A}_{P(T)})^*\bigr)$.  Then $\varrho_+ f =-T\varrho_- f$
(cf. Remark \plref{rem3.10}) and exploiting the selfadjointness of $J_0^tT$
we obtain
\begin{equation}
\begin{split}
    &\scalar{(C_++C_-)\xi}{J_0^t\varrho_+f}\\
  = &\scalar{\varrho_+\widetilde{G}\varrho^* J_0\xi -
       T^{-1}\varrho_-\widetilde{G}\varrho^* J_0\xi}{J_0^t\varrho_+ f}\\
  = &\scalar{\varrho_+\widetilde{G}\varrho^*J_0\xi}{J_0^t\varrho_+f}
      - \scalar{\varrho_-\widetilde{G}\varrho^*J_0\xi}{J_0J_0^{-1}(T^{-1})^tJ_0^t\varrho_+f}\\
  = &\scalar{(\varrho_+\oplus\varrho_-)\widetilde{G}\varrho^*J_0\xi}%
       {(J_0^t\oplus J_0)(\varrho_+f\oplus\varrho_-f)}\\
  = &\scalar{\widetilde{G}\varrho^*J_0\xi}{(\widetilde{A}_{P(T)})^*f}\\
  = &\scalar{\varrho^*J_0\xi}{f}=\scalar{\xi}{J_0^t\varrho f}.
\end{split}
\end{equation}
This proves (ii).

Finally, let $T = (J_0^t)^{-1}$ and pick $\xi\in N_+^0,
\eta\in T^{-1}N_-^0$. Choose $f\in Z_+^{1/2}$ with
$\varrho_+f=\xi$ and $g\in Z_-^{1/2}$ with $T^{-1}\varrho_+g =
\eta$. Then Green's formula Lemma \plref{l:Green-formula} gives\index{Green's formula}
\begin{equation}
\scalar{\xi}{\eta}  =  \scalar{J_0\varrho_+f}{\varrho_+g} =  -\scalar{Af}{g} + \scalar{f}{A^tg} =   0.
\end{equation}
Hence $N_+^0 \perp T^{-1} N_-^0$ and we are done.

To prove the pseudodifferential property we recall from
\cite[Appendix]{See:TPO} Seeley's construction of the Calder{\'o}n
projection which always yields a pseudodifferential projection
$\cP_+^s$ onto $N_+^s$. By orthogonalization, cf. Lemma \plref{l:idempotent-invertible},
we obtain an orthogonal pseudodifferential projection onto $N_+^s$
which must coincide with $C_+$ for $s=0$.\index{orthogonalization}
\end{proof}

%%%%%%%%%%%%%%%%%%%%%%%%%%%%%%%%%%%%%%%%%%%%%%%%%%%%%%%%%%%%%%%%%%%%%%%%%%%%%%%%%%%%%%
%\commentary{Identification of our new $C_+^s$ with Seeley's $\cP_+^s$
%can possibly be extended to $s<0$ and perhaps also to $s>0$!?.}
%%%%%%%%%%%%%%%%%%%%%%%%%%%%%%%%%%%%%%%%%%%%%%%%%%%%%%%%%%%%%%%%%%%%%%%%%%%%%%%%%%%%%
%
%
%
%

\section{The General Cobordism Theorem}\label{s:cobord}

In the preceding sections,
we gave a new definition of the Calder{\'o}n projection.
We achieved a canonical construction, free of extensions and other choices, and in greatest generality.
Our main goal with the new definition was a construction
which admits to follow precisely a
continuous variation of the coefficients of an elliptic differential operator up to the induced
variation of the Calder{\'o}n projection. We shall return to this application
below in Section \ref{s:parameter}.

An added bonus of our construction of the Calder{\'o}n projection is that it leads
immediately, and somewhat surprisingly, to a simple proof and a
wide generalization of the classical Cobordism Theorem. We shall now give five different formulations
of the Cobordism Theorem and show that the first claim (I), expressed in the language of
symplectic functional analysis, follows immediately from our
 construction of the Calder{\'o}n projection, and that the four other
definitely non--trivial claims (II)-(V) are easily derived from
the first claim. Put differently, we shall show that the Cobordism
Theorem and its various generalizations and reformulations are an
almost immediate consequence of our construction of the Calder{\'o}n
projection.

In all of this section we shall assume that our first order elliptic
operator $A$ over the smooth compact manifold $M$ with boundary $\gS$
is formally selfadjoint, i.e., $A=A^t$\,.
For convenience we shall write $B_0:={B(0)}$ and $J_0:={J(0)}$ as in the
previous sections.

\subsection{The General Cobordism Theorem}
The main result of this section is

\index{Cobordism Theorem}
\begin{theorem}[The General Cobordism Theorem]\label{t:cobord}
Let $A:\ginf{M;E}\to \ginf{M;E}$ be a first order formally
selfadjoint elliptic differential operator on a smooth compact
manifold $M$ with boundary acting between sections of the vector bundle
$E$. We assume that \eqref{eq:3.45-additional} is satisfied\footnote{As noted in 
Remark \plref{r:comment-commutator} this is the case
if $B_0-B_0^t$ is of order $0$.} 
by $T=J_0 (-J_0^2)^{-1/2}$

Then we have the following results:

\enum{I} Let $C_{\pm}$ denote the Calder{\'o}n projections introduced in
Definition \plref{def3.18}, constructed from the invertible double
with ${T}\in
\bigl\{(J_0^t)\ii,{J_0},{J_0}(-{J_0}^2)^{-1/2}\bigr\}$. Then the range of $C_+$ is a
Lagrangian subspace of the strongly symplectic Hilbert space
$\bigl(L^2(\gS,E_\gS), -{J_0}\bigr)$.
Note that $\im C_+$ is independent of ${T}$.
Moreover, $\im C_-$ is also Lagrangian, if ${T} := {J_0}(-{J_0}^2)^{-1/2}$.
%% \marginpar{Too general? Problems with $\im C_-$}

\enum{II} We have $\sign i P_0 {J_0}\rrr W_0 =0$. Here $W_0$ denotes
the (finite--dimensional) sum of the generalized eigenspaces of
${B_0}$ corresponding to purely imaginary eigenvalues and $P_0$ denotes the orthogonal
projection onto $W_0$; in general ${J_0}$ will not map $W_0$ to
itself.

If ${B_0}={B_0}^t$\,, then ${J_0}$ anticommutes with ${B_0}$ and we have
$\sign i{J_0}\rrr\ker {B_0} =0$.

\enum{III} Under the same additional assumption, i.e., for
${B_0}={B_0}^t$\,, the tangential operator ${B_0}$ is odd with
respect to the grading given by the unitary operator $\ga:=i
{J_0}(-{J_0}^2)^{-1/2}$ and hence splits into matrix form
${B_0}=\begin{pmatrix} 0& B^-\\  B^+& 0\end{pmatrix}$ with respect
to the $\pm 1$--eigenspaces of $\ga$. The index of $B^+:\ker
(\ga-1)\longrightarrow \ker(\ga+1)$ vanishes.

\enum{IV} While we do not know whether $C_+$ is 
a pseudodifferential operator for $T=J_0(J_0^2)^{-1/2}$, we can prove
the following:

There exists a pseudodifferential projection $P$ over
$\gS$ such that $\ker P$ is a Lagrangian subspace of the
strongly symplectic Hilbert space $\bigl(L^2(\gS,E_\gS)$,
$-{J_0}\bigr)$, and $(\ker P,\im C_+)$ is a Fredholm pair of closed
subspaces of $L^2(\gS,E_{\gS})$.

\enum{V} There exists a selfadjoint pseudodifferential Fredholm
extension $A_P$.
\end{theorem}

In the following, we shall first give our view of elements of
symplectic functional analysis. Based on that, we shall prove the
preceding theorem as indicated above, i.e., we shall prove {\bf (I)}
directly from our new construction of the Calder{\'o}n projection, then
the implications $\text{\bf (I)} \implies \text{\bf
(II)}\implies\text{\bf(III)}$, and then $\text{\bf (II)}\implies
\text{\bf (IV),(V)}$.

To us, our order of proving Theorem \ref{t:cobord} is the most
simple and natural, beginning with and footing on claim {\bf (I)}.
However, at the end of this section we shall explain that one can
reverse the order of the proof. In particular, we shall show that {\bf
(V)} was essentially proved by Ralston in 1970 in \cite{Ral:DISO}
and that {\bf (I)} can be derived from {\bf (V)} independently.

\begin{remark}
\enum{1} An alternative reading of {\bf (III)} is the following: the
index of {\it any} elliptic first order differential operator $C$
with smooth coefficients over a smooth closed manifold $\gS$ must
vanish, if the block operator $B:= \begin{pmatrix} 0& C^t\\  C&
0\end{pmatrix}$ can be written as the tangential operator of an
elliptic formally selfadjoint operator $A$ on a smooth compact
manifold $M$ with $\dpa M=\gS$\,. That is a new generalization of
the illustrious {\em Cobordism Theorem} for Dirac operators on
spin manifolds which played a decisive role
for the first proof of the Atiyah-Singer Index Theorem (1963).
Since then it was
slightly generalized and an impressive variety of different proofs
were given. Our point here is to show that the Cobordism Theorem has
nothing to do with the specific form of Dirac type operators but
generalizes to all elliptic formally selfadjoint differential
operators of first order.

\enum{2} To some extent, our approach is motivated by
\cite[Section 1.C]{BruLes:BVP}, and is
similar to C. Frey's
\cite[Section 3.4]{Fre:NLB}. The main difference between our
approach and Frey's is that we reduce the arguments a bit more to
purely algebraic reasoning. That permits us to get through also in
the general case, {\em not} assuming Dirac type, i.e., not assuming
constant coefficients near the boundary in normal direction - as
Frey does and all the other literature on cobordism invariance.

\enum{3} From a geometric point of view, the assumption of constant
coefficients may appear sufficiently general for many applications,
justified for index problems by $K$--theory and homotopy invariance of the
index, and sufficiently challenging for constantly attracting ever new and
more simple and more ingenious proofs of the cobordism invariance of
the index over the last 50 years. For a recent example and summary
of the highlights we refer to Braverman \cite{Bra:NPC}. However,
from an analysis point of view, it is not natural to assume constant
coefficients near the boundary. Moreover, it seems more than timely
that the pilot work by Ralston, admitting general coefficients and
so showing a way for our General Cobordism Theorem, is taken
into account in the analysis and topology literature.

\end{remark}

\subsection{Elements of symplectic functional analysis}

Let us briefly summarize the basic set--up of symplectic functional analysis
(see, e.g., \cite{BooFur:MIF}, \cite{BooZhu:WSF} or \cite{KirLes:EIM}).

\subsubsection{Basic definitions}
\index{symplectic!weakly}
Let $H$ be a real or complex Hilbert space. A weakly symplectic form
on $H$ is a non--degenerate Hermitian sesqui--linear form $\go$ on
$H$. Sesqui--linear means, following the tradition in mathematical
physics, $\go(\gl x,y)=\ovl{\gl}\go(x,y),\go(x,\gl
y)={\gl}\go(x,y),$ and Hermitian means $\go(y,x) =
-\overline{\go(x,y)}$ for $x,y\in H$. Finally, non--degeneracy means
that the map $H\ni x\mapsto \go(x,\cdot)\in H^*$ is an injective
continuous linear map; it then follows that the range of this map is
dense.

\index{symplectic!Hilbert space}
The pair $(H,\go)$ is called a (complex) \emph{weakly symplectic
Hilbert space}. Since $\go$ is continuous there is a unique
skew--symmetric injective map $\gamma\in\cB(H)$ such that
\begin{equation}
   \go(x,y)=\scalar{\gamma x}{y},\quad x,y\in H.
 \end{equation}
Here as before, $\cB(H)$ denotes the space of bounded endomorphisms of $H$.
\glossary{$\cB(H)$} 

For a subspace $\gl\< H$ we have
\begin{equation}
    \gl^\go:=\bigsetdef{x\in H}{\go(x,y)=0 \text{ for } y\in\gl}=\bigl(\gamma \gl\bigr)^\perp,
\end{equation}
hence $\gl$ is Lagrangian (i.e. $\gl^\go=\gl$) if and only if $\gl=\bigl(\gamma \gl\bigr)^\perp$.
\index{Lagrangian}
\glossary{$\go$}
\glossary{$\lambda^\go$}  

\index{symplectic!strongly}
The pair $(H,\go)$ is called \emph{strongly} symplectic, if the
injective operator with dense range $H\ni x\mapsto \go(x,\cdot)\in
H^*$ is in fact surjective and hence has a bounded inverse.
Equivalently, the skew--symmetric operator $\gamma$ which implements
$\go$ is invertible.

In a strongly symplectic Hilbert space we may choose an equivalent scalar product
$\scalar{\cdot}{\cdot}_\gamma$ such that
the operator which implements $\go$ with respect to $\scalar{\cdot}{\cdot}_\gamma$
is the \emph{unitary reflection} $\gamma(-\gamma^2)^{-1/2}$: namely
\begin{equation}
  \label{eq:strong-symplectic-hilbert}
  \go(x,y)=\scalar{\gamma x}{y}=\scalar{(-\gamma^2)^{1/2} \gamma (-\gamma^2)^{-1/2} x}{y}
          =\scalar{\gamma (-\gamma^2)^{-1/2} x}{y}_\gamma
\end{equation}
with $\scalar{\xi}{\eta}_\gamma:=\scalar{(-\gamma^2)^{1/2}\xi}{\eta}$.

The scalar product $\scalar{\cdot}{\cdot}_\gamma$ is equivalent to $\scalar{\cdot}{\cdot}$ in the sense
that there is a constant $C$ such that
\begin{equation}
     C^{-1} \scalar{x}{x} \le  \scalar{x}{x}_\gamma\le     C \scalar{x}{x},\quad x\in H.
\end{equation}
In view of \eqref{eq:strong-symplectic-hilbert} the operator which implements $\go$
with respect to $\scalar{\cdot}{\cdot}_\gamma$ is the unitary reflection $\gamma (-\gamma^2)^{-1/2}$.

Finally, we comment on isomorphisms: let
$R:(H_1,\go_1)\longrightarrow (H_2,\go_2)$ be an invertible bounded
linear map which is symplectic but \emph{not} necessarily isometric.
If $\go_j(\cdot,\cdot)=\scalar{\gamma_j\cdot}{\cdot}, j=1,2$ with
skew--symmetric $\gamma_j$ then $\gamma_1=R^*\gamma_2 R$, in
particular $(H_1,\go_1)$ is strongly symplectic if and only if
$(H_2,\go_2)$ is.

\begin{remark}\label{r:complex-symplectic}
For simplicity, we assume from now on that all vector spaces, Hilbert spaces,
and symplectic spaces are complex. Note, however, that our definition
of symplectic Hilbert space does not require the existence of a Lagrangian
subspace (as opposed to e.g. \cite[Sec. 6]{KirLes:EIM}). Indeed,
it might be that there are no Lagrangian subspaces at all. Take, e.g., $H:=\C$
and $\go(x,y):= i\overline{x}y$; see also \cite[Rem. 6.13]{KirLes:EIM}.
\end{remark}

\subsubsection{Algebraic observations}
Dealing with elliptic problems on manifolds with boundary naturally leads to symplectic
Hilbert spaces, see below Paragraph \ref{sss:von-Neumann}.
However, keeping the arguments on a purely algebraic level
where possible may make proofs more transparent. Moreover, Sobolev chains
of symplectic Hilbert spaces are equipped with a variety of
non--equivalent norms but compatible symplectic forms.
For such applications it is nice to establish
results independent of topological choices.

So, let $(H,\go)$ be a symplectic vector space; i.e., no boundedness
of the symplectic form is assumed.

First we note that the quadratic form $x\mapsto i\go(x,x)$ has well--defined signature, if
$H$ is finite--dimensional. In that case we have
\begin{equation}\label{e:vanishing-signature}
\sign i\go = 0 \iff \exists \gl \subset H \text{ Lagrangian
subspace},
\end{equation}
cf. Remark \ref{r:Lagrangian-existence}.1.

Next, we recall
a simple algebraic observation, taken from \cite[Lemma
1.2]{BooZhu:WSF}. Here, and in the proposition further below, the
point is to establish the Lagrangian property for isotropic\index{isotropic}
subspaces via a {\em purely algebraic Fredholm pair property},
i.e., finite--dimensional intersection and finite codimension of sum;
no closedness is assumed.
The algebraic Fredholm pair property
can be considered as the natural generalization of the well-known
definition of Lagrangian subspaces in finite dimension as isotropic
subspaces of {\em maximal dimension}.

\begin{lemma}\label{l:isotropic-sum-to-lagrangian}
Let $(H,\go)$ be a symplectic vector space with transversal
subspaces $\gl,\mu$\,, i.e., $\lambda+\mu=H, \lambda\cap\mu=0$. If
$\gl,\mu$ are isotropic subspaces, then they are Lagrangian
subspaces.
\end{lemma}

\begin{proof}
From linear algebra we have
\[
\gl^{\go}\cap \mu^{\go} =(\gl+\mu)^{\go} =\{0\},
\]
since $\gl+\mu=H$. From
\begin{equation}\label{e:isotropics}
\gl\subset\gl^{\go}, \mu\subset\mu^{\go}
\end{equation}
we get
\begin{equation}\label{e:new-splitting}
H= \gl^{\go} \oplus \mu^{\go}\,.
\end{equation}
Since $H=\lambda\oplus \mu$ we conclude from \eqref{e:isotropics}
and \eqref{e:new-splitting} that $\lambda=\lambda^\go$ and
$\mu=\mu^\go$.
\end{proof}

The preceding result can be generalized:

\begin{prop}\label{p:isotropic_fp_to_lagrangian}
Let $(H,\go)$ be a symplectic vector space with isotropic subspaces
$\gl,\mu$\,. If $(\gl,\mu)$ forms an algebraic Fredholm pair
with $\ind(\gl,\mu)\geq 0$, then $\gl$
and $\mu$ are Lagrangian subspaces of $H$ and we have
\[
\ind(\gl,\mu)= 0, \ (\gl+\mu)^{\go}= \gl\,\cap\,\mu\,,
\tand (\gl+\mu)^{\go\go}= \gl\,+\,\mu\,.
\]
\end{prop}

For the proof we refer to \cite[Proposition 1.13a]{BooZhu:WSF}.

\subsubsection{Symplectic reduction}
We recall a lemma on
symplectic reduction from \cite[Prop 6.12]{KirLes:EIM} (see also
\cite[Proposition 2.2]{BooZhu:SR} for a generalization to weakly symplectic
Hilbert spaces):

\begin{lemma}\label{l:symp-red-ml}
Let $(H,\go)$ be a strongly symplectic Hilbert space, $\gl\subset H$
a Lagrangian subspace and $W\subset H$ a closed co-isotropic subspace.
Assume that $\gl+W^\go$ is closed.
Then the form
\[\widetilde\go(x+W^\go,y+W^\go):=\go(x,y),\quad x,y\in W\]
is a strongly symplectic form on $W/W^\go$. Moreover, the
symplectic reduction of $\gl$ by $W$
\begin{equation}\label{e:reduction}
\Red_{W}(\gl):=\bigl((\gl+W^\go)\cap W\bigr)/W^\go \subset  W/W^\go
\end{equation}
is a Lagrangian subspace of $W/W^\go$\,.
\end{lemma}

\begin{remark}\label{r:Lagrangian-existence}
\enum{1} Without the assumption $\gl+W^\go$ closed $\widetilde \go$
will still be a non--degenerate sesqui-linear form on the quotient $W/W^\go$.
However, in that case it might be that there are no Lagrangian subspaces
at all as pointed out in Remark \ref{r:complex-symplectic}.
If the form $\widetilde \go$ is written as
$\widetilde\go(\xi,\eta)=\scalar{\widetilde \gamma\xi}{\eta}$ with a suitable scalar product
and a skew--symmetric unitary operator $\widetilde \gamma$,
then the existence of a Lagrangian subspace
is equivalent to the fact that the $\pm i$ eigenspaces of $\widetilde \gamma$ have
the same dimension, corresponding to \eqref{e:vanishing-signature}.

\enum{2}
In \cite[Prop. 6.12]{KirLes:EIM} the Lemma was formulated under seemingly
more restrictive assumptions. Let us give an equivalent formulation of
the Lemma which clarifies the link to loc. cit. and which will be useful
below:
\end{remark}

\begin{lemma}\label{l:symp-red-ml-subspace}
Let $(H,\go)$, $\gl\subset H$ and $W\subset H$ be as in Lemma \plref{l:symp-red-ml}.
Let $\gamma$ be the invertible skew--symmetric operator in $H$ such that
$\go(\cdot,\cdot)=\scalar{\gamma\cdot}{\cdot}$.

Let $W_0\subset W$ be a closed subspace such that $W=W_0\oplus W^\go$ (the sum is not
necessarily orthogonal). Furthermore, let $P_0$ denote the orthogonal projection onto $W_0$
and let $Q_0$ denote the projection \emph{along} $W^\go$ onto $W_0$.

Then $\go\rrr W_0=\scalar{P_0\gamma\cdot}{\cdot}$ is a \emph{strongly} symplectic form
on $W_0$, $Q_0(\lambda\cap W)$ is a Lagrangian subspace of $(W_0,\go)$ and
the quotient map $\pi:W_0\longrightarrow W/W^\go$ is a bounded invertible symplectic
linear operator.
\end{lemma}

\begin{proof} Let us briefly sketch the proof of this and of the previous lemma.
First as remarked at the beginning of this subsection we may choose a scalar product
such that the corresponding $\gamma$ is unitary. Consider first $W_0=(W^\go)^\perp=
W\cap \gamma W$. \cite[Prop. 6.12]{KirLes:EIM} and its proof show that $(W_0,\go)$
is symplectic and that $Q_0=P_0$ maps $\gl\cap W$ onto a Lagrangian subspace of $W_0$.
Furthermore, since $W_0=(W^\go)^\perp$ in this case, the quotient map $\pi\rrr W_0:(W_0,\go)\to
(W/W^\go,\widetilde \go)$ is a unitary symplectic isomorphism. This proves that $(W/W^\go,\widetilde \go)$
is indeed a strongly symplectic Hilbert space and that $\Red_W(\gl)$ is a Lagrangian
subspace. Hence Lemma \plref{l:symp-red-ml} is proved.

To prove Lemma \plref{l:symp-red-ml-subspace} for an arbitrary
closed subspace $\widetilde W_0\subset W$ with $W=\widetilde W_0\oplus W^\go$
we only have to note that we have the following commutative diagram
\begin{equation}
\xymatrix{\widetilde W_0\hspace*{1.5em} \ar[dr]^{\widetilde\pi}\ar@<-1ex>[d]^{P_{W_0}} & \\
            W_0 \hspace*{1.5em}         \ar@<3ex>[u]^{Q_{\widetilde W_0}}\ar[r]^-\pi  & W/W^\omega}
\end{equation}
where $\pi,\widetilde \pi$ denote the quotient map $W\longrightarrow W/W^\go$ restricted
to $W_0,\widetilde W_0$ respectively, $P_{W_0}$ denotes the orthogonal projection onto $W_0$ and
$Q_{\widetilde W_0}$ denotes the projection along $W^\go$ onto $\widetilde W_0$. $\pi,\widetilde \pi$ are
symplectic bounded invertible maps. From this all remaining claims follow.
\end{proof}

\subsubsection{The von-Neumann quotient of all natural boundary values}\label{sss:von-Neumann}
%
%% \marginpar{I think our notation for domains is $\cD$}
%
We recall the basic findings about selfadjoint extensions and the relation
between Fredholm Lagrangian pairs in the von-Neumann quotient
$\cD(A_{\max})/\cD(A_{\min})=: \gb(A)$ and in
$L^2(\gS,E_\gS)$:

Let $A_m$ be a closed symmetric operator with domain $\cD_m$ in a
Hilbert space $H$. Following von Neumann and the Russian tradition
of M. Krein, Vishik, and Birman, the operator $A_m$
defines a (strongly) symplectic Hilbert\glossary{$\gb(A_m)$}
space $\gb(A_m):=\cD_{\max}/\cD_m$ of {\em natural boundary values}. Here
$\cD_{\max}$ denotes the domain of $A_m^*$\/. The Hilbert space structure on
$\gb(A_m)$ is given by the graph scalar product, and the symplectic
form is given by Green's form
\begin{equation}\label{e:smplectic-beta}
\go(x+f,y+g):= \lla Ax,y\rra - \lla x, Ay\rra \quad\text{for
$x,y\in\cD_{\max}$},
\end{equation}
independent of the choice of $f,g\in \cD_m$ (see our Equation
\eqref{eq2.10} and \cite[Section 3]{BooFur:MIF}). It is well known
that there is a one-to-one correspondence between
\begin{itemize}
\item domains $\cD_m\< \cD\subset \cD_{\max}$ which yield a
selfadjoint operator $A_{\cD}:=A_m^*\rrr \cD$
\item and the Lagrangian subspaces $\gl$ of $\gb(A_m)$
\end{itemize}
by
\[
\cD\mapsto \gl:=\cD/\cD_m \quad \tand\quad \gl\mapsto \cD:=\{x\in H\mid
x+\cD_m\in\gl\}.
\]

In our situation, we set $H:=L^2(M,E)$ and consider $A_m= A_{\min}
=A\rrr\cD(A_{\min})$ with $\cD_m=\cD(A_{\min})=L^2_{1,0}(M,E)$,
the closure of $\ginfz{M\setminus\Sigma;E}$
in $L_1^2(M,E)$ like in Remark \ref{rem3.3}.  Then
$\gb(A_m)$ naturally becomes a subspace of $L^2_{-1/2}(\gS,E_\gS)$.
From now on we will use this identification and view $\gb(A_m)$
as a subspace of $L^2_{-1/2}(\Sigma,E_\Sigma)$.

\smallskip

To discuss selfadjoint extensions by boundary conditions given by
pseudodifferential projections, it is helpful to consider two other
symplectic spaces: the strongly symplectic Hilbert space\index{symplectic!Hilbert space}
$L^2(\gS,E_{\gS})$ with symplectic form induced by $-{J_0}$ and the
weakly symplectic Hilbert space $L^2_1(M,E)/L^2_{1,0}(M,E)$ with
symplectic form induced by Green's form, or, equivalently, by 
$-\widetilde J_0$, as well. Here we identify the quotient with the subspace
$L^2_{1/2}(\gS,E_{\gS}) \< L^2(\gS,E_{\gS})$ of sections over the
boundary, but with different scalar product
$\lla\cdot,\cdot\rra_{L^2_{1/2}(\gS,E_{\gS})}$, hence $\widetilde
{J_0}=(\Id+|B|)\ii {J_0}$. Note that $\widetilde {J_0}$ is not invertible for
$\dim M >1$, as explained in \cite[Remark 1.6b]{BooZhu:WSF}.

The relations between the Lagrangian subspaces of these three
different symplectic spaces are somewhat delicate because neither
$L^2(\gS,E_{\gS})\< \gb(A_m)$ nor $L^2(\gS,E_{\gS})\supset
\gb(A_m)$. We may, however, recall a very general result from
\cite[Theorem 1.2a]{BooFurOts:CCR}:

Let $\gb$ and $L$ be strongly symplectic Hilbert spaces with symplectic
forms $\go_\gb$ and $\go_L$\,, respectively. Let
\begin{equation}\label{e:dsd}
\gb = \gb_- \oplus\gb_+ \qquad \tand
\qquad L=L_-\oplus L_+
\end{equation}
be direct sum decompositions by transversal (not necessarily
orthogonal) pairs of Lagrangian subspaces. We assume that there
exist continuous, injective mappings
\begin{equation}\label{e:criss-cross}
i_-:\gb_-\too L_- \qquad \tand \qquad i_+:L_+\too \gb_+
\end{equation}
with dense images and which are compatible with the symplectic
structures, i.e.,
\begin{equation}\label{e:compatible}
\go_L(i_-(x),a) = \go_\gb(x,i_+(a)) \quad\text{ for all $a\in
L_+$ and $x\in \gb_-$}\,.
\end{equation}

\index{Fredholm Lagrangian Grassmannian}
Let ${\gl_0}$ be a fixed Lagrangian subspace of $\gb$. We consider
the {\it Fredholm Lagrangian Grassmannian} of ${\gl_0}$
\[
\flg{{\gl_0}}(\gb)
:= \{\mu\<\gb\mid \mu \text{ Lagrangian subspace
and $(\mu,\gl_0)$ Fredholm pair}\}.
\]
The topology of $\flg{{\gl_0}}(\gb)$ is defined
by the operator norm of the orthogonal projections onto the
Lagrangian subspaces.\index{Fredholm pair}

\begin{theorem}[Boo{\ss}-Bavnbek, Furutani, Otsuki] \label{t:criss-cross}
Under the assumptions \eqref{e:dsd}, \eqref{e:criss-cross}, and
\eqref{e:compatible},
we have a natural continuous mapping
\[
\tau  : \flg{\gb_-}(\gb)  \too \flg{L_-}(L),\quad \mu \longmapsto \mu\cap L,
\]
where $\gb$ and $L$ are identified with subspaces of $\gb_+\oplus L_-$\,.
\end{theorem}

The following splitting lemmata are of independent interest. Note
that we do not claim a direct sum decomposition  of $\gb(A)$ into
Lagrangian subspaces for now. Later, this will be a consequence of
our Theorem \ref{t:cobord}. See also the recent \cite[Section
1]{BalBruCar:RIT}. In the following lemma, we could use our $C_+$
instead of using Seeley's Calder{\'o}n projection $\cP_+$\,. All these
direct sum decompositions  of $\gb(A)$ are different, but equally
valid, see Lemma \ref{l:decomposition-comparison} and Remark
\ref{r:decomposition}.

\glossary{$\cP_+$}
\begin{lemma}\label{l:beta-decomposition}
Let $A$ be an elliptic formally selfadjoint first order
differential operator on a compact smooth manifold $M$ with smooth
boundary $\gS$ and let $\cP_+$ denote Seeley's corresponding
(pseudodifferential) Calder{\'o}n projection. Then the space
$\gb(A):= \cD(A_{\max})/\cD(A_{\min})$ can be described explicitly
as the direct sum
\begin{equation}
\gb(A)=\im(\cP_{+,-\half}) \oplus \im(\Id-\cP_{+,\half}),
\end{equation}
where $\cP_{+,s}$ denotes the extension/restriction of the pseudodifferential
$\cP$ to $L^2_s(\gS,E_\gS)$.
\end{lemma}

\begin{proof}
(1) First we show the inclusion $\im(\cP_{+,-\half}) \subset\gb(A)$\,. Let
$f$ belong to $\im(\cP_{+,-\half})$. Then there exists $f_1\in
L^2_{-\half}(\gS,E_\gS)$ with $f=r_+\widetilde G\rho^*{J_0}f_1$\,, where
$\widetilde G$ denotes Seeley's `inverse on the double' (which in
difference to our $\widetilde G$ is not canonically defined) and $\rho^*$
Seeley's dual of the trace (which, once again, in difference to our
$\rho^*$ is neither canonically defined). We observe that $g:= \widetilde
G\rho^*{J_0}f_1\in \cD(\widetilde A_{\max})$ with $\widetilde A g = 0$. Note that
this is Seeley's $\widetilde A$ which is neither canonical and therefore
not suitable for discussing the parameter dependence, but has the
advantage of delivering a pseudodifferential Calder{\'o}n projection.
So $f=\rho r_+ g$ with $r_+ g \in \cD(A_{\max})$. Hence
$f\in\gb(A)$.

\noi (2) Next we observe that $L^2_{\half}(\gS,E_\gS) \subset \gb(A)$
since $L^2_1(M,E) \subset \cD(A_{\max})$.

\noi (3) Together with argument (1) this implies
\[
\im(\cP_{+,-\half}) \oplus \im(\Id-\cP_{+,\half}) \subset \gb(A).
\]

\noi (4) To show the equality, we notice $\gb(A)\subset
L^2_{-\half}(\gS,E_\gS)$. Applying Seeley's result
\begin{equation}\label{e:Seeley-Calderon-splitting}
\cP_{+,s}+\cP_{-,s}=\Id
\end{equation}
for $s=-\half$ we can write each $f\in\gb(A)$ in the form
\[
f=f_1+f_2\,, \text{ where $f_1\in \im(\cP_{+,-\half})$ and $f_2\in
\im(\cP_{-,-\half})$}.
\]
By (1), $f_1\in\gb(A)$, so $f_2=f-f_1 \in\gb(A)$, i.e., there exists
a $g\in \cD(A_{\max})$ such that $f_2=\rho g$. Note $\cP_+f_2=0$ by
the splitting \eqref{e:Seeley-Calderon-splitting}. Applying one
version of G{\aa}rding's inequality (see, e.g., \cite[Chapter
18]{BooWoj:EBP})
\begin{equation}\label{e:Gaarding}
\|g\|_{L^2_1(M,E)} \le C\bigl(\|g\|_{L^2(M,E)} + \|Ag\|_{L^2(M,E)} +
\|\cP_+\rho g\|_{L^2_{\half}(\gS,E_\gS)}\bigr),
\end{equation}
we obtain $g\in L^2_1(M,E)$ and so $f_2=\rho g \in
L^2_{\half}(\gS,E_\gS)$, i.e., $f_2\in \im(\cP_{-,\half})$.
\end{proof}

\begin{lemma}\label{l:decomposition-comparison}
Let $P,Q$ be two pseudodifferential projections with the same leading
symbol. Then
\begin{equation} \im(P_{-\half}) \oplus \im(\Id-P_{\half}) =
 \im(Q_{-\half}) \oplus \im(\Id-Q_{\half}).
\end{equation}
\end{lemma}

\begin{proof}
So, let $f\in \im(P_{-\half})\oplus \im(\Id-P_{\half})\,.$
Then there are $\varphi\in L^2_{-1/2}(M,E)$, $\psi \in L^2_{1/2}(M,E)$
such that 
\[
    f= P\varphi +(I-P)\psi=Q\varphi + (I-Q)\psi+\underbrace{(P-Q)(\varphi-\psi)}_{=:h}.
\]
By assumption $P-Q$ is a pseudodifferential operator of order $\le -1$, hence
$h\in L^2_{1/2}(M,E)$ and thus
$f= Q(\varphi+h)+(I-Q)(\psi+h)$ with $\varphi+h\in L^2_{-1/2}(M,E)$
and $\psi+h\in L^2_{-1/2}(M,E)$, proving the claim.
\end{proof}

\begin{remark}\label{r:decomposition}
By combining the two preceding lemmata we obtain the useful formula
\begin{equation}\label{e:beta-decomp-gen}
\gb(A)=\im(P_{-\half}) \oplus \im\bigl((\Id-P)_{\half}\bigr)
\end{equation}
for all pseudodifferential projections $P$ with
$\gs_0(P)=\gs_0(\cP_+)$ where $\cP_+$ denotes Seeley's
(pseudodifferential) Calder{\'o}n projection.

Equation \eqref{e:beta-decomp-gen} generalizes a previous result in
\cite[Proposition 7.15]{BooFur:SFA} obtained for the spectral
Atiyah-Patodi-Singer projection $P:=P_{\ge}({B_0})$. There, the von
Neumann space $\gb(A)$ was expressed as the direct sum of
\begin{itemize}

\item the $L^2_{1/2}(\gS,E_\gS)$-closure of the linear span of the
negative eigenspaces of the tangential operator ${B_0}$ with

\item
the $L^2_{-1/2}(\gS,E_\gS)$-closure of the linear span of the
nonnegative eigenspaces of ${B_0}$

\end{itemize}
in the special case that the operator $A$ is of Dirac type in
product metric near the boundary (in particular, that $A$ has a
formally selfadjoint tangential operator ${B_0}$ and constant
coefficients in normal direction near the boundary).
\end{remark}

\subsection{Proof of Theorem \ref{t:cobord}}
As announced above, we prove the five claims successively.

\subsubsection{Application of the new construction of the Calder{\'o}n projection}

\begin{proof}[Proof of {\bf (I)}]
We recall that $(u,v)\mapsto \go(u,v):=\lla -{J_0}u,v\rra$ is a
symplectic form for the Hilbert space $\bigl(L^2(\gS,E_\gS),
\lla\cdot,\cdot\rra\bigr)$. It is strong since ${J_0}$ is a bundle
isomorphism. Then the range  $\im(C_+)=N_+^0$ is an isotropic
subspace because of Green's formula \eqref{eq-2.10}, applied to the\index{Green's formula}
kernel of the formally selfadjoint operator $A$. Here we use $T:=
{J_0}\bigl(-{J_0}^2\bigr)^{-1/2}$ to construct $C_\pm$\,, see Remark
\ref{rem3.20}. Then also $\im(C_{-})=T\ii(N_+^0)$ is an isotropic
subspace since the chosen $T$ is clearly symplectic. By
Proposition \ref{prop3.19}, we have $C_+ + C_- = \Id$, so $N_+^0$
and $T\ii(N_+^0)$ make a pair of transversal isotropic subspaces
of $L^2(\gS,E_\gS)$. Then {\bf (I)} follows by applying Lemma
\ref{l:isotropic-sum-to-lagrangian}.
\end{proof}

\begin{remark} We notice that the splitting $\cP_++\cP_-=\Id$ in
Seeley, \cite[Lemma 5]{See:SIB} does not provide two transversal
Lagrangian subspaces but only an isotropic range of $\cP_+$ with the
preceding argument in the case of symmetric $A$. The problem is that
even for symmetric $A$ Seeley's continuation into a collar and
further over the double does not preserve symmetry in general.
\end{remark}

\subsubsection{Stability arguments}
Before deriving {\bf (II)} we shall address stability aspects of the
issue.

We see at once that any formally selfadjoint operator of the form
\[
A:=-j_t\dd t-\frac 12(\dd t j_t)-b_t, \quad j_t \text{ invertible}
\]
on the unit interval $t\in [0,1]$ admits selfadjoint boundary
conditions. The symplectic form on the space of boundary values is
given by $J:= j_0\oplus (-j_1)$ with respect to the reversed
orientation at the ends of the interval. By continuity, we have
$\sign j_0=\sign j_1$. So, $\sign iJ=0$. (Note that there is no
tangential operator in the 1-dimensional case).

In higher dimension, a similar continuity argument does not work in
general.

The following lemma yields the stability of the signature of the
almost complex form $J_t$ on the nullspace $\ker B_t$ under
variation of the parameter $t$. It can be considered as an index
stability statement (and certainly can be proved also that way instead of the
proof given below).

\begin{lemma}\label{l:continuity-sign}
Let $(H, \lla \cdot,\cdot\rra)$ be a Hilbert space and $(S_t)$ and
$(J_t)$ be two continuous families of bounded invertible operators
on $H$, $t\in[0,1]$. Assume that all $S_t$ are positive definite.
Let $(B_t)$ be a continuous family of closed Fredholm operators.
We assume that all $iJ_t$ and $B_t$ are selfadjoint with respect to
the new metric $\lla x,y\rra_t:=\lla S_tx,y\rra$. Moreover, we
assume that all $J_t$ have bounded inverses, and $J_tB_t=-B_tJ_t$
for all $t$. Then we have
\begin{equation}
\sign(iJ_t\rrr {\ker B_t}) = \text{\rm constant}.
\end{equation}
\end{lemma}

\begin{proof} We divide the proof into three steps.

\noi {\em Step 1.}  We can assume that $S_t=\Id$. Indeed, denote by
$A^{*_t}$ the adjoint operator of $A$ with respect to the scalar
product induced by $S_t$\/. Then we have
$$
\langle S_tx,A^{*_t}y\rangle=\langle S_tAx,y\rangle.
$$
So $A^{*_t}=S_t^{-1}A^*S_t$. Set
$$
J_t':=S_t^{\frac{1}{2}}J_tS_t^{-\frac{1}{2}},\quad
B_t':=S_t^{\frac{1}{2}}B_tS_t^{-\frac{1}{2}}\/.
$$
Then $iJ_t'$ and $B_t'$ are selfadjoint, $J_t'B_t'=-B_t'J_t'$, and
$$
\sign(iJ_t\rrr {\ker B_t})=\sign(iJ_t'\rrr {\ker B_t'}).
$$
So we can assume $S_t=\Id$.

\noi {\em Step 2.} We reduce to the finite--dimensional case: For
each $t\in[0,1]$, there is a small $\e>0$ such that
$[-\e,\e]\cap\gs(B_t)\subset\{0\}$. Then for $s$ close to $t$,
$\pm\e\notin \gs(B_s)$. Let $D_{\e}:=\{z\in\C\mid|z|<\e\}$. Define
$$
P_s:=-\frac{1}{2\pi i}\int_{\partial D_{\e}}(B_s-z\Id)^{-1}dz.
$$
Then $P_s$ is a continuous family of orthogonal projections of
finite rank, and $P_sJ_sP_s$, $P_sB_sP_s:\ran P_s\to\ran P_s$
satisfy our assumptions, and $P_tB_tP_t=0$.

\noi {\em Step 3.} Since in the finite--dimensional case $\sign
(iJ_t)$ is constant, it suffices to prove the following

\noi {\em Claim.} Let $H$ be finite--dimensional. Then we have $\sign
(iJ_t)=\sign (iJ_t\rrr {\ker B_t}).$

In fact, let $V_t$ denote the orthogonal complement of $\ker B_t$.
Then both $\ker B_t$ and $V_t$ are invariant subspaces of $J_t$.
Since on $V_t$, $B_t$ is invertible, and since we have
$-J_t=B_t^{-1}J_tB_t$, we have $\sign(iJ_t\rrr {V_t})=0$. Hence our
claim follows.
\end{proof}

\subsubsection{Proof of {\bf (II),(III)} }
We now proceed to show the General Cobordism Theorem {\bf (II)}.

We exploit the formal selfadjointness of $A$ and choose,
in a collar of the boundary, the normal form
\[
\begin{split}
A={J_x}\Bigl(\dd x + {B_x}\Bigr) +\frac 12 J'_x
\end{split}
\]
of Equation \eqref{e:standard-symmetric}
with the relations $J^*=-J$, $JB=-B^tJ$.
The relation $JB=-B^tJ$ has consequences
for the positive/negative sectorial
spectral subspaces with regard to the natural symplectic structure
on $L^2(\Sigma,E_\Sigma)$, see the proof of
Lemma \plref{l:co-isotropic} below.

Similarly to Figure \ref{f:spec-section} we now choose contours
$\Gamma_<, \Gamma_>$ and $\Gamma_0$ as follows (see Figure
\ref{f:spec-three-section}): $\Gamma_<$ encircles all eigenvalues in
the left half plane, $\Gamma_>$ encircles all eigenvalues in the
right half plane, and $\Gamma_0$ encircles all eigenvalues on the
imaginary axis $i\R$. The corresponding spectral projections are
denoted by $P_<(B_0), P_>(B_0)$ and $P_0(B_0)$. In view of Theorem
\plref{t:boundedness-positive-sectorial}, these are
pseudodifferential operators of order $0$, and hence bounded, and as
closed idempotents they do have closed range (see Paragraph
\ref{sss:idempotents-hilbert-space}). $P_0(B_0)$ is of finite--rank
and hence a smoothing operator with range being the sum of the
generalized eigenspaces of $B_0$ to imaginary eigenvalues. We
therefore have a direct sum decomposition
\begin{equation}\label{e:W-splitting}
L^2(\gS,E_\gS) = \im P_<(B_0)\oplus \im P_0(B_0)\oplus\im P_>(B_0)
              =: W_<\oplus W_0\oplus W_>\,.
\end{equation}
In particular, $(W_>,W_<)$ is a Fredholm pair of closed subspaces \index{Fredholm pair}
of $L^2(\Sigma,E_\Sigma)$. Recall that the closedness of $W_<,W_>$ and
the finite codimension of $W_<+W_>$ in $L^2(\Sigma,E_\Sigma)$ imply that
$W_<+W_>$ is closed, see \cite[Remark A.1]{BooFur:SFA}.
%% \marginpar{check?}

\begin{figure}
\ifarxiv{\centerline{\includegraphics[height=6cm]{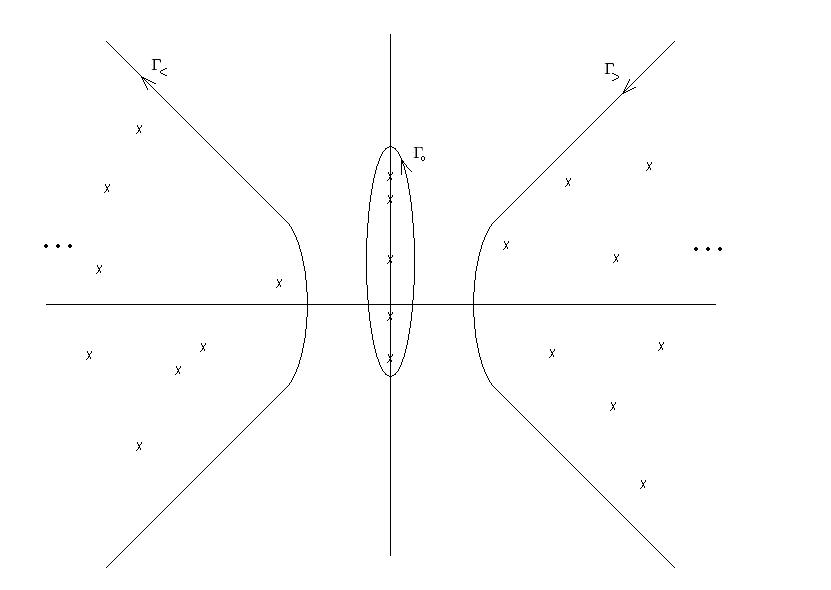}}}{%arxiv
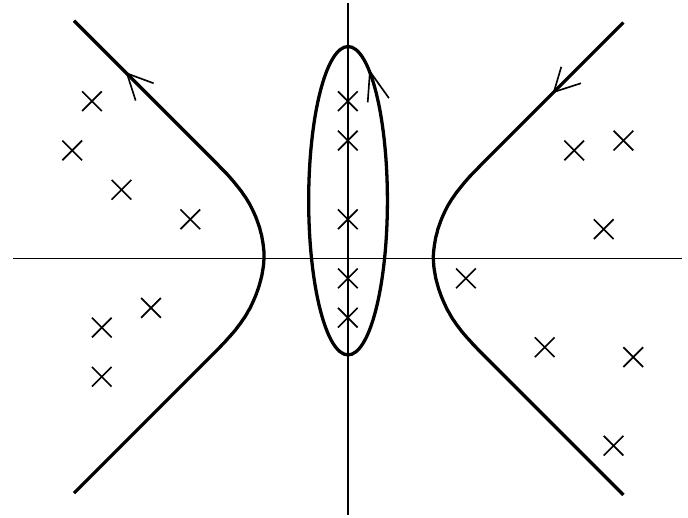}
\caption{\label{f:spec-three-section} Three contours encircling
all eigenvalues in the right half plane, on the imaginary axis,
and all eigenvalues in the left half plane, respectively.}
\end{figure}

\begin{lemma}\label{l:co-isotropic}
$W_>, W_<$ are isotropic subspaces with $W_>^\omega=W_0\oplus W_>$ and
$W_<^\omega=W_<\oplus W_0$.
\end{lemma}

\begin{proof} This is a consequence of the relation $J_0 B_0=-B_0^t J_0$ which implies
for $\xi\in L^2_1(\Sigma,E_\Sigma)$ (cf. Subsection \ref{ss:sectorial-abstract})
\begin{equation}\begin{split}
  J_0 P_>(B_0)\xi &=\frac{1}{2\pi i} \int_{\Gamma_>} \gl \ii J_0 (\gl-B_0)\ii d\gl\;  B_0\xi\\
                  &=\frac{1}{2\pi i} \int_{\Gamma_>} \gl \ii (\gl+B_0^t)\ii d\gl \;B_0^t J_0\xi\\
                  &= P_<(B_0^t)J_0\xi.
        \end{split}
\end{equation}
Then
\[\begin{split}
   W_>^\omega&=(J_0 W_>)^\perp=(J_0 \im P_>(B_0))^\perp=\im P_<(B_0^t)^\perp\\
            &=\ker P_<(B_0^t)^*=\ker P_<(B_0)=W_0\oplus W_>.
          \end{split}
\]
The other claim is proved analogously.
\end{proof}

Lemma \ref{l:co-isotropic} is now the key to the proofs of the
remaining implications in this paragraph. Before proceeding, we
note:

\begin{lemma}\label{l:fredholm-pair}
Let $P$ and $Q$ be bounded idempotents in a Banach space $H$, and
$P-Q$ compact. Then the pair $(\ker P,\im Q)$  is Fredholm.
\end{lemma}
\index{Fredholm pair}

\begin{proof} Consider the operator
\begin{align*}
R&:=QP+(\Id-P)(\Id-Q)=Q(Q+P-Q)+(\Id-Q+Q-P)(\Id-Q)\\
&=Q+Q(P-Q)+\Id-Q+(Q-P)(\Id-Q)\\
&=\Id+Q(P-Q)+(Q-P)(\Id-Q). 
\end{align*} 
Since $P-Q$ is compact, $R$ is Fredholm. Since $(\ker P\cap \im Q)\subset\ker R$ 
and $(\ker P+\im Q)\supset \im R$, we have $\dim(\ker P\cap \im Q)<+\infty$ 
and $\dim(H/(\ker P+\im Q))<+\infty$. 
Thus the pair $(\ker P,\im Q)$ is algebraically Fredholm, 
and hence Fredholm, since the spaces $\ker P,\im Q$ are closed.
\end{proof}

\begin{proof}[Proof of {\bf (I)} $\implies$ {\bf (II)}]
Let $\cP_+$ denote Seeley's corresponding (pseudodifferential)
Calder{\'o}n projection. Then we have $\im C_+=\im \cP_+$. Since the
difference between $\cP_+$ and $P_>(B_0)$ is of order -1, we have that
$(\im \cP_+,W_<)$ is a Fredholm pair by the preceding argument (and
finite--dimensional perturbation). Hence $\im \cP_+\oplus W_<$ is a
closed subspace.

By {\bf (I)}, $\im \cP_+$ is a Lagrangian subspace. Applying Lemma
\ref{l:symp-red-ml-subspace} to the co-isotropic subspace
$W_<^\go=W_0\oplus W_<$ we obtain that
\[
\Red_{W_<^\go}(\im \cP_+ )\simeq \pi(\Red_{W_<^\go}(\im \cP_+))\subset
W_0
\]
is a Lagrangian subspace of
\[
W_<^\go/W_<^{\go\go} = (W_<+W_0)/W_< \simeq W_0.
\]
So, the finite--dimensional symplectic Hilbert space
$\bigl(W_0,\scalar{i P_0{J_0}\cdot}{\cdot}\bigr)$
has a Lagrangian subspace. Therefore
\[\sign i P_0 {J_0}\rrr W_0=0.\qedhere\]
\end{proof}

We assume ${B_0}={B_0}^t$ and put $\ga=-i {J_0}(-{J_0}^2)^{-1/2}$\,.
Then the special case of claim {\bf (II)} of Theorem
\plref{t:cobord} follows. Moreover, we have under that assumption:

\begin{proof}[Proof of {\bf (II)} $\Leftrightarrow$ {\bf (III)}]
We follow the conventional lines and refer, e.g., to \cite[Theorem
21.5]{BooWoj:EBP} with the immediate modifications: Note that $\ga$
is a grading which thanks to ${B_0}={B_0}^t$ and
${B_0}{J_0}=-{B_0}^t{J_0}$ anticommutes with ${B_0}$, making ${B_0}$ an
odd operator. More precisely, let $(E_{\Sigma})^{\pm}$ denote the
positive (negative) eigenspace of $\ga$. Under the direct sum
decomposition $E_{\Sigma}=(E_{\Sigma})^+\oplus(E_{\Sigma})^-\,,$
the operator ${B_0}$ takes the form ${B_0}=\begin{pmatrix}0&B^-\\
B^+&0\end{pmatrix}$, where $B^- := (B^{+})^t$ and $B^+:\ker
(\ga-1)\to\ker(\ga+1)\,$. Then we have $\ker {B_0}=\ker B^+\oplus\ker
B^-$, and the positive (negative) eigenspace of $i{J_0}|_{\ker {B_0}}$
is $\ker B^{\pm}\,$. That proves
\[\sign(i{J_0}|_{\ker {B_0}})=\ind B^+\,.\qedhere\]
\end{proof}

\subsubsection{Proof of {\bf (II)} $\implies$ {\bf (IV),(V)}}
To show that {\bf (IV)} and {\bf (V)} are trivial consequences of {\bf (I)},
one is tempted to set
$P:=C_+$\,. However, in Proposition \ref{prop3.19} we have established that
$C_+$ is the $L^2$ extension of a pseudodifferential projection only for
the boundary operator $T:= (J_0^t)\ii$\,, contrary to our assumption
$T:= J_0(-J_0^2)^{-1/2}$ in {\bf (I)} for achieving that $\ker C_+ = \im C_-$ becomes a
Lagrangian subspace of $L^2(\gS,E_\gS)$. To prove that $C_+$ is also pseudodifferential
for the last choice of $T$ would require applying more advanced elliptic boundary
value theory. Instead of that, we give a simple construction of the
wanted $P$ as a perturbation of the positive spectral projection $P_>(B_0)$
by a suitable finite--rank operator, and let simple symplectic analysis
do the remainder of the work:

\begin{proof}[Proof of {\bf (II)} $\implies$ {\bf (IV),(V)}]
The vanishing of the signature $\sign iP_0{J_0}\rrr W_0$ on the
finite--dimensional space $W_0$ implies the existence of a
transversal pair of Lagrangian subspaces $(\gl,\mu)\subset W_0$. The
pair $\bigl(W_> +\gl, W_<+\mu\bigr)$ is a transversal pair of
Lagrangian subspaces of $L^2(\gS,E_\gS)$. Denote by $P$ the
projection of $L^2(\gS,E_\gS)$ onto $W_> +\gl$. Then $P$ is a zeroth
order pseudodifferential operator, and $P-\cP_+$ is of $-1$ order.
Then $(\ker P,\im \cP_+)$ is a Fredhom pair, and $\ker P=W_<+\mu$ is a
Lagrangian subspace of $L^2(\gS,E_\gS)$. Since $\im C_+=\im \cP_+$,
(IV) follows.

Now we consider {\bf (V)}. By Remark \ref{r:decomposition} we have
$\gb(A)=\im P_{-\half}\oplus\im (\Id-P_{\half})$\,. Clearly $\im
P_{-\half}$ and $\im (\Id-P_{\half})$ are isotropic subspaces of
$\gb(A)$. By Lemma \ref{l:isotropic-sum-to-lagrangian}, they are
Lagrangian subspaces. Then the extension $A_P$ is a selfadjoint
extension. Fredholmness follows from leading symbol consideration.
\end{proof}

\subsection{Alternative routes to the General Cobordism Theorem}
We describe alternative routes to prove Theorem \ref{t:cobord}. In the present context {\bf (I)} is
an immediate consequence of our Calder{\'o}n construction. Therefore, we began
the proof of the General Cobordism Theorem with a proof of {\bf (I)}.

One true alternative is to begin with a proof of {\bf (V)}: That
claim was proved in \cite[Theorem I]{Ral:DISO} for bounded regions
$M$ in Euclidean space. However, Ralston's arguments fully
generalize and can be summarized in the following way (in our
present notation):

\begin{proof}[Outlines of a proof of {\bf (V)}]
(1) First we notice the pointwise vanishing of the signature of the
form $i{J_0}\rrr p$ on the fibre $E_p$ for each $p\in \gS$. This can
be obtained by deforming ${J_0}\rrr p$ into a strict anti-involution
$\widetilde {J_0}\rrr p$ with $\bigl(\widetilde {J_0}\rrr p\bigr)^2 = -\Id$ and
exploiting the anti-commutative relation \eqref{eq2.12} with the
elliptic symbol. Consequently, the fibre dimension of the bundle $E$
must be even. That permits the first trick, namely to split
$E=E_+\oplus E_-$ and to show that there exists a well-posed
symmetric Fredholm extension given by the graph of a
pseudodifferential elliptic operator
$P:\ginf{\gS;E_+\rrr\gS}\to\ginf{\gS;E_-\rrr\gS}$\,.

\noi (2) Next we show that the deficiency indices $\gk_1\,, \gk_2$
of $A$ are finite and that their difference is equal to the index
$\ind P$.

\noi (3) Then we show that $-\ind P = \sign
-i\omega\rrr{\operatorname{vN}(P)}$ where $\omega$ denotes the Green
form and $\operatorname{vN}(P)$ denotes a suitably defined subspace
of the von Neumann quotient $\cD(A_{\max})/\cD(A_{\min})=: \gb(A)$.

\noi (4) Finally we show that $\ind P$ vanishes, and hence that
$A_P$ can be extended to a selfadjoint $A_{\widetilde P}$ with domain
given by a pseudodifferential projection and preserving the
Fredholm property.
\end{proof}

Note that we can deduce {\bf (I)} from {\bf (V)} in the following
way, and independently of the delicacies of our Calder{\'o}n
construction. We firstly show that we have {\bf (V)}$\implies${\bf
(II)}.

\begin{proof}[Proof of {\bf (V)} $\implies$ {\bf (II)}]
Since $A_P$ is a selfadjoint Fredholm operator, by
\cite[Proposition 3.5]{BooFur:MIF}, $\im \cP_{+,-\half}$ is a Lagrangian
subspace of $\gb(A)$.

By Remark \ref{r:decomposition} and Equation (\ref{e:W-splitting}),
we have
\begin{equation} \gb(A)= \im P_{<,\half}(B_0)\oplus \im
P_0(B_0)\oplus\im P_{>,-\half}(B_0)\,.
\end{equation}

Using the same method as in the proof of {\bf (I)} $\implies$
{\bf (II)} and applying Lemma \ref{l:symp-red-ml-subspace} to the
co-isotropic subspace $\im \left(P_{<,\half}(B_0)\right)^\go =\im
P_0(B_0)\oplus \im P_{<,\half}(B_0)$, we obtain that
\[\Red_{\left(P_{<,\half}(B_0)\right)^\go}(\im \cP_{+,-\half})\simeq
\pi(\Red_{\left(P_{<,\half}(B_0)\right)^\go}(\im
\cP_{+,-\half}))\subset W_0\] is a Lagrangian subspace of $W_0$. So,
the finite--dimensional symplectic Hilbert space
$\bigl(W_0,\scalar{i P_0{J_0}\cdot}{\cdot}\bigr)$ has a Lagrangian
subspace. Therefore
\[\sign i P_0 {J_0}\rrr W_0=0.\qedhere\]
\end{proof}

\begin{proof}[Proof of {\bf (V)} $\implies$ {\bf (I)}] By the above proofs we have
$L^2(\gS,E_\gS)=\im P\oplus\im (\Id-P)$ and $\gb(A)=\im
P_{-\half}\oplus\im (\Id-P_{\half})$\,. By Theorem
\ref{t:criss-cross}, $\im \cP_{+,-\half}\cap L^2(\gS,E_\gS)=\im
\cP_+=\im C_+$ is a Lagrangian subspace of $L^2(\gS,E_\gS)$.
\end{proof}

\section{Parameter dependence}\label{s:parameter}

%%%%%%%%%%%%%%%%%%%%%%%%%%%%%%%%%%%%%%%%%%%%%%%%%%%%%%%%%%%%%%%%%%%%%%%%%%%%%%%%%%%%%%%%%
%\commentary{
%\begin{itemize}
%\item Notation: $\R_+=[0,\infty)$. Fix at the beginning
%\item General reference for Sobolev spaces, trace norms etc: book of Lions Magenes
%\end{itemize}
%}
%%%%%%%%%%%%%%%%%%%%%%%%%%%%%%%%%%%%%%%%%%%%%%%%%%%%%%%%%%%%%%%%%%%%%%%%%%%%%%%%%%555555

In this section we discuss the continuous dependence of the
Calder{\'o}n projection and the Poisson operator on the input
data. That is, given a first order elliptic differential operator $A\in\Diff^1(M;E,F)$
and $T\in\Diff^0(\Sigma;E_\Sigma,F_\Sigma), \Sigma=\partial M$ 
(cf. Definition \plref{def3.6}),
we want to have criteria to ensure that $(A,T)\mapsto K_+=K_+(A,T)$
respectively $(A,T)\mapsto C_+=C_+(A,T)$ is continuous in an appropriate sense.

Therefore we first introduce various metrics on the spaces of pairs $(A,T)$.
Referring to \eqref{eq4.15} and \eqref{eq:4.15-adjoint} we consider
$J_0, B_0, \stackrel{(\sim)}{C_0}, \stackrel{(\sim)}{C_1}$ as functions
of $A$. 

\glossary{$\cV(M;E,F)$} 
\begin{dfn}\label{d:strength-SALambda}
\noi \textup{(a)} Let $\cV(M;E,F)$
denote the subspace of $\Diff^1(M;E,F)\times \Diff^0(\Sigma;E_\Sigma,F_\Sigma)$
consisting of those $(A,T)$ for which $[B_0^t,J_0^tT]$ is of order $0$
(cf. \eqref{eq:3.45-additional}). By $\cE(M;E,F)$ we denote the subspace
of $\cV(M;E,F)$ consisting of those pairs $(A,T)$ where
\begin{thmenum}
\item $A$ is elliptic
\item $T$ is invertible and satisfies \eqref{eq3.45-rep}.
\end{thmenum}
Finally, we denote by $\cE_{\UCP}(M;E,F)$ the subspace of $\cE(M;E,F)$ consisting
of pairs $(A,T)$ where $A$ and $A^t$ satisfy weak inner \UCP.

\noi \textup{(b)} On the linear space $\cV(M;E,F)$ we introduce two norms:
\begin{align}\label{eq:weak-norm}
   N_0(A,T)&:= \|A\|_{1,0}+\|A^t\|_{1,0}+\|T\|_{\half,\half}\,,\\
\intertext{and}
    N_1(A,T)&:=\|B_0\|_{1,0}+\|B_0^t\|_{1,0}+\|[B_0^t,J_0^tT]\|_0
                       +\|T\|_0 \label{eq:strong-norm}\\
                   &\quad +\|J_0\|_0+ \|C_1\|_{1,0}+\|C_0\|_0+\|\widetilde C_1\|_{1,0}+\|\widetilde C_0\|_0.\nonumber
\end{align} 
Except for $C_1,\widetilde C_1$ the norms in \eqref{eq:strong-norm} are the mapping norms between Sobolev spaces
over $\Sigma$ while the norms for $C_1, \widetilde C_1$ are mapping norms between
Sobolev spaces over the collar $[0,\varepsilon)\times \Sigma$.

\noi \textup{(c)}
We equip the space $\cE(M;E,F)$ with
the metric $d_0$ induced by the metric $N_0$ of \eqref{eq:weak-norm}. I.e. 
\begin{equation}
   d_0((A,T),(A',T')):=N_0(A-A',T-T').
\end{equation}
\end{dfn}
\glossary{$N_0(A,T), N_1(A,T)$} 
\glossary{$\cE(M;E,F)$} 

The norms $N_0,N_1$ induce metrics on subspaces of $\cV(M;E,F)$, in particular
on $\cE(M;E,F)$.

To study the dependence of $K_+$ and $C_+$ on $(A,T)$ the formulas
\eqref{eq4.23} and \eqref{eq4.24} in Proposition \plref{prop4.5} are crucial.

To illustrate this let us consider a map $Z\ni z\mapsto (A(z),T(z))\in\cE(M;E,F)$
from a metric space $Z$ to $\cE(M;E,F)$. To conclude that the corresponding
map $z\mapsto K_+(A(z),T(z))\in \cB(L^2(\Sigma,E_\Sigma),L^2_s(M,E))$ is
continuous for some fixed $0\le s\le \half$ it suffices to show the continuity
of
\begin{enumerate}
\item\label{ad1}
$z\mapsto \varphi R_{T(z)}(A(z))\in \cB(L^2(\Sigma,E_\gS),L^2_s(M,E\oplus F))$,
\item\label{ad3} 
$z\mapsto S(A(z),T(z))\in\cB(L^2(\Sigma,E_\gS),L^2(M,F\oplus E))$,
\item\label{ad2} 
$z\mapsto \widetilde G(A(z),T(z))\in\cB(L^2(M,F\oplus E),L^2_1(M,E\oplus F))$.
\end{enumerate}

For the continuity of $z\mapsto C_+(A(z),T(z))\in\cB(L^2(\Sigma,E_\gS))$
\enumref{ad1} has to be replaced by the continuity of the map
\begin{enumerate}
\item[(1')]\label{ad1-Cald}
$z\mapsto P_+(B_0(A(z)))\in\cB(L^2(\Sigma,E_\gS))$.
\end{enumerate}
The continuity of these maps is by no means necessary for ensuring the
continuous dependence of $K_+, C_+$. 
In order to keep the presentation reasonable in size
we estimate generously - we are not striving for optimality here.

Let us now state the main result of this section. 

\index{strong metric}
We define the \emph{strong metric} on the space $\cE(M;E,F)$
by 
\begin{equation}\label{eq:strong-metric}
\dstr((A,T),(A',T')):=N_0(A-A',T-T')
    +N_1(A-A',T-T').
\end{equation}
  
Note that by complex interpolation $\|T-T'\|_{s,s}\le
\dstr((A,T),(A',T'))$ for all $0\le s\le\half$.

\glossary{$\cE_{\UCP}(M;E,F)$} 
\begin{theorem}\label{thm4.14}
\noi \textup{(a)}
The map 
\[(\cE_{\UCP}(M;E,F),\dstr)\longrightarrow \cB(L^2(\gS,E_{\gS}),L^2_s(M,E))\]
sending $(A,T)$ to the
Poisson operator $K_+(A,T)$ is continuous for $0\le s<\half$.

\noi \textup{(b)} Let $(A(z),T(z))_{z\in Z}$ be a continuous family in $(\cE_{\UCP}(M;E,F),\dstr)$
pa\-ra\-metrized by a metric space $Z$. Assume that the corresponding
family
\[z\mapsto P_+(B_0(A(z)))\in\cB(L^2_s(\Sigma,E_\Sigma))\]
of positive spectral projections of the tangential operator is continuous
for some fixed $s\in[-\half,\half]$.
Then the map 
\[Z\longrightarrow \cB(L^2_s(\gS,E_{\gS}))\]
sending $z$ to the Calder{\'o}n projection $C_+(A(z),T(z))$ is
continuous.
\end{theorem}

\begin{remark}
\enum{1}
Of course, analogous statements hold for $K_-, C_-$. We leave it
as an intriguing problem whether the statement about $K_+$ still
holds for $s=\half$. This would be more natural since $K_+$ maps
$L^2$ to $L^2_{\half}$.

\enum{2} The fact that the continuous dependence of $P_+$ in 
(b) has to be assumed is not very satisfactory. The point here
is \emph{not} that the construction of $P_+$ requires a spectral cut.
Suppose a spectral cut for $B_0=B_0(A(z_0))$ is chosen.
Then $P_+$ should vary continuously as long as no eigenvalues
approach the contours $\Gamma_\pm$ (cf. Figure \plref{f:gamma_pm}).
Unfortunately we cannot prove the continuity of $B\mapsto P_+(B)\in \cB(L^2(\Sigma,E_\gS))$
if we equip $\Diff^1(\Sigma;E_\gS)$ say with the norm $\|\cdot\|_{1,0}$;
we cannot prove it for any other norm either.
We will come back to this problem below in Subsection \plref{ss:dependence-Pplus},
where we will give a criterion for the continuity of $P_+$ in special cases.
\end{remark}

\begin{proof} The Theorem follows from Proposition \plref{prop4.5},
  Proposition \plref{prop4.6}, Proposition \plref{prop4.7}, and
Theorem \plref{thm4.10}.
\end{proof}

The discussion in Subsection \plref{ss:dependence-Pplus} will give at least
the following result:
\begin{cor}\label{thm4.14-sa} Denote by $\cE_{\UCP}^{\rm sa}(M;E,F)$ the
subspace of $\cE_{\UCP}(M;E,F)$ consisting of pairs $(A,T)$ where the
corresponding tangential operator $B_0(A)$ has selfadjoint leading symbol.
Then for $s\in [-\half,\half]$ the map
\[    (\cE_{\UCP}^{\rm sa}(M;E,F),\dstr)\longrightarrow \cB(L^2_s(\gS,E_{\gS}))\]
sending $(A,T)$ to the Calder{\'o}n projection $C_+(A,T)$ is continuous.
\end{cor}
\begin{proof} We may adopt the language of a family
$(A(z),T(z))_{z\in Z}$ (with $Z=\cE_{\UCP}^{\rm sa}(M;E,F)$!) of the previous
Theorem \plref{thm4.14}.
 
The point is that locally (cf. Convention \ref{fix_c_gamma_omega})
one can choose a continuous family $z \mapsto P_+(B_0(z))$.
Then the result follows from Theorem \plref{thm4.14}b.

To see this we recall from Remark \plref{r:self-adjoint-leading} that if the
tangential operator has selfadjoint leading symbol the $B_0=B_0(A(z))$ 
in \eqref{eq4.15} can be chosen to be selfadjoint. Hence let $B_0(z)$ now
denote this selfadjoint operator in the representation \eqref{eq4.15}.
To show continuity at $z_0$ pick a spectral cut $c$ for $B_0(z_0)$. Then
by Proposition \plref{p:dependence-Pplus-self-adjoint} below the family
$P_+(B_0(z)):=1_{[c,\infty)}(B_0(z))\in\cB(L^2_s(\Sigma,E_\Sigma))$ depends
continuously on $z$ in a neighborhood of $z_0$. Hence Theorem \plref{thm4.14}b
yields the claim.
\end{proof}

\begin{remark}\label{r:cald-contin}
\cite[Section 3.3]{BooFur:MIF} gives a purely functional analytic proof of the continuous variation 
of the Cauchy data spaces as subspaces of the von-Neumann quotient $\gb(A)$ of all natural boundary values, 
as defined above in Paragraph \ref{sss:von-Neumann}, in great generality: only the symmetry of $A$, weak inner UCP and the existence of a selfadjoint Fredholm extension are assumed. In particular, no product form near the boundary or symmetry of a tangential operator is assumed. However, loc. cit. is restricted to continuous variation by bounded perturbations, i.e., perturbations of lower order in the operator norm, whereas the preceding corollary admits arbitrary continuous variations, though in the strong metric. 

Below in Proposition \ref{p:lower-order} we shall explain why lower order perturbations of a fixed operator lead to continuous variation of the sectorial projection in our setting, and hence to continuous variation of the Calder{\'o}n projection according to the preceding Theorem. 
\end{remark}

We now proceed to give criteria for the continuity of the maps \enumref{ad1}-\enumref{ad3}, \enum{1'}.
We start with some basic estimates.

\subsection{Some estimates}\label{ss:some-estimates}

We fix a first order elliptic differential operator $B\in\Diff^1(\Sigma;E_\Sigma)$ such that
$B-\gl$ is parameter dependent elliptic in a conic neighborhood of $i\R$ (cf. Subsection
\plref{ss:parameter-dependent-ellipticity}). Choose
contours $\Gamma_\pm$ accordingly as in Figure \plref{f:gamma_pm}. 
%Finally we
%fix a $\lambda_0\not\in\spec B\cup\Gamma_+\cup\Gamma_-$.

Before we address the continuous dependence of the sectorial projections on the data, i.e., on $B$,
we shall give some useful estimates.

We will frequently use that for $V\in\CL^1(\Sigma;E_\Sigma)$ we have by duality
$\|V\|_{0,-1}=\|V^t\|_{1,0}$.

Our first result is the following perturbation lemma.

\begin{lemma}\label{lemma4.2}
Let $V\in \Diff^1(\Sigma;E_{\gS})$. If
$\|V\|_{1,0}+\|V^t\|_{1,0}$ is sufficiently small, then $B+V-\gl$ is parameter dependent
elliptic in a conic neighborhood of $i\R$ containing $\gG_+,\, \gG_-$\,.

Furthermore, for $|s|, |s'|, |s-s'|\le 1$ we have for $\gl\in \gG_-\cup\gG_+$
\[\begin{split}
    \|(\gl&-(B+V))^{-1}-(\gl- B)^{-1}\|_{s,s'}\\
                      &\le C(s,s',B)\, \bigl(\|V\|_{1,0} + \|V^t\|_{1,0}\bigr)\,
                      |\gl|^{-1+s'-s}.
  \end{split}
\]
\end{lemma}

\begin{proof} The first claim is clear. The second follows from a straightforward application of the Neumann series
for the resolvent of $B+V$ and complex interpolation.
For the convenience of the reader we present some details of the
estimate.
%%%%%%%%%%%%%%%%%%%%%%%%%%%%%%%%%%%%%%%%%%%%%%%%%%%%%%%%%%%%%%%%%%%%%%%%%%%%%%%%%%%%%%
%% \commentary{Dear coauthors: decide whether you want to have these
%% details or not}
%%%%%%%%%%%%%%%%%%%%%%%%%%%%%%%%%%%%%%%%%%%%%%%%%%%%%%%%%%%%%%%%%%%%%%%%%%%%%%%%%%%%%

For $s\in\{0,1\}$ Equation \eqref{e:resolvent-estimate} yields for $\gl\in\Gamma_+\cup \Gamma_-$
\begin{equation}
    \begin{split}
           \|(B-\gl)\ii V\|_{s,s} &\le \|(B-\gl)\ii\|_{s-1,s}\, (\|V\|_{1,0}+\|V^t\|_{1,0})\\
                                 &\le C(s)\, (\|V\|_{1,0}+\|V^t\|_{1,0}),
    \end{split}
\end{equation}
and similarly
\begin{equation}
                 \|V(B-\gl)\ii\|_{s,s} \le C(s) \, (\|V\|_{1,0}+\|V^t\|_{1,0}).
\end{equation}
Furthermore, complex interpolation (or Hadamard's three line
theorem) gives that there is a constant $C$ such that
\begin{equation}
    \begin{split}
       \sup_{0\le s\le 1,\, \gl\in\Gamma_+\cup\Gamma_-}  \|V(B-\gl)\ii\|_{s,s} 
        &+ \sup_{0\le s\le 1,\, \gl\in\Gamma_+\cup\Gamma_-}  \|(B-\gl)\ii V\|_{s,s}\\
       \le C\, (\|V\|_{1,0}&+\|V^t\|_{1,0}).
    \end{split}
\end{equation}
Choose $V$ such that $C\,(\|V\|_{1,0}+\|V^t\|_{1,0})<\half$. Consequently, $B+V-\gl$ is invertible
for all $\gl\in\Gamma_+\cup\Gamma_-$ and as operator $L^2_s\to L^2_{s-1}$\,,
$0\le s\le 1$, its inverse is given by the Neumann series
\begin{align*}
(B+V-\gl)\ii &= (\Id + (B-\gl)\ii V)\ii (B-\gl)\ii\\
&=\sum_{n\ge 0} (-1)^n \bigl( (B-\gl)\ii V\bigr)^n(B-\gl)\ii\,.
\end{align*}
Hence
\begin{equation}\label{eq:neumann-series}
\|(B+V-\gl)^{-1}-(B-\gl)^{-1}\|_{s,s'}
\le \sum_{n\ge 1} \bigl\|\bigl((B-\gl)^{-1}V\bigr)^n\, (B-\gl)^{-1}\bigr\|_{s,s'}\,.
\end{equation}

Now one has to check case by case.

\noi {\it 1}. Let $s'\ge 0,\, s'-s\ge 0$. Then
\begin{equation}
  \begin{split}
   \bigl\|\bigl((B-\gl)^{-1}&V\bigr)^n (B-\gl)^{-1}\bigr\|_{s,s'} \\
                      &\le \|(B-\gl)^{-1}\|_{s,s'} \, \|(B-\gl)^{-1}V\|^n_{s',s'}\\
                      &\fle{\eqref{e:resolvent-estimate}} C'\, |\gl|^{-1+s'-s} \bigl(C\, (\|V\|_{1,0}+\|V^t\|_{1,0})\bigr)^n\\
                      &\le \widetilde C |\gl|^{-1+s'-s} (\half)^{n-1} \, (\|V\|_{1,0}+\|V^t\|_{1,0}).
  \end{split}
\end{equation}
Summing up gives the claim in this case.

\noi {\it 2}. Let $-1\le s \le s'\le 0$. Then
%\begin{multline}
\[  \begin{split}
   \bigl\|\bigl((B-\gl)^{-1}&V\bigr)^n (B-\gl)^{-1}\bigr\|_{s,s'}\\
                      &\le \|(B-\gl)^{-1}V\|_{0,s'} \, \bigl\|\bigl((B-\gl)^{-1}V\bigr)^{n-1}\bigr\|_{0,0} \, \|(B-\gl)^{-1}\|_{s,0}\\
                      &\le C\, |\gl|^{-1+s'+1} \, (\|V\|_{1,0}+\|V^t\|_{1,0}) \, (\half)^{n-1} \, |\gl|^{-1-s}.
  \end{split}\]
%\end{multline}
Again, summing up gives the claim also in this case.

For estimating $\bigl\|\bigl((B-\gl)^{-1}V\bigr)^n (B-\gl)^{-1}\bigr\|_{s,s'}$\,, the roles of $s$ and $s'$ are symmetric. Hence
the cases $s'\le s$ follow analogously.
\end{proof}
%
%\marginpar{Here BBB was too fast when he removed the prop. we need it anyway.
%There is no free lunch}

We shall investigate the stability of the sectorial projections
under perturbation of the input data $B$ by $V$ and  show that the
operator norm of $\varphi(Q_\pm (B+V)-Q_\pm(B))$ from
$L^2_s(\Sigma,E_{\gS})$ to $L^2_{s',\comp}(\R_+\times
\Sigma,E_{\gS})$ is bounded by a constant depending on $s,s', B,
\varphi$ times $(\|V\|_{1,0}+\|V^t\|_{1,0})$.

\begin{prop}\label{prop4.3}
Let $\varphi\in\cinfz{\R_+}$ and $V$ as in Lemma \plref{lemma4.2}.
$Q_\pm(B)$ and $Q_\pm(B+V)$ are the operator families of $B$ respectively $B+V$
introduced in Definition \plref{d:gen-spec-proj}.

\begin{thmenumm}
\item For
$-\half<s\le s'<s+\half, s'\le 1$ we have for $\xi\in
L^2_s(\Sigma,E_{\gS})$
\begin{equation}\label{eq4.11}
    \bigl\|\varphi\bigl(Q_\pm(B+V)-Q_\pm(B)\bigr)\xi\bigr\|_{s'}
     \le C(s,s',B,\varphi) \, (\|V\|_{1,0}+\|V^t\|_{1,0})\, \|\xi\|_s\,.
\end{equation}

\item For $-1\le s\le 0$ we have for $\xi\in L^2_s(\Sigma,E_{\gS})$
\begin{equation}\label{eq4.12}
  \bigl\|\id_{\R_+}\varphi \bigl(Q_\pm(B+V)-Q_\pm(B)\bigr)\xi\bigr\|_{s+1}
     \le C(s,B,\varphi) \, (\|V\|_{1,0}+\|V^t\|_{1,0})\, \|\xi\|_s\,.
\end{equation}
\end{thmenumm}
\end{prop}
\begin{proof} We use Lemma \plref{lemma4.2} and estimate for $-1\le s\le s'\le s+1, s'\le 1$, 
$\xi\in L^2_s(\Sigma,E_{\gS})$
and $m\in\{0,1\}$ 
\begin{equation}\label{eq4.13}
  \begin{split}
    \bigl\|&x^m \varphi(x)\bigl(Q_\pm(B+V)-Q_\pm(B)\bigr)(x)\, \xi\bigr\|_{s'}\\
      &\le \frac{x^m}{2\pi }\varphi(x)\int_{\gG_\pm} |e^{-x\gl}|
      \bigl\|\bigl((\gl-(B+V))^{-1}-(\gl -B)^{-1}\bigr)\xi\bigr\|_{s'}\, |d\gl|\\
      &\le C(s,s',B) \frac{x^m\varphi(x)}{2\pi }\int_{\gG_\pm}
      |e^{-x\gl}|\, |\gl|^{-1+s'-s} \, |d\gl|\, (\|V\|_{1,0}+\|V^t\|_{1,0})\, \|\xi\|_s\\
    &\le C(s,s',B) (\|V\|_{1,0}+\|V^t\|_{1,0})\, \|\xi\|_s\, x^{m-s'+s} |\log x| \varphi(x).
  \end{split}
\end{equation}

Analogously and using the previous estimate \eqref{eq4.13} we find
for $0\le s\le s'=1$
\begin{multline}\label{eq4.14}
%\begin{split}
  \Bigl\|\pl_x\Bigl(x^m \varphi(x)\bigl(Q_\pm(B+V)-Q_\pm(B)\bigr)\xi(x)\Bigr)\Bigr\|_{L^2(\Sigma,E_{\gS})}\\
     \le C(s,B,\varphi) (\|V\|_{1,0}+\|V^t\|_{1,0})\, \|\xi\|_s\times\\
         \times\Bigl( x^s\pl_x(x^m\varphi(x)) +\varphi(x) x^{m-1+s}\Bigr)|\log x|.
%        \end{split}
%\end{equation}
\end{multline}
   
The $\log x$--terms on the right of \eqref{eq4.13}, \eqref{eq4.14} are necessary
only in the case $s=s'$. They are obsolete if $s<s'$.

From these two estimates both claims will follow in a
straightforward manner:

\smallskip

\noi (a). It suffices to prove the claim for $-\half<s<s'=0$ and for
$\half<s\le s'=1$. It is clear that it then holds for $-\half<s\le
s'\le 0$ and $\half<s\le s'\le 1$. The general case then follows
from the complex interpolation method.

So let us start with the case $-\half<s\le s'=0$. Since $s>-\half$
we may integrate the square of \eqref{eq4.13} and reach the
conclusion.

If $\half<s\le s'=1$ then apply \eqref{eq4.13} and \eqref{eq4.14}
with $m=0$. Squaring and integrating the inequality gives the claim
in view of \eqref{eq4.2}.

\smallskip

\noi (b). By interpolation theory it is enough to deal with the cases
$s=-1$ and $s=0$. If $s=-1$ apply \eqref{eq4.13} with $s'=s+1=0$ and
if $s=0$ apply \eqref{eq4.13} with $s'=s+1=1$ and \eqref{eq4.14}.
Again referring to \eqref{eq4.2} we are done.
\end{proof}

%\subsection{Parameter dependence of Poisson operator, \Calderon\ projector, and
%  invertible double}
%\label{ss:parameter-dependence}

\subsection{Continuous dependence of $\varphi R_T(B_0)$}

Since $R_T$ is the
multiplication of $R$ by a simple matrix containing $\Id$ and
$T$ it suffices to study the dependence of $R=R(B_0)$ on $B_0$:

\begin{prop}\label{prop4.6} Let $(A,T)\in\cE(M;E,F)$.
Fix a real number $c>0$ such that 
$\spec B_0\cap\bigsetdef{z\in\C}{|z|=c}=\emptyset$, cf.  
Convention \plref{fix_c_gamma_omega}. Let
$V\in\Diff^1(\gS;E_{\gS})$ be a first order differential operator. 
According to Lemma \plref{lemma4.2} assume 
that $(\|V\|_{1,0}+\|V^t\|_{1,0})$ is small enough so that
$B_0+V-\lambda$ is parameter dependent elliptic in a conic neighborhood
of $i\R$ containing $\Gamma_+\cup \Gamma_-$.

Then for $-\half<s\le s'<s+\half, s'\le 1$ we have
\begin{equation}
   \bigl\|\varphi \bigl(R(B_0+V)-R(B_0)\bigr)\bigr\|_{s,s'}\le
     C(s,s',B_0,\varphi) (\|V\|_{1,0}+\|V^t\|_{1,0}).
\end{equation}
\end{prop}

\begin{proof} This follows immediately from Proposition \plref{prop4.3}
and Equations \eqref{eq4.18} and \eqref{eq4.19}.
\end{proof}

\subsection{Continuous dependence of the invertible double}
\label{ss:invertible-double-continuous}
%\marginpar{This paragraph checked, is mathematically clean 30.01.08}

Dealing with the generalized inverse would make the discussion of
the parameter dependence of the invertible double rather tedious.
Therefore we assume in this Subsection \plref{ss:invertible-double-continuous}
that UCP holds. Although with some care the results of this section
carry over to families of operators where the dimensions of the spaces
of \textqm{ghost solutions} $Z_{+,0}(A),Z_{-,0}(A)$, see \eqref{eq3.26},
remain fixed (cf. 
\cite[Thm. 3.16]{HimKirLes:CPH}).

Recall the definition of spaces and norms of Definition \plref{d:strength-SALambda}.
The goal of this subsection is to prove:

\begin{theorem}\label{thm4.10} The map 
$(\cE_{\UCP},d_0)\longrightarrow \cB(L^2(M,F\oplus E),L^2_1(M,E\oplus F))$,
$(A,T)\mapsto \widetilde A_{P(T)}^{-1}$ is continuous.
\end{theorem}

Note that this is much more than just graph continuity of $(A,T)\mapsto \widetilde A_{P(T)}\ii$. 
Namely, by
construction of the metric $d_0$ this means that also
$(A,T)\mapsto \bigl((\widetilde A_{P(T)})^*\bigr)^{-1}=(\widetilde
A^t_{P(-T^{-1})})^{-1}$ (cf. \eqref{eq3.28})
is continuous as a map to
$\cB(L^2(M,F\oplus E),L^2_1(M,E\oplus F)$. Graph continuity 
of $(A,T)\mapsto \widetilde A_{P(T)}\ii$ means
that $(A,T)\mapsto A_{P(T)}^{-1}\in\cB(L^2,L^2)$ and
$(A,T)\mapsto \bigl((A_{P(T)})^*\bigr)^{-1}\in\cB(L^2,L^2)$ are
continuous.

The result should not come as a surprise. We should take the
invertible double as a guideline. Under more restrictive assumptions on $A$ one can
construct from $A$ an invertible operator $\widetilde A$ on the double
$\widetilde M$ of $M$. If $A$ varies continuously in the metric $d_0$
then the geometric invertible double is a continuously varying
family in $\cB(L^2_1(\widetilde M,\widetilde E\oplus\widetilde F),L^2(\widetilde
M,\widetilde F\oplus\widetilde E))$ and hence\footnote{Note that for any pair of Banach spaces $X,Y$
the inversion map $\cB(X,Y)\longrightarrow \cB(Y,X), S\mapsto S\ii$ is continuous.} 
its inverse varies continuously
in $\cB(L^2(\widetilde M,\widetilde F\oplus\widetilde E),L^2_1(\widetilde
M,\widetilde E\oplus\widetilde F))$. In the case of a geometric invertible double
the nice thing is that the domains of all first order elliptic
differential operators coincide with $L^2_1$.

The difficulty we are facing here is that the domains
$L^2_{1,T}$ vary with $T$. So our first task will be
to transform, at least locally, the whole situation to families of
operators with constant domain.

\begin{dfn} \label{def4.11} Let $e:L^2_s(\gS,F_\gS)\to L^2_{s+\half}(M,F), s>0,$ be a linear
  right--inverse to the trace map (cf. e.g. Remark \ref{rem3.3}.1 and \cite[Definition 11.7e]{BooWoj:EBP}).
%\cite[Prop. 2.18]{BruLes:BVP}).
For $T,T'\in\Diff^0(\Sigma;E_\Sigma,F_\Sigma)$ we put
\begin{equation}\label{eq4.28}
    \Phi_{T,T'}\binom{f_+}{f_-}:=\binom{f_+}{f_-+e(T'-T)\varrho
    f_+}.
\end{equation}
\end{dfn}

We record some properties of $\Phi_{T,T'}$ which are
straightforward to verify.

\begin{lemma}\label{lemma4.12}
  $\Phi_{T,T'}\in\cB(L^2_1(M,E\oplus F))$ and
we have $\Phi_{T',T''}\circ
\Phi_{T,T'}=\Phi_{T,T''}$. In particular
$\Phi_{T,T'}$ is invertible with inverse
$\Phi_{T',T}$.

Furthermore, $\Phi_{T,T'}$ maps
$L^2_{1,T}(M,E\oplus F)$ bijectively onto
$L^2_{1,T'}(M,E\oplus F)$.

Finally, we have
\begin{equation}\label{eq4.29}
   \|\Phi_{T,T'}-\Id\|_{1,1}\le C(e)
   \|T-T'\|_{\half,\half}.
\end{equation}
\end{lemma}

After these preparations, the proof of Theorem \plref{thm4.10} is
straightforward:

\begin{proof}[Proof of Theorem \plref{thm4.10}]
The map $(A,T)\mapsto \widetilde A_{P(T)}^{-1}$ can be
factorized as follows: fix a $T_0\in \Diff^0(\Sigma;E_\Sigma,F_\Sigma)$. Then the
following maps are continuous:

\begin{equation}
   \begin{split}
    \cE &\longrightarrow  G\cB(L^2_{1,T_0},L^2)\stackrel{\text{inversion}}{\longrightarrow}
    \cB(L^2,L^2_{1,T_0})\\
    (A,T)&\longmapsto \widetilde A_{T}\circ \Phi_{T_0,T}.
    \end{split}
\end{equation}
Here, $G\cB(L^2_{1,T_0},L^2)$ denotes the \emph{invertible} bounded linear maps between
$L^2_{1,T_0}$ and $L^2$.

The continuity is seen as follows:
\begin{equation}\begin{split}
     \|&\widetilde A_{P(T)}\circ\Phi_{T_0,T}-\widetilde
     A'_{P(T')}\circ\Phi_{T_0,T'}\|_{L^2_{1,T_0}\to L^2}\\
   &\le \|\widetilde A_{P(T)}\circ\Phi_{T_0,T}-\widetilde
     A'_{P(T')}\circ\Phi_{T_0,T'}\|_{1,0}\\
   &\le \|\widetilde
     A\|_{1,0}\|\Phi_{T_0,T}-\Phi_{T_0,T'}\|_{1,1}\\
   &\quad +\|\widetilde A-\widetilde A'\|_{1,0}\|\Phi_{T_0,T'}\|_{1,1}.
\end{split}
\end{equation}

Furthermore, the map
\begin{equation}
     \begin{split}
       \Diff^0(\Sigma;E,F)\times \cB(L^2,L^2_{1,T_0})&\longrightarrow \cB(L^2,L^2_1)\\
        (T,T)&\longmapsto \Phi_{T_0,T}\circ T
     \end{split}
\end{equation}
is continuous in view of Lemma \plref{lemma4.12}. Hence
\begin{equation}
(A,T)\mapsto \widetilde
A_{P(T)}^{-1}=\Phi_{T_0,T}\circ (\widetilde A\circ
\Phi_{T_0,T}\rrr L^2_{1,T_0})^{-1}
\end{equation}
is continuous as claimed.\end{proof}

\subsection{Continuous dependence of $S(A,T)$}

Next we give a simple criterion for the continuous dependence of $S(A,T)$
on the input data. 

\begin{prop} \label{prop4.7} 
The map $(A,T)\mapsto S(A,T)\in \cB(L^2(\Sigma,E_\Sigma),L^2_{\comp}(M,E\oplus F))$ is continuous with respect to the norm $N_1$ introduced in Definition 
\plref{d:strength-SALambda}.
\end{prop}
\begin{proof} This follows immediately from \eqref{eq4.18} and \eqref{eq4.19}.
\end{proof}

\subsection{Continuous dependence of $P_\pm$ on input data}\label{ss:dependence-Pplus}
Finally we study the dependence of $P_\pm$ on $B$, where $B\in\Diff^1(\Sigma;E_\Sigma)$
satisfies the usual assumptions (cf. Subsection \plref{ss:some-estimates}).

The definition of $P_+=P_+(B)$ for $B\in \Diff^1(\Sigma;E_\Sigma)$ requires a choice of a
\emph{spectral cut}, i.e., a $c>0$ such that
\begin{equation}\label{eq:spectral-cut}
     \spec B\cap \bigsetdef{z\in\C}{|z|=c}=\emptyset.
\end{equation}
Obviously, for a choice of $c$ the map
$B\mapsto P_+(B)$ has a discontinuity at $B's$ where \eqref{eq:spectral-cut} is violated.

Apart from that $P_+(B)$ should depend continuously on $B$ with respect to the 
norm $V\mapsto \|V\|_{1,0}+\|V^t\|_{1,0}$.
Unfortunately, we cannot prove or disprove this conjecture. Instead we mention two simple
continuity criteria. 
The first deals with lower order perturbations of a fixed operator and
the second deals with selfadjoint operators where the Spectral Theorem yields continuity
in a rather simple fashion.

\begin{prop}\label{p:lower-order} Let $B\in\Diff^1(\Sigma;E_\gS)$ be a first order elliptic differential
operator such that $B-\lambda$ is parameter dependent elliptic for $\lambda$ in a conic neighborhood
of $i\R$. Furthermore, let $V\in\CL^\ga(\Sigma;E_\gS)$, $\ga<1$, with $\|V\|_{\ga,0}$ sufficiently small.
Then $B+V-\lambda$ is also parameter dependent elliptic in a conic neighborhood of $i\R$ and for $c>0$
such that
$\spec (B+tV) \cap \bigsetdef{z\in\C}{|z|=c}=\emptyset,\quad 0\le t\le 1,$
we have the estimate
\[
\|P_+(B+V)-P_+(B)\|_{0,0}\le C(B) \|V\|_{\ga,0}.
\]
\end{prop}
\begin{proof} By a Neumann series argument it is clear that such a $c$ exists. Let $\Gamma_+$
(see Figure \plref{f:gamma_pm}) be the usual contour. Analogously to Lemma \plref{lemma4.2}
one shows the estimate
\begin{equation}
       \|(\gl-(B+V))\ii-(\gl-B)\ii\|_{0}\le C(B) \|V\|_{\ga,0} |\gl|^{-2+\ga}, \quad \gl\in\Gamma_+,
\end{equation}
from which the claim, thanks to $-2+\ga<-1$, follows by invoking the contour integral.
\end{proof}

We now turn to (formally) selfadjoint $B$. We first give a slight improvement of
\cite[Prop. 2.2]{Les:USF} on the continuity of the \emph{Riesz}--map (cf. loc. cit. Equ. (2.2)).
\index{Riesz--map}

\begin{prop}\label{p:Riesz-transform} Let $\Ell^{1,\sa}(\Sigma;E_\Sigma)\subset\Diff^1(\Sigma;E_\Sigma)$ denote the space
of \emph{selfadjoint} first order elliptic differential operators. Then the \emph{Riesz}--map
\[B\mapsto B(\Id+B^2)^{-1/2}\]
is continuous $(\Ell^{1,\sa},\|\cdot\|_{1,0})\to \cB(L^2_s(\Sigma,E_\Sigma))$
for all $s\in [-\half,\half]$.
\end{prop}
\begin{proof} For $s=0$ this was proved in \cite[Prop. 2.2]{Les:USF}. The proof in loc. cit., however
shows the claimed stronger statement (cf. \cite[Equ. (2.19)]{Les:USF}). For the convenience of the reader
let us present the argument. 

It suffices to prove the continuity of $F$ for $s=\half$. Since $F(B)$ is selfadjoint
$\|F(B)-F(\widetilde B)\|_{\half,\half}=\|F(B)-F(\widetilde B)\|_{-\half,-\half}$ and hence by complex interpolation
(cf. \cite[Appendix]{Les:USF}) $\|F(B)-F(\widetilde B)\|_{s,s}\le \|F(B)-F(\widetilde B)\|_{\half,\half}$
for $|s|\le \half$.

Fix a $B\in\Ell^{1,\sa}(\Sigma,E_\Sigma)$ and we have to prove the continuity of 
$F$ at $B$. Let $0<q<\frac 12$ and consider $\widetilde B\in\Ell^{1,\sa}(\Sigma,E_\Sigma)$ with
\begin{equation}\label{eq:USF-1}
     \|(B-\widetilde B)(B\pm i)^{-1}\|_{0,0}+ \|(B\pm i)^{-1}(B-\widetilde B)\|_{0,0}\le q.
\end{equation}
Note that by ellipticity the graph norm of $B$ is equivalent to the Sobolev norm $\|\cdot \|_1$,
hence the left hand side of \eqref{eq:USF-1} induces a metric on $\Ell^{1,\sa}(\Sigma,E_\Sigma)$
which is equivalent to the metric induced by $\|\cdot\|_{1,0}$.
The Neumann series then immediately implies
\begin{equation}
    \|(\widetilde B+i)^{-1}(B+i)\|_{0,0}\le \frac{1}{1-q}.
\end{equation}
Thus, for $f\in L^2(\Sigma,E_\Sigma)$ we have
\begin{equation}
     \|(\widetilde B+i)^{-1}f\|_0\le \frac{1}{1-q}\|(B+i)^{-1}f\|_0
\end{equation}
and
\begin{equation}
     \|(B+i)^{-1}f\|_0\le \|(B+i)^{-1}(\widetilde B+i)\|_{0,0}\,\|(\widetilde
     B+i)^{-1}f\|_0 \le (1+q)\|(\widetilde B+i)^{-1}f\|_0.
\end{equation}
This implies the operator inequalities
\begin{equation}
    \frac{1}{(1+q)^2}|B+i|^{-2}\le |\widetilde B+i|^{-2}\le
    \frac{1}{(1-q)^2}|B+i|^{-2}.
\end{equation}
Since the square root is an operator--monotonic increasing function 
(\cite[Prop. 4.2.8]{KadRin:FTOI}) we may take the square root of these
inequalities and after subtracting $|B+i|^{-1}$ we arrive at
\begin{equation}
    -\frac{q}{1+q}|B+i|^{-1}\le |\widetilde B+i|^{-1}-|B+i|^{-1}\le
    \frac{q}{1-q}|B+i|^{-1}.
\end{equation}
This gives
\begin{equation}
   \|\,|B+i|^{1/2}|\widetilde B+i|^{-1}|B+i|^{1/2}-\Id\|_{0,0}\le \frac{q}{1-q}.
\end{equation}
After these preparations we find
\begin{equation}\begin{split}\label{eq:USF-8}
           &\|F(B)-F(\widetilde B)\|_{\half,\half}\\
         &=\bigl\|\,|i+B|^{1/2}(F(B)-F(\widetilde B))     |i+B|^{-1/2}\|_{0,0}\\
         &=\bigl\|\,|i+B|^{-1/2}(F(B)-F(\widetilde B)) |i+B|^{1/2}\|_{0,0}\\
         &\le \bigl\|\,|i+B|^{-1/2}(B-\widetilde B)|i+B|^{-1/2}\bigr\|_{0,0}\\
            &\quad  +\bigl \|\, |B+i|^{-1/2}\bigl(\widetilde B(|i+B|^{-1}-
               |i+\widetilde B|^{-1})\bigr)|i+B|^{1/2}\bigr\|_{0,0}\\
        &\le \||i+B|^{-1}(B-\widetilde B)\|_{0,0}\\
           &\quad +\|\,|i+B|^{-1/2}\widetilde B |i+B|^{-1/2}\|\,\|\Id-|i+B|^{1/2}|i+
                    \widetilde B|^{-1}|i+B|^{1/2}\|_{0,0}\\
        &\le q+\|\,|i+B|^{-1}\widetilde B\|_{0,0} \frac{q}{1-q}\\
        &\le q(1+\frac{1+q}{1-q}).
      \end{split}
\end{equation}
Here we have used that for a first order operator $V$ one has 
\begin{equation}
\||i+B|^{-\half}V|i+B|^{-\half}\|_{0,0}\le \||i+B|\ii V\|_{0,0}.
\end{equation}
This inequality also follows from complex interpolation (see \cite[Appendix]{Les:USF}).
This shows that if $\|B_n-B\|_{1,0}\to 0$ then $F(B_n)\to F(B)$ in $L^2_\half(\Sigma,E_\Sigma)$ and
we are done.
\end{proof}

\begin{prop}\label{p:dependence-Pplus-self-adjoint}
Let $\Ell_c^{1,\sa}(\Sigma;E_\Sigma)\subset \Ell_c^{1,\sa}(\Sigma;E_\Sigma)
\subset\Diff^1(\Sigma;E_\Sigma)$ denote the space of \emph{selfadjoint}
first order elliptic differential operators B with $\pm c\not\in\spec B$.
Then for $|s|\le \half$ the map
\begin{equation}
\begin{split}      
\Bigl(\Ell_c^1,\|\cdot\|_{1,0}\Bigr)&\longrightarrow \cB(L^2_s(\Sigma,E_\Sigma))\\
              B&\mapsto 1_{[c,\infty)}(B)
            \end{split}
\end{equation}
is continuous.
\end{prop}
\begin{proof}  We first note that for $B\in \Ell^{1,\sa}_c$ we have
by the Spectral Theorem 
\[  1_{[c,\infty)}(B)=1_{[F(c),\infty)}\bigl(F(B)\bigr).\]
Since, independently of $s\in [-\half,\half]$, we have
\begin{equation}
      \spec\bigl(F(B))\bigr)\subset [-1,1],
\end{equation}
$1_{[c,\infty)}(B)$ is given by the contour integral
\begin{equation}
    \frac{1}{2\pi i}\oint_{|z-(F(\gl)+2)|=2} \Bigl(z-F(B)\Bigr)^{-1}dz.
\end{equation}
In view of the previous Proposition \plref{p:Riesz-transform} this 
proves the claim. cf. also \cite[Lemma 3.3]{Les:USF}.
\end{proof}

\subsection{Continuity of families of \emph{well--posed} selfadjoint
Fredholm extensions}

In Theorems 3.8 and 3.9 in \cite{BooLesPhi:UFOSF} the continuous dependence of the invertible double, the
Calder{\'o}n and Poisson operators and the graph continuity of realizations of well--posed boundary value
problems are discussed in a special case. More precisely it was assumed that
$J^2=-\Id$ and that the tangential operator had a selfadjoint leading symbol.

Unfortunately the proof of Theorem 3.9 in \cite{BooLesPhi:UFOSF} was incomplete. 
Now we can present correct statements with complete proofs. Our result is more general than loc. cit.
in the sense that we do neither have to assume that $J^2=-\Id$, nor do we have to assume
a selfadjoint leading symbol of the tangential operator in all cases. On the negative side we must
admit that the topology we have to impose on the space of differential operators is stronger
than we hoped at the time of writing of \cite{BooLesPhi:UFOSF}. The correct replacement
for Theorem 3.9a in loc. cit. is Theorem \plref{thm4.14}a and the correct replacement
for Theorem 3.9b in loc. cit. are Theorem \plref{thm4.14}b and Corollary \plref{thm4.14-sa}.

Theorem \plref{thm4.10} above generalizes Theorem 3.8 in loc. cit. and \cite[Proposition B.1]{Nic:GSG}

Next we deal with families of realizations of well--posed boundary conditions (cf. Theorem 3.9d in loc. cit.).

\begin{theorem}\label{thm:continuous-well-posed}
Consider the space of pairs $(A,P)$ where $A\in \Diff^1(M;E)$ is elliptic and
formally selfadjoint and $P\in\CL^0(\Sigma;E_\Sigma)$ is an orthogonal projection which is well--posed
with respect to $A$ and such that $A_P$ is selfadjoint. Equip this space with the metric $d_0$, i.e.,
\begin{equation}
   d_0((A,P),(A',P'))=N_0(A-A',P-P').
\end{equation}
Then the map $(A,P)\mapsto (A_P+i)\ii\in \cB(L^2(M,E),L^2_1(M,E))$ is continuous with respect to
the $d_0$ metric on the space of such pairs $(A,P)$. In particular $(A,P)\mapsto (A_P+i)\ii$ is continuous or, equivalently, $(A,P)\mapsto A_P$ is graph continuous.
\end{theorem}

\begin{proof} The proof is basically the same as the proof of Theorem \plref{thm4.10} once
the analogue of the maps $\Phi_{T,T'}$ is established. Note that if $P,Q\in\CL^0(\Sigma;E_\gS)$
are orthogonal projections with $\|P-Q\|_{\half,\half}<1$ they form an \emph{invertible pair}, i.e.,
$P:\im Q\longrightarrow \im P$ is invertible. For such $P,Q$ we put analogously to Definition
\plref{def4.11}
\begin{equation}
   \Psi_{P,Q}(f):= f-e(P-Q)\varrho f.
\end{equation}
Then as in Lemma \plref{lemma4.12} we have
\begin{equation}
   \|\Psi_{P,Q}-\Id\|_{1,1}\le C(e) \|P-Q\|_{\half,\half},
\end{equation}
and thus $\Psi_{P,Q}$ is invertible for $\|P-Q\|_{\half,\half}$ small enough. Furthermore,
$\Psi_{P,Q}$ maps $\cD(A_Q)=\bigsetdef{f\in L^2_1(M,E)}{Q\varrho f=0}$ bijectively
onto  $\cD(A_P)=\bigsetdef{f\in L^2_1(M,E)}{Q\varrho f=0}$. 

Now one mimics the proof of Theorem \plref{thm4.10} with $\Psi_{P,Q}$ instead of $\Phi_{T,T'}$
and $A_P+i$ instead of $\widetilde A_{P(T)}$.
\end{proof}

Finally we state a more precise version of \cite{BooLesPhi:UFOSF}, Theorem 3.9c. Note that
the following version applies to a much wider class of operators than loc. cit.

\begin{theorem}\label{thm:continuous-A-Calderon} 
Let $\Ell^{1,\sa}_{\UCP}(M;E)\subset \Diff^1(M;E)$ 
denote the space of formally selfadjoint elliptic differential operators
acting on sections of the Hermitian vector bundle $E$ which satisfy $\UCP$
and whose tangential operator $B_0$  has a selfadjoint leading symbol.
We equip this space with the strong metric induced by the embedding
$\Ell^{1,\sa}_{\UCP}(M;E)\to \cE_{\UCP}(M;E), A\mapsto (A,(J_0^t(A))\ii)$

Then the map
\[ \Ell^{1,\sa}(M;E)\longrightarrow \cB(L^2_1(M,E),L^2(M,E)),\quad A\longmapsto A_{C_+}\]
sending $A$ to the selfadjoint well--posed realization associated to the
Calder{\'o}n projection is continuous. Here $C_+$ denotes the version of the Calder{\'o}n
projection constructed from $(J_0^t(A))\ii$, cf. Proposition \plref{prop3.19}.
\end{theorem}
\begin{proof}
Note that $A_{C_+}$ is selfadjoint by Theorem \plref{t:cobord} (II),
Prop. \plref{prop3.19}. $A_{C_+}$ is indeed a selfadjoint realization of
a well--posed boundary value problem.

By Corollary \plref{thm4.14-sa}
$A\mapsto (A,C_+)$ is now a continuous map from $\Ell^{1,\sa}_{\UCP}(M,E)$ 
to the space of pairs described in Theorem \plref{thm:continuous-well-posed}
and hence Theorem \plref{thm:continuous-well-posed} yields the claim.
\end{proof}

\newpage
\appendix
%\setcounter{secnumdepth}{1}
%\setcounter{theorem}{0}
%\setcounter{section}{0}
%\setcounter{equation}{0}
%% \setcounter{secnumdepth}{2}
% {p. APP.\arabic{page}}
% \markboth{}{}
\addtocontents{toc}{\medskip \noi}
\section{Smooth symmetric elliptic continuations with constant
coefficients in normal direction}
\label{s:const-coef}

In this appendix, we restrict ourselves to formally selfadjoint operators.
To begin with, we write $D=J(\dd x + B)$ like in \eqref{eq2.9} with the
relations $J^*=-J$ and $JB=J'-B^tJ$ of \eqref{eq2.12} without loss of generality.

Sometimes one is interested in operators satisfying additional relations. We shall
consider the following cases:

\begin{enumerate}\renewcommand{\labelenumi}{(\Roman{enumi})}
\item  $D$ arbitrary symmetric elliptic, e.g. no additional relations.
\item $J^2(x)=-\Id$. We will see below that then, 
after a suitable coordinate transformation, we can even obtain that $J_x=J$ is constant.
This is the Dirac operator case.
\item $B_0-B_0^t$ of order $0$. In view of \eqref{eq2.12} this implies that $J_0B_0+B_0J_0$
is of order $0$, too.
\item $B-B^t$ of order $0$. Analogously then $JB+BJ$ is of order $0$, too.
\item $D=J\Bigl(\frac{d}{dx}+B\Bigr)+\frac 12 J'+C$ with $J^*=-J$, $B=B^t, JB+BJ=0, J'=\frac{dJ}{dx},$ and $C$ of order $0$.
Then automatically $C=C^*$.
\end{enumerate}

One can think of even more cases. But the preceding cases suffice for the moment.
From now on we shall write
\begin{equation}
D=J\Bigl(\frac{d}{dx}+B\Bigr)+\frac 12 J'+C
\label{eq-1.8}
\end{equation}
in all five cases
although the notation is redundant in cases I-IV. Recall, that here $C$ is of order
$0$ and $B$ is of order $1$.

\begin{prop} Consider the case \textup{(II)}. Then there is a smooth unitary gauge
transformation $U\in\cinf{[0,\eps),\ginf{\gS;\cU(E_{\gS})}}$ such
that
\[
J_x=U(x)J_0U(x)^*.
\]
With the unitary transformation
\[
\begin{matrix}
\Phi_1&:&L^2([0,\eps)\times\gS,E_{\gS})&\too&
L^2([0,\eps)\times\gS,E_{\gS}),\\
\ && f &\longmapsto& Uf
\end{matrix}
\]
we find
\begin{equation}
  \Phi_1^{-1} D\Phi_1=J_0\Bigl(\frac{d}{dx}+U^*BU\Bigr) +\widetilde C.
\end{equation}
\end{prop}
Here and in the following we use the abbreviation $E_{\gS}:=E\rrr
\gS$ introduced in Subsection \ref{ss:product} and, by slight abuse
of notation, $E_\gS:=E\rrr [0,\eps)\times \gS$ as well. Note that
selfadjointness of $B$ or $C$ and the relation $JB+BJ=0$ (of lower
order) are preserved under $\Phi_1$.

\begin{proof} We give a brief sketch; it is basically a standard fact often used
in K--theory \cite[Prop. 4.3.3]{Bla:KTO}, see also \cite[Section 3]{BruLes:BVP}.

We only show a bit less, namely that the claim is true after making $\eps$
a bit smaller; but this is not really a loss of generality.

After possibly making $\eps$ smaller we may assume that
\begin{equation}
\|\Id+J_xJ_0\|=\|\Id-J_x^{-1}J_0\|<2, \quad \text{for } 0\le x\le \eps.
\end{equation}
Then
the operator $Z_x:=\Id-\frac 12(J_xJ_0+\Id)=\frac 12 (\Id-J_xJ_0)$
is invertible. Moreover, since $J_x^2=-\Id$ we find
$J_xZ_x=\frac 12(J_x+J_0)=Z_xJ_0$ and thus $J_x=Z_xJ_0Z_x^{-1}$.

One now checks by direct calculation that $Z_x$ is normal and that
$Z_xZ_x^*$ commutes with $J_x$ and $J_0$
\cite[(3.7)]{BruLes:BVP}. Hence we may put
$U(x):=Z_x^{-1}\sqrt{Z_xZ_x^*}$ to reach the conclusion.

The remaining assertions are now clear.
\end{proof}

\begin{remark}\label{r:symplectic}
\enum{1} One may ask under what conditions we can obtain unitary $J_0$,
i.e., $J_0^2 =-\Id$? It is not clear whether this question has a
definite answer. E.g., it seems impossible to find a coordinate
transformation preserving the symmetry of the fixed given differential
operator $D$ and providing a unitary leading symbol $\widetilde J$ in the new
normal direction, unless all eigenvalues of the original $J_0$ are
of the form $\pm i$.

\enum{2} For the symplectic Hilbert space
$\Bigl(L^2(\gS,E_{\gS}),\lla\cdot,\cdot\rra,\go(\cdot,\cdot)\Bigr)$,
however, there is a simple answer. Here we set
\[
\lla f,g\rra := \int_{\Sigma}\scalar{f(p)}{g(p)}_{E_p}\dvol \quad\tand\quad \go(f,g):=
         \scalar{J_0f}{g}.
\]
As always in symplectic Hilbert spaces (see, e.g., \cite[Lemma
1.5]{BooZhu:WSF}), we can preserve the symplectic form $\go$ while
deforming the inner product $\lla\cdot,\cdot\rra$ of
$L^2(\gS,E_{\gS})$ over $\gS$ smoothly into
\[
\lla\cdot,\cdot\rra^{\widetilde\ } :=  \lla Sf,g\rra \quad\text{with
selfadjoint $S:=\sqrt{J_0^*J_0}$}
\]
in such a way that $\go(f,g)=\scalar{\widetilde Jf}{g}^{\widetilde\ }$ with $\widetilde
J^*=-\widetilde J$ and $\widetilde J^2 =-\Id$.
\end{remark}

The notation \eqref{eq-1.8} has another advantage in case %\JAImmobilieninAlfter
(IV). Namely, replacing $B$ by $\frac 12(B+B^t)$ and
$C$ by $C+\frac 12 J(B-B^t)$ if necessary we see that we may assume
that $B=B^t$. Note that this does not work so easily in case (III).

Summing up, the cases (I)-(V) may now be described as follows (when $D$ is
written as in \eqref{eq-1.8}):
\begin{enumerate}\renewcommand{\labelenumi}{(\Roman{enumi})}
\item  $D$ arbitrary symmetric elliptic, e.g. no additional relations.
\item $J^2=-\Id$, $J_x=J_0$ constant. Again, this is called the Dirac operator
case.
\item $B_0-B_0^t$ and $J_0B_0+B_0J_0$ of order $0$.
\item $B=B^t$, $JB+BJ$ of order $0$.
\item $B=B^t, JB+BJ=0, C=C^*$.
\end{enumerate}

We consider $D$ as before. The goal of this appendix is to prove the
following Theorem.

\begin{theorem} Let $D$ be as in \eqref{eq-1.8} with families
$J_x,B_x,C_x$ smoothly depending on $x$.
Then there is a $\delta>0$ and a symmetric elliptic first order
differential operator $\widetilde D$ (i.e., with smooth coefficients) on
$\ginf{[-\delta,\eps)\times \gS;E_{\gS}}$ with the following properties:
\begin{enumerate}
\item $\widetilde D\rrr[0,\eps)\times\gS=D$, i.e., $\widetilde D$ extends $D$.
\item
$\widetilde D\rrr[-\delta,-2/3\delta)\times \gS)=J_0\Bigl(\frac{d}{dx}+B_0\Bigr)+C_0.$
In particular, $\widetilde D$ has constant coefficients near the new boundary
$\{-\delta\}\times\gS$, and the constant coefficients are given just
by smoothly \textqm{rewinding} to the coefficients of $D$ at $0$.
\end{enumerate}
\end{theorem}
Note that in the formula for $\widetilde D$ near the boundary $\frac 12 J'$ is left
out deliberately to make the constant coefficient operator symmetric.

Note also that due to the concrete formula for $\widetilde D$ near the boundary
the additional relations in the cases (II)--(V) still hold for the extended
operator. It is not claimed (and it is open in some cases) that the relations
in (IV),(V) are preserved on the whole interval $[-\delta,0]$.

\begin{proof}
1. Let
\begin{equation}
     D_0:= J_0\Bigl(\frac{d}{dx}+B_0\Bigr)+C_0.\label{eq-2.1}
\end{equation}
Since $D_0$ has constant coefficients we may think of $D_0$ as
acting on $\ginfz{\R\times \gS;E_{\gS}}$. Note that since $D$ is formally
selfadjoint, so is $D_0$. To see that we recall that $D=D^t$ is 
(thanks to our singling--out of  $\frac 12 J'$ in Equation \eqref{eq-1.8})
equivalent to the relations
\begin{equation}
   J^*=-J, \quad -B^tJ+C^*=JB+C.\label{eq-2.2}
\end{equation}
Then
\begin{equation}
   D_0^t= J_0\frac {d}{dx} -B_0^tJ_0+C_0^*=J_0\frac {d}{dx} +J_0B_0+C_0=D_0.
\end{equation}

2. Now we apply the definition of smoothness of maps defined on a
manifold with boundary to conclude that $J,B,C$ have smooth
extensions to the whole negative half-line. More precisely there
exist
\begin{multline*}
\widetilde J\in\ginf{(-\infty,\eps)\times\gS;E_{\gS}}, \quad \widetilde
C\in\ginf{(-\infty,\eps)\times \gS;E_{\gS}} \tand\\
\widetilde B\in\cinf{(-\infty,\eps),\Diff^1(\gS;E_{\gS})}
\end{multline*}
such that
\[
\widetilde J\rrr[0,\eps)\times \gS=J, \ \widetilde C\rrr[0,\eps)\times \gS=C, \tand
\widetilde B\rrr[0,\eps)\times \gS=B.
\]
Replacing $\widetilde J$ by $\frac 12(\widetilde J-\widetilde
J^*)$ if necessary, we can, additionally, obtain that $\widetilde J^*=-\widetilde J$.

The extensions of $J$ and $C$ are immediate. However, for $B$ one might
feel a bit uneasy because of the target space
$\Diff^1(\gS;E_{\gS})$. Well, first extend the leading symbol of
$B$, which works like for $J$ and $C$. Then choose a right inverse
to the symbol map on $\gS$ to obtain an operator map $\widetilde B_1$. On
$[0,\eps)\times \gS$, $\widetilde B_1$ coincides with $B$ up to order $0$.
The difference $\widetilde B_1-B\rrr[0,\eps)\times \gS$ is again smooth.
Extend it and subtract it to make up for the $0$ order defect. This yields the wanted $\widetilde B$.

So, we can form the differential operator
\begin{equation}
    \widetilde D_1:=\widetilde J\Bigl(\frac{d}{dx}+\widetilde B\Bigr)+\frac 12 \widetilde
    J'+\widetilde C
\end{equation}
which is now defined on $(-\infty,\eps)\times \gS$, has smooth
coefficients and $\widetilde D_1\rrr[0,\eps)\times\gS=D$.

So far $\widetilde D_1$ is not necessarily formally selfadjoint. Put
\begin{equation}
    \widetilde D_2:=\frac 12\bigl(\widetilde D_1+\widetilde D_1^t\bigr)
               =:\widetilde J\Bigl(\frac{d}{dx}+\widetilde B_2\Bigr)+\frac12 \widetilde J'+
                      \widetilde C_2.
\end{equation}

Next, consider a cut--off function $\varphi\in\cinf{\R}$ with
\begin{equation}
      \varphi(x)=\begin{cases}  1,& x\le -\frac 23 \delta,\\
                                0, & x\ge -\frac 13 \delta.
                 \end{cases}
\end{equation}
Then we consider the operator
\begin{equation}
    \widetilde D:=\varphi D_0+(1-\varphi) \widetilde D_2+\frac 12 \varphi'(J_0-\widetilde J).
\end{equation}
The last summand was left out in \cite[(3.14)]{BooLesPhi:UFOSF}; the additional term
is, however, necessary to make $\widetilde D$ formally selfadjoint.

$\widetilde D$ has the following properties:
\begin{enumerate}
  \item $\widetilde D^t=\widetilde D$. That follows immediately from the formal selfadjointness
of $D_0$, $\widetilde D_2$ and the relations $[D_0,\varphi]=J_0\varphi',\
[\widetilde D_2,\varphi]=J\varphi'$.
  \item $\widetilde D\rrr[0,\eps)\times\gS=D$.
  \item $\widetilde D\rrr[-\delta,-2/3\delta)\times\gS=J_0\Bigl(\frac{d}{dx}+B_0\Bigr)+C_0$.
\end{enumerate}
It remains to discuss the ellipticity of $\widetilde D$. So let
\[
x\in
[-\delta,0],\ p\in\gS,\ \xi\in T^*_p(\gS) \tand \lambda dx+\xi\in S_{x,p}^*(\R\times\gS),
\]
$S^*$ denoting the cosphere bundle. Then $|\gl|^2+|\xi|^2=1$ and we
find for the leading symbol of $\widetilde D$
\begin{equation}
    \begin{split}
       \sigma_{\widetilde D}&(x,p)[\gl dx+\xi]=
     \Bigl(\varphi(x) \sigma_{D_0}(x,p)
         +(1-\varphi(x))\sigma_{\widetilde D_2}(x,p)\Bigr)[\gl dx+\xi]\\
     &= \varphi(x)J_0\bigl(i\gl +\sigma_{B(0)}(\xi)\bigr)+(1-\varphi(x))J_x\bigl(i\gl+
             \sigma_{B(x)}(\xi)\bigr)\\
    &= J_0\bigl(i\gl +\sigma_{B(0)}(\xi)\bigr)+(1-\varphi(x))\bigl\{
      J_x(i\gl+\sigma_{B(x)}(\xi))-J_0(i\gl+\sigma_{B(0)}(\xi))\bigr\}.
    \end{split}
\end{equation}
Hence, by the compactness of $\gS$ and by the continuity of $J_x$
and $\sigma_{B(x)}$ we may choose $\delta$ so small that $\widetilde D$ is
elliptic.

The theorem is proved.
\end{proof}

\begin{remark} We briefly discuss for the various cases (I)--(V)
whether the construction of $\widetilde D$ can be modified such that
the relations continue to hold.
\begin{enumerate}\renewcommand{\labelenumi}{(\Roman{enumi})}
\item Since there are no relations, there is nothing to worry about.
\item Since $J$ is constant, we may extend it constantly.
\item At the new boundary $\{-\delta\}\times \gS$ we have by construction
$B(-\delta)=B(0)$ hence (III) also works fine.
\item[(IV), (V)] By construction of $D_0$ it is clear that $\widetilde D$
also satisfies (IV) or (V) in the collar $[-\delta,-2/3\delta]\times
\gS$ of the new boundary. However, we do not know so far whether
$\widetilde D$ can be constructed in such a way that the relations hold on
the whole interval $[-\delta,0]$.

We leave it to the reader to prove that the latter is indeed possible in the case (II)+(IV) or (II)+(V).
\end{enumerate}
\end{remark}

In Subsection \ref{ss:product} we first applied a unitary transformation to
our operator. It remains to clarify what happens if we transform the whole
construction back to the original situation. %% :\marginpar{A}
The result reads as follows:

\begin{theorem} Let $(M,g_1)$ be a compact Riemannian manifold with boundary,
$(E,h_1)$ a Hermitian vector bundle over $M$ and $A:\ginfz{M;E}\to \ginfz{M;E}$
a first order elliptic differential operator which is formally selfadjoint
in the Hilbert space $L^2(M,E;g_1,h_1)$.

Choose metrics $g$ on $M$, $h$ on $E$ which are product near the
boundary $\gS=\partial M$ and such that $g\rrr\partial
M=g_1\rrr\partial M, h\rrr\partial M=h_1\rrr\partial M$. Let $\Phi$
be the isometry of \eqref{eq2.8} and assume that we have
\begin{equation}
    \Phi A\Phi^{-1}=J\Bigl(\frac{d}{dx}+B\Bigr)+\frac 12 J'+C.
\end{equation}
For $\delta>0$ form $M_\delta:=\bigl([-\delta,0]\times\gS\bigr)\cup_\gS M$.

Then for $\delta$ sufficiently small there are a Hermitian vector bundle
$(E_\delta,h_{1,\delta})$
over $M_\delta$, a Riemannian metric
$g_{1,\delta}$ on $M_\delta$ and a first order symmetric elliptic differential
operator $A_\delta$ on $M_\delta$ such that:
\begin{enumerate}
 \item $E_\delta$, $g_{1,\delta}$, $h_{1,\delta}$, $A_\delta$ are extensions
  of $E$, $g_1, h_1, A$ respectively.
 \item $g_{1,\delta}$ and $h_{1,\delta}$
are product metrics near $\partial M_\delta$. More precisely, we have
on $[-\delta,-2/3\delta]\times \gS$
\begin{equation}\begin{split}
       g_{1,\delta}&=dx^2\oplus g(0), \\
       h_{1,\delta}(x)&= h(0).
                \end{split}\label{eq-2.10}
\end{equation}
\item With the natural extension of $\Phi$ to $M_\delta$ we have
on $[-\delta,-2/3\delta]\times\gS$:
\begin{equation}
    \Phi A\Phi^{-1}=J_0\Bigl(\frac{d}{dx}+B_0\Bigr)+C_0.
\end{equation}
\end{enumerate}
\end{theorem}
\begin{proof} We apply the previous Theorem to  $\Phi A\Phi^{-1}$
in $L^2(M,E;g,h)$. Since the spaces of metrics on $M$ and on $E$ are
positive cones, we may extend the metrics $g_1,h_1$ to smooth
metrics $\widetilde g_\delta,h_{1,\delta}$ on $M_\delta$ in such a way that
on the collar $[-\delta,-2/3\delta]\times \gS$ they are of the form
\eqref{eq-2.10}. Now consider the unitary transformation
$\Psi:L^2(M_\delta,E;g_{1,\delta},h_{1,\delta})\to
L^2(M_\delta,E;g,h)$ as in Lemma \plref{l2.1} ($g,h$ extend
trivially to $M_\delta$ since they are already product metrics). By
the construction explained before Lemma \plref{l2.1} we see that
$\Psi\rrr[-\delta,-2/3\delta]\times\gS=\Id$ and therefore we reach
the conclusion from the previous theorem.
\end{proof}

\begin{remark} 
\enum{1} The preceding theorem shows that it is always possible to
extend a given symmetric elliptic differential operator (of first
order) $A$ to a symmetric elliptic differential operator $\widetilde A$ by
attaching a collar with finite cylindrical end with new boundary
$\gS '$ such that the Riemannian and Hermitian structures become
product close to $\gS '$ and the operators $J_x,\ B_x \tand C_x$
become independent of the normal variable $x$ close to $\gS'$\,;
i.e., we can always bring the operator in product form near a new
boundary by suitable prolongation under preservation of the symmetry
property.

\enum{2} The previous discussion shows that $J_0^2=-\Id$ is perfectly
enough to obtain, by coordinate transformation, $J_x=J_0$ 
and hence $J_x^2=-\Id$ in a whole
neighborhood of the new boundary. See also Remark \ref{r:symplectic}
for conditions for $J_0^2=-\Id$.

\enum{3} The previous discussion can be made parameter dependent
with the right notion of parameter dependency, cf. our Section
\ref{s:parameter}.

\enum{4} The question remains whether weak inner UCP can be
preserved under the symmetric prolongation. The short answer is
\begin{itemize}
\item {\em yes} in the cases (III)-(V), i.e., when the leading symbol
of the tangential operator is symmetric;
\item {\em yes} in a very restricted sense, namely that the
UCP-defect dimension
\begin{equation}\label{e:ucp-defect}
d(x):= \dim\bigsetdef{u}{Au=0\tand u\rrr \Sigma(x)=0}
\end{equation}
is constant on the last part of the collar, i.e., for sufficiently
negative tangential coordinate $x$ when constant coefficients in
normal direction are obtained. Here $\Sigma(x):=\{x\}\times\partial
M$ denotes the parallel surface in the collar.
\end{itemize}
%For an extended discussion of various related questions see Appendix \ref{s:wiucp}.\marginpar{oops, reference to the %deleted Appendix B}
\end{remark}

\newcommand{\Toappear}{to appear in}
%\bibliography{mrabbrev,mlabbr2003-0,papers2005,books2005,local}
%\input{050504bolezh.bbl}
%\bibliographystyle{amsalpha}
%\input{main.bbl}
% 2008-03-28: bbl file included and modified by hand

\providecommand{\bysame}{\leavevmode\hbox to3em{\hrulefill}\thinspace}
\providecommand{\MR}{\relax\ifhmode\unskip\space\fi MR }
\providecommand{\MRhref}[2]{%
  \href{http://www.ams.org/mathscinet-getitem?mr=#1}{#2}
}
\providecommand{\href}[2]{#2}

\printindex

\newcommand{\glossaryentry}[2]{#1, #2\\}
\section*{Notation}
\markboth{Notation}{Notation}
%\addcontentsline{toc}{chapter}{Notation}\label{Symbolverz}

\begin{multicols}{3}
\raggedright
\noindent
\input BooLesZhu.gls
\end{multicols}

\end{document}